\documentclass[10pt]{amsart}
\pdfoutput=1
\usepackage{amscd}
\usepackage{amssymb}
\usepackage[all]{xy}
\usepackage{lmodern}
\usepackage[T1]{fontenc}
\usepackage{hyperref}
\hypersetup{colorlinks=false}
\date{4 July 2018} 
\title[Weak Stability]{Weak Proregularity, Weak Stability, and the 
Noncommutative MGM Equivalence}
\author{Rishi Vyas and Amnon Yekutieli}

\address{Vyas: Department of Mathematics, 
Ben Gurion University, Be'er Sheva 84105, Israel}
\email{vyas.rishi@gmail.com}

\address{Yekutieli: Department of Mathematics, Ben Gurion University,
Be'er Sheva 84105, Israel}
\email{amyekut@math.bgu.ac.il}

\thanks{{\em Mathematics Subject Classification} 2010. 
Primary: 18G10; Secondary: 13D45, 16S90, 16E35}
\keywords{Torsion classes, derived torsion, derived completion.} 
\thanks{Supported by the Israel Science Foundation grants no.\ 253/13 and 
170/12. The first author was partially supported by the Center for Advanced 
Studies in Mathematics at Ben-Gurion University, and by the 
Israel Council for Higher Education.}
 
\newtheorem{thm}[equation]{Theorem}
\newtheorem{cor}[equation]{Corollary}
\newtheorem{prop}[equation]{Proposition}
\newtheorem{lem}[equation]{Lemma}
\theoremstyle{definition}
\newtheorem{dfn}[equation]{Definition}
\newtheorem{rem}[equation]{Remark}
\newtheorem{exa}[equation]{Example}

\newtheorem{que}[equation]{Question}

\newtheorem{conv}[equation]{Convention}

\numberwithin{equation}{section}

\newcommand{\sub}{\subseteq}
\newcommand{\iso}{\xrightarrow{\simeq}}

\newcommand{\xar}{\xrightarrow}
\newcommand{\opn}{\operatorname}
\newcommand{\cat}[1]{\operatorname{\mathsf{#1}}}
\newcommand{\catt}[1]{{\operatorname{\mathsf{#1}}}}

\newcommand{\cd}{\,{\cdot}\,}

\newcommand{\rmitem}[1]{\item[\text{\textup{(#1)}}]}
\newcommand{\mfrak}[1]{\mathfrak{#1}}

\newcommand{\mrm}[1]{\mathrm{#1}}

\newcommand{\Ga}{\Gamma}
\newcommand{\La}{\Lambda}
\newcommand{\si}{\sigma}

\newcommand{\de}{\delta}
\renewcommand{\th}{\theta}

\newcommand{\al}{\alpha}
\newcommand{\be}{\beta}
\newcommand{\ga}{\gamma}
\newcommand{\ep}{\epsilon}
\newcommand{\ze}{\zeta}

\newcommand{\m}{\mfrak{m}}

\renewcommand{\b}{\mfrak{b}}
\renewcommand{\a}{\mfrak{a}}

\renewcommand{\aa}{\bsym{a}}

\newcommand{\K}{\mathbb{K}}

\newcommand{\Z}{\mathbb{Z}}
\newcommand{\N}{\mathbb{N}}

\newcommand{\centover}{ /_{\! \mrm{c}}\,}

\newcommand{\tup}[1]{\textup{#1}}
\newcommand{\bsym}[1]{\boldsymbol{#1}}

\newcommand{\ot}{\otimes}

\newcommand{\til}[1]{\tilde{#1}}
\newcommand{\what}[1]{\widehat{#1}}

\newcommand{\lb}{\linebreak}

\begin{document}

\begin{abstract}
Let $A$ be a commutative ring, and let $\a$ be a finitely generated ideal in 
it. It is known that a necessary and sufficient condition for the derived 
$\a$-torsion and $\a$-adic completion functors to be nicely behaved is the 
{\em weak proregularity} of $\a$. In particular, the {\em MGM Equivalence} 
holds. 

Because weak proregularity is defined in terms of elements of the ring 
(it involves limits of Koszul complexes), it is not suitable for 
noncommutative ring theory. 

In this paper we introduce a new condition on a torsion class $\catt{T}$ in a
module category: {\em weak stability}. Our first main theorem says that in the 
commutative case, the ideal $\a$ is weakly proregular if and only if the 
corresponding torsion class $\catt{T}$ is weakly stable. 

We then study weak stability of torsion classes in module categories over 
noncommutative rings. There are three main theorems in this context:
$\vartriangleright$ For a torsion class $\catt{T}$ that is weakly 
stable, quasi-compact and finite dimensional, the 
right derived torsion functor is isomorphic to a left derived tensor functor. 
\lb $\vartriangleright$ The {\em Noncommutative MGM Equivalence}, that holds 
under the same assumptions on $\catt{T}$. 
$\vartriangleright$ A theorem about {\em symmetric derived torsion} for 
complexes of bimodules. 
This last theorem is a generalization of a result of Van den Bergh from 1997, 
and corrects an error in a paper of Yekutieli \& Zhang from 2003. 
\end{abstract}

\maketitle

\tableofcontents

%\cleardoublepage
\setcounter{section}{-1}
\section{Introduction}

Let $A$ be a commutative ring, and let $\a$ be a finitely generated ideal in 
$A$. The category of $A$-modules is denoted by $\cat{M}(A)$. There are two 
functors on the category $\cat{M}(A)$ that the ideal $\a$ determines: the 
{\em $\a$-torsion functor $\Ga_{\a}$}, and the {\em $\a$-adic completion 
functor $\La_{\a}$}. These are idempotent additive functors: 
$\Ga_{\a} \circ \Ga_{\a} \cong \Ga_{\a}$ and 
$\La_{\a} \circ \La_{\a} \cong \La_{\a}$.

The functors $\Ga_{\a}$ and $\La_{\a}$ seem as though they could be adjoint to 
each other. This is false however. What is true is that under suitable 
assumptions, the derived functors
\[ \mrm{R} \Ga_{\a}, \mrm{L} \La_{\a} : \cat{D}(A) \to \cat{D}(A) \]
are adjoint to each other. Here $\cat{D}(A)$ is the (unbounded) derived 
category of $A$-modules. The most general condition under which this is known 
to hold is when the ideal $\a$ is {\em weakly proregular}. 

By definition, the ideal $\a$ is weakly proregular if it can be generated by a 
{\em weakly proregular sequence}. A sequence of elements 
$\aa = (a_1, \ldots, a_n)$ in $A$ is called weakly proregular if a rather 
complicated condition is satisfied by the {\em Koszul complexes} associated to 
powers of $\aa$; see Definition \ref{dfn:48}. This condition was first stated, 
without a name, by A. Grothendieck (see \cite{LC} and \cite{SGA2}); the name was 
only given by L. Alonso, A. Jeremias and J. Lipman in \cite[Correction]{AJL}.
If $A$ is noetherian, then every ideal in it is weakly proregular; but there 
are many non-noetherian examples (see Examples \ref{exa:1325}-\ref{exa:1326}).  

The next theorem is the culmination of results by Grothendieck \cite{LC}, 
\cite{SGA2}; E. Matlis \cite{Ma}; J.P.C. Greenlees and J.P. May \cite{GM};  
Alonso, Jeremias and Lipman \cite{AJL}; P. Schenzel \cite{Sn}; 
M. Kashiwara and P. Schapira \cite{KS3}; and M. Porta, 
L. Shaul and Yekutieli \cite{PSY1}. There is parallel recent work on these 
matters by L. Positselski \cite{Po}. 

Let $\a \sub A$ be a finitely generated ideal. A complex 
$M \in \cat{D}(A)$ is called {\em derived $\a$-torsion} if 
the canonical morphism $\mrm{R} \Ga_{\a}(M) \to M$ is an isomorphism.
Similarly, a complex $M \in \cat{D}(A)$ is called {\em derived $\a$-adically 
complete} if  the canonical morphism $M \to \mrm{L} \La_{\a}(M)$ is an 
isomorphism. (In \cite{PSY1} these complexes were called {\em cohomologically 
complete} and {\em cohomologically torsion}, respectively; but now we realize 
that the adjective ``derived'' is better suited than ``cohomologically''. See 
Definition \ref{dfn:2025} below.) 
We denote by $\cat{D}(A)_{\a\tup{-tor}}$ and 
$\cat{D}(A)_{\a\tup{-com}}$ 
the full subcategory of $\cat{D}(A)$ on the derived $\a$-torsion
and the derived $\a$-adically complete complexes, respectively. 
These are full triangulated subcategories. 

\begin{thm}[MGM Equivalence, {\cite{PSY1}}] \label{thm:1375}
Let $\a$ be a weakly proregular ideal in a commutative ring $A$. 
\begin{enumerate}
\item The functor $\mrm{L} \La_{\a}$ is right adjoint to the functor 
$\mrm{R} \Ga_{\a}$.

\item The functors $\mrm{R} \Ga_{\a}$ and $\mrm{L} \La_{\a}$ are 
idempotent. 

\item The categories $\cat{D}(A)_{\a\tup{-tor}}$ and 
$\cat{D}(A)_{\a\tup{-com}}$ are the essential images of the functors 
$\mrm{R} \Ga_{\a}$ and $\mrm{L} \La_{\a}$, respectively. 

\item The functor 
\[ \mrm{L} \La_{\a} : \cat{D}(A)_{\a\tup{-tor}} \to 
\cat{D}(A)_{\a\tup{-com}} \]
is an equivalence of triangulated categories, with quasi-inverse 
$\mrm{R} \Ga_{\a}$.
\end{enumerate}
\end{thm}

Actually, item (1) of this theorem is usually called {\em GM Duality}, and it 
is \cite[Erratum, Theorem 9]{PSY1}. Item (2) means that 
$\mrm{R} \Ga_{\a} \circ \mrm{R} \Ga_{\a} \cong \mrm{R} \Ga_{\a}$ and 
$\mrm{L} \La_{\a} \circ \mrm{L} \La_{\a} \cong \mrm{L} \La_{\a}$.
Item (4) is MGM Equivalence itself, as it was called in \cite{PSY1}, and it is 
a slight rephrasing of \cite[Theorem 1.1]{PSY1}. 

{\em The goal of this paper is to find a noncommutative analogue of weak 
proregularity, and to prove a suitable version of the MGM Equivalence.}
It should be emphasized that all prior characterizations of weak proregularity 
were in terms of elements (formulas involving limits of Koszul complexes). 
Such formulas rarely make any sense in the noncommutative setting. 

Let $A$ be a noncommutative  (i.e.\ not necessarily commutative) ring. By 
default $A$-modules are left modules. 
Let $\cat{M}(A)$ be the category of (left) $A$-modules. Recall that a {\em 
torsion class} in $\cat{M}(A)$
is a class of objects $\catt{T} \subseteq \cat{M}(A)$ 
that is closed under taking subobjects, quotients, extensions and 
infinite direct sums. In \cite{St} and other texts, this is called a 
{\em hereditary torsion class}. A module $M \in \cat{M}(A)$ is said to be a
{\em  $\catt{T}$-torsion module} if it belongs to $\catt{T}$. 
The torsion class $\catt{T}$ gives rise to a {\em 
torsion functor} $\Ga_{\catt{T}}$, that is a left exact additive functor from 
$\cat{M}(A)$ to itself. For an $A$-module $M$, the module 
$\Ga_{\catt{T}}(M)$ is the biggest $\catt{T}$-torsion submodule of $M$. 
Thus $M \in \catt{T}$ if and only if $\Ga_{\catt{T}}(M) = M$. 
Following \cite{YZ}, we say that an $A$-module $M$ is {\em $\catt{T}$-flasque} 
if $\mrm{R}^q \Ga_{\catt{T}}(M) = 0$ for all $q > 0$. 
Every injective $A$-module is $\catt{T}$-flasque, but often there are many 
more.

Here is the categorical notion that we propose as a generalization of weak 
proregularity. 

\begin{dfn} \label{dfn:1375}
Let $\catt{T} \subseteq \cat{M}(A)$ be a torsion class. 
We call $\catt{T}$ a {\em weakly stable torsion class} if for every injective 
$A$-module $I$, the module $\Ga_{\catt{T}}(I)$ is $\catt{T}$-flasque.
\end{dfn}

The name ``weakly stable'' reflects the standard usage of the name ``stable'': 
$\catt{T}$ is called a stable torsion class if for every injective 
$A$-module $I$, the module $\Ga_{\catt{T}}(I)$ is injective. See \cite{St}.

When $A$ is commutative and $\a$ is a finitely generated ideal in it, 
the {\em $\a$-torsion class} $\catt{T} \sub \cat{M}(A)$ is  
the class of modules $M$ such that $\Ga_{\a}(M) = M$; and the torsion functor is
$\Ga_{\catt{T}} = \Ga_{\a}$. 

Here is our first main result.

\begin{thm} \label{thm:1377}
Let $A$ be a commutative ring, let $\aa$ be a finite sequence of elements of 
$A$, let $\a$ be the ideal generated by $\aa$, and let $\catt{T}$ be the 
associated torsion class in $\cat{M}(A)$. The following two conditions 
are equivalent\tup{:} 
\begin{enumerate}
\rmitem{i} The sequence $\aa$ is weakly proregular. 

\rmitem{ii} The torsion class $\catt{T}$ is weakly stable.
\end{enumerate}
\end{thm}

This is repeated as Theorem \ref{thm:47} in Section \ref{sec:comm-rings} and 
proved there. For other examples of weakly stable torsion classes, see Examples 
\ref{exa:1315}-\ref{exa:1066}. 

We now move back to the noncommutative setting. 
Consider a noncommutative ring $A$, that is 
central and flat over a commutative base ring $\K$. The opposite ring is 
$A^{\mrm{op}}$, and the enveloping ring is $A^{\mrm{en}} := A \ot_{\K} 
A^{\mrm{op}}$.  Thus $\cat{M}(A^{\mrm{en}})$ is the category of $\K$-central 
$A$-bimodules, and $\cat{D}(A^{\mrm{en}})$ is its derived category.

The assumption that $A$ is flat over $\K$ is only for the sake of 
simplicity. A general treatment, not assuming flatness, requires 
the use of DG rings, and is substantially more involved. See Remarks 
\ref{rem:1420} and \ref{rem:1350}.

In order to state a noncommutative analogue of the MGM Equivalence, we must 
first introduce a suitable categorical framework. There are several ingredients 
involved:
\begin{itemize}
\item Certain properties  of triangulated functors and torsion classes
(namely: quasi-compact, weakly stable, finite dimensional and idempotent). This 
is done in Sections \ref{sec:cpct-funcs}-\ref{sec:tors-cls}.

\item Derived categories of bimodules, derived functors between them, and 
the monoidal structure on $\cat{D}(A^{\mrm{en}})$.
This is done in Section \ref{sec:bimods}.

\item Idempotent copointed objects in the monoidal category 
$\cat{D}(A^{\mrm{en}})$, and the triangulated functors they induce by monoidal 
actions. See Section \ref{sec:cop-obj}. 
\end{itemize}
 
Given another flat central $\K$-ring $B$, the category of $\K$-central 
$A$-$B$-bimodules is $\cat{M}(A \ot_{\K} B^{\mrm{op}})$, and 
its derived category is $\cat{D}(A \ot_{\K} B^{\mrm{op}})$. 
A torsion class $\catt{T}$ in $\cat{M}(A)$ gives rise to a derived torsion 
functor $\mrm{R} \Ga_{\catt{T}}$ on the category 
$\cat{D}(A \ot_{\K} B^{\mrm{op}})$. 
There is a left monoidal action 
$(- \ot_A^{\mrm{L}} -)$ of $\cat{D}(A^{\mrm{en}})$ on 
$\cat{D}(A \ot_{\K} B^{\mrm{op}})$.

Here is our second main result. 

\begin{thm}[Representability of Derived Torsion] \label{thm:1405}
Let $A$ and $B$ be flat central $\K$-rings. Let $\catt{T}$ be a quasi-compact, 
finite dimensional, weakly stable torsion class in $\cat{M}(A)$.
Define the object 
\[ P := \mrm{R} \Ga_{\catt{T}}(A) \in \cat{D}(A^{\mrm{en}}) . \]
Then there is an isomorphism 
\[ P \ot_{A}^{\mrm{L}} (-) \cong  \mrm{R} \Ga_{\catt{T}} \]
of triangulated functors from $\cat{D}(A \ot_{\K} B^{\mrm{op}})$ to itself.
\end{thm}

This is part of Theorem \ref{thm:1200}, which is more detailed.

There is a canonical morphism $\rho : P \to A$ in $\cat{D}(A^{\mrm{en}})$. 
In Theorem \ref{thm:1370} we prove that the pair $(P, \rho)$ is 
an {\em idempotent copointed object} in the monoidal category 
$\cat{D}(A^{\mrm{en}})$, in the sense of Definition \ref{dfn:1185}.

In the commutative weakly proregular situation, 
where $\catt{T}$ is the torsion class defined by an ideal $\a \sub A$ and 
$\K = A$, the object 
$P = \mrm{R} \Ga_{\catt{T}}(A)$ in $\cat{D}(A)$ 
is represented by the {\em infinite dual Koszul complex} 
$\opn{K}^{\vee}_{\infty}(A; \aa)$
associated to a weakly proregular sequence $\aa$ that generates $\a$. 
Positselski \cite{Po} calls the object $P$ a {\em dedualizing complex}. 

Theorems \ref{thm:1200}, \ref{thm:1370} and \ref{thm:1075} combined are the 
technical results needed to prove our remaining main theorems, that are stated 
below. 

Let $\catt{T} \sub \cat{M}(A)$ be a torsion class. 
As in the commutative setting, we say that a complex $M \in \cat{D}(A)$
is a {\em derived $\catt{T}$-torsion complex} if the canonical morphism
$\mrm{R} \Ga_{\catt{T}}(M) \to M$ is an isomorphism. 
The full subcategory of $\cat{D}(A)$ on the derived $\catt{T}$-torsion 
complexes is denoted by $\cat{D}(A)_{\catt{T} \tup{-tor}}$. 

Let 
$P = \mrm{R} \Ga_{\catt{T}}(A) \in \cat{D}(A^{\mrm{en}})$
be as in Theorem \ref{thm:1405}. It gives rise to a triangulated functor
\begin{equation} \label{eqn:1655}
G_{\catt{T}} : \cat{D}(A) \to \cat{D}(A) , \quad 
G_{\catt{T}} := \opn{RHom}_A(P, -) ,
\end{equation}
that we call the {\em abstract $\catt{T}$-completion functor}. 
A complex $M \in \cat{D}(A)$ is called a {\em derived $\catt{T}$-complete 
complex} if the canonical morphism
$M \to G_{\catt{T}}(M)$ is an isomorphism. 
The full subcategory of $\cat{D}(A)$ on the derived $\catt{T}$-complete 
complexes is denoted by $\cat{D}(A)_{\catt{T} \tup{-com}}$. 

The categories $\cat{D}(A)_{\catt{T} \tup{-tor}}$ and
$\cat{D}(A)_{\catt{T} \tup{-com}}$
are full triangulated subcategories of $\cat{D}(A)$.
See Sections \ref{sec:cop-obj} and \ref{sec:tors-to-obj} for more 
details. 

\begin{thm}[Noncommutative MGM Equivalence] \label{thm:1378}
Let $A$ be a flat central $\K$-ring, and let $\catt{T}$ be a 
quasi-compact, weakly stable, finite dimensional torsion class in $\cat{M}(A)$. 
Then\tup{:}
\begin{enumerate}
\item The functor $G_{\catt{T}}$ is right adjoint to $\mrm{R} \Ga_{\catt{T}}$. 

\item The functors $\mrm{R} \Ga_{\catt{T}}$ and $G_{\catt{T}}$ are 
idempotent. 

\item The categories $\cat{D}(A)_{\catt{T} \tup{-tor}}$
and $\cat{D}(A)_{\catt{T} \tup{-com}}$ 
are the essential images of the functors $\mrm{R} \Ga_{\catt{T}}$ and 
$G_{\catt{T}}$, respectively. 

\item The functor 
\[ \mrm{R} \Ga_{\catt{T}} : \cat{D}(A)_{\catt{T} \tup{-com}} \to 
\cat{D}(A)_{\catt{T} \tup{-tor}} \]
is an equivalence, with quasi-inverse $G_{\catt{T}}$. 
\end{enumerate}
\end{thm}

This is repeated -- in greater detail -- as Theorem \ref{thm:1400} in Section 
\ref{sec:NC-MGM}, and proved there.

Now to our fourth main result. We consider flat central $\K$-rings $A$ and $B$.
Let $\catt{T} \sub \cat{M}(A)$
and $\catt{S}^{\mrm{op}} \sub \cat{M}(B^{\mrm{op}})$
be torsion classes. These extend to bimodule torsion classes
\[ \catt{T} , \, \catt{S}^{\mrm{op}} \sub
\cat{M}(A \ot_{\K} B^{\mrm{op}})  \]
as follows: a bimodule
$M \in \cat{M}(A \ot_{\K} B^{\mrm{op}})$
is $\catt{T}$-torsion if it is so after forgetting the $B$-module 
structure. Likewise (but on reversed sides) for 
$\catt{S}^{\mrm{op}}$-torsion. There are corresponding derived torsion 
functors 
\[ \mrm{R} \Ga_{\catt{T}}, \, 
\mrm{R} \Ga_{\catt{S}^{\mrm{op}}} :
\cat{D}(A \ot_{\K} B^{\mrm{op}}) \to \cat{D}(A \ot_{\K} B^{\mrm{op}}) . \]
Consider a complex $M \in \cat{D}(A \ot_{\K} B^{\mrm{op}})$. 
We say that $M$ has {\em weakly symmetric derived 
$\catt{T}$-$\catt{S}^{\mrm{op}}$-torsion} if 
\[ \opn{H}^q(\mrm{R} \Ga_{\catt{T}}(M)) \, , \
\opn{H}^q(\mrm{R} \Ga_{\catt{S}^{\mrm{op}}}(M)) \ \in \  
\catt{T} \cap \catt{S}^{\mrm{op}} \]
for all $q$. The complex $M$ has {\em symmetric derived 
$\catt{T}$-$\catt{S}^{\mrm{op}}$-torsion} if there is an isomorphism 
\[ \ep_M : \mrm{R} \Ga_{\catt{T}}(M)) \iso 
\mrm{R} \Ga_{\catt{S}^{\mrm{op}}}(M) \]
in $\cat{D}(A \ot_{\K} B^{\mrm{op}})$, called a {\em symmetry isomorphism},
that respects the canonical morphisms to $M$. 
See Definition \ref{dfn:2000} for details. 
Of course symmetric implies weakly symmetric. 

\begin{thm}[Symmetric Derived Torsion] \label{thm:1379}
Let $A$ and $B$ be flat central $\K$-rings, and let $\catt{T} \sub \cat{M}(A)$ 
and $\catt{S}^{\mrm{op}} \sub \cat{M}(B^{\mrm{op}})$
be quasi-compact, weakly stable, finite dimensional torsion classes. 
Let $M \in \cat{D}(A \ot_{\K} B^{\mrm{op}})$
be a complex with weakly symmetric derived 
$\catt{T}$-$\catt{S}^{\mrm{op}}$-torsion.

Then $M$ has symmetric derived $\catt{T}$-$\catt{S}^{\mrm{op}}$-torsion. 
Moreover, the symmetry isomorphism $\ep_M$ is unique, and it is functorial in 
such complexes $M$. 
\end{thm}

Theorem \ref{thm:1379} is repeated as Theorem \ref{thm:1401} in Section 
\ref{sec:NC-MGM}, and proved there. This theorem is a correction of 
\cite[Theorem 1.23]{YZ}; see Remark \ref{rem:1400} for details. 

Theorems \ref{thm:1378} and \ref{thm:1379} are expected to serve as the 
foundation for a proof (along the lines of the proofs by M. Van den Bergh in 
\cite{VdB} and Q.S. Wu and J.J. Zhang in \cite{WZ}) of the existence of a 
balanced dualizing complex
over a noncommutative ring $A$ that is noetherian,  semilocal, complete and of
mixed characteristics (i.e.\ it does not contain a field). 
This is outlined in the lecture notes \cite{Ye7}, and is work in progress 
\cite{VY}. 

{\bf Acknowledgments.}
We wish to thank James Zhang and Liran Shaul for helpful conversations. 
Thanks also to the anonymous referee, for a careful reading of the paper and 
some
useful suggestions.

%\cleardoublepage
\section{Quasi-Compact Finite Dimensional Functors} 
\label{sec:cpct-funcs}

In this section we discuss several finiteness properties of additive functors, 
that shall play a role in our work. 

Let $A$ be a ring. We work with left $A$-modules. These notations are used: 
the abelian category of $A$-modules is $\cat{M}(A)$, the category of complexes 
$A$-modules is $\cat{C}(A)$, its homotopy category is $\cat{K}(A)$, and the 
derived category is $\cat{D}(A)$.
The categorical localization functor is the triangulated functor 
$\opn{Q} : \cat{K}(A) \to \cat{D}(A)$.
As usual, $\cat{D}^+(A)$, $\cat{D}^-(A)$ and $\cat{D}^{\mrm{b}}(A)$
are the full subcategories of $\cat{D}(A)$ on the complexes with bounded below, 
bounded above and bounded cohomologies, respectively.  
We follow the book \cite{Ye6} in our treatment of derived categories and 
functors, with regards to definitions and notation. Other books on 
the subject include \cite{RD}, \cite{We}, \cite{KS1} and \cite{KS2}. 

For the purpose of describing vanishing conditions for complexes and functors,  
we shall use the following numerical conventions. 
By {\em generalized integer} we mean an element of the ordered set 
$\Z \cup  \{ \pm \infty \}$. A generalized integer $n$ will be called 
{\em finite} if $n < \infty$, i.e.\ if $n \in \Z \cup  \{ -\infty \}$.
This somewhat unusual choice of nomenclature will be quite handy. 

Given generalized integers $d_0 \leq d_1$,
the {\em integer interval} they bound is
\[ [d_0, d_1] := \{ i \in \Z \mid d_0 \leq i \leq d_1 \} . \]
Observe that this will be the empty interval $\varnothing$ when 
$d_0 = d_1 = \infty$ or $d_0 = d_1 = -\infty$. 
If $[e_0, e_1]$ is another integer interval, 
and both are nonempty, then we let 
\[ [d_0, d_1] + [e_0, e_1] := [d_0 + e_0, d_1 + e_1] . \]
For the empty interval $\varnothing$, and another interval $S$, we let 
\[ S + \varnothing  = \varnothing + S := \varnothing . \]
The integer intervals are partially ordered by inclusion. 

An integer interval $S$ has a supremum $\opn{sup}(S)$ and an infimum 
$\opn{inf}(S)$, that are both generalized integers. The amplitude of $S$ is 
\[ \opn{amp}(S) := \opn{sup}(S) - \opn{inf}(S) \in \N \cup  \{ \pm \infty \} . 
\]
Note that for a nonempty interval $S = [d_0, d_1]$ we have 
$\opn{sup}(S) = d_1$, $\opn{inf}(S) = d_0$ and 
$\opn{amp}(S) = d_1 - d_0 \in \N \cup  \{ \infty \}$. 
For the empty interval $S = \varnothing$ we have
$\opn{sup}(S) = -\infty$, $\opn{inf}(S) = \infty$ and 
$\opn{amp}(S) = -\infty$. 

Let $N = \bigoplus_{i \in \Z} N^i$ be a graded $A$-module. 
The {\em concentration} of $N$ is the smallest integer interval 
$\opn{con}(N)$ containing the 
set $\{ i \in \Z \mid N^i \neq 0 \}$. 
We use the abbreviations 
$\opn{inf}(N) := \opn{inf}(\opn{con}(N))$,
$\opn{sup}(N) := \opn{sup}(\opn{con}(N))$ and
$\opn{amp}(N) := \opn{amp}(\opn{con}(N))$.

Using this numerical terminology, a complex $M$ belongs to 
$\cat{D}^{\mrm{b}}(A)$ if and only if $\opn{amp}(\opn{H}(M)) < \infty$;
$M \in \cat{D}^+(A)$ if and only if 
$\opn{inf}(\opn{H}(M)) >  -\infty$; etc.   

It is well-known (see \cite{Sp}, \cite{BN}, \cite{Ke}, \cite[Chapter 
09JD]{SP} or \cite[Section 11]{Ye6}) 
that every complex $M \in \cat{C}(A)$ admits a quasi-isomorphism
$M \to I$, where $I$ is a K-injective complex, each $I^q$ is an injective 
$A$-module, and 
$\opn{inf}(I) =  \opn{inf}(\opn{H}(M))$. 

Let $B$ be another ring, and let $F : \cat{M}(A) \to \cat{M}(B)$ be an 
additive functor. The functor $F$ extends in the obvious way to a  
functor $F : \cat{C}(A) \to \cat{C}(B)$
on complexes of modules, and this induces a triangulated functor 
$F : \cat{K}(A) \to \cat{K}(B)$.
The triangulated functor $F$ admits a right derived functor 
$(\mrm{R} F, \xi^{\mrm{R}})$. Recall that 
\[ \mrm{R} F : \cat{D}(A) \to \cat{D}(B) \]
is a triangulated functor, and 
\[ \xi^{\mrm{R}} : F \to \mrm{R} F \circ \opn{Q} \]
is a morphism of triangulated functors 
$\cat{K}(A) \to \cat{D}(B)$
that has a certain universal property. 
The right derived functor $(\mrm{R} F, \xi^{\mrm{R}})$
can be constructed using a K-injective presentation: for each 
$M \in  \cat{K}(A)$ we choose a K-injective resolution 
$\ze_M : M \to I_M$, and then we take 
$\mrm{R} F(M) := F(I_M)$ and $\xi^{\mrm{R}}_M := F(\ze_M)$. 
See \cite[Section 8]{Ye6}. 

Let $F : \cat{M}(A) \to \cat{M}(B)$ be a left exact functor.
The classical right derived functors of $F$ are 
\[ \mrm{R}^q F = \opn{H}^q(\mrm{R} F) : \cat{M}(A) \to \cat{M}(B) . \]
Recall that the {\em right cohomological dimension of $F$} is 
\[ n := \opn{sup} \, \{ q \in \N \mid \mrm{R}^q F \neq 0 \} \in 
\N \cup \{ \pm \infty \} . \]
If $F \neq 0$ then $n \in \N \cup \{ \infty \}$, but for $F = 0$ the 
dimension is $n = -\infty$. 
Our convention regarding finiteness (a generalized integer $n$ is finite if and 
only if $n < \infty$) was designed to give the zero functor 
finite right cohomological dimension. 

\begin{dfn} \label{dfn:1055}
Let $F : \cat{M}(A) \to \cat{M}(B)$ be a left exact additive functor.
A module $I \in \cat{M}(A)$ is called a {\em right $F$-acyclic module} if 
$\mrm{R}^q F(I) = 0$ for all $q > 0$. 
\end{dfn}

Of course every injective $A$-module is a right $F$-acyclic 
module. But often there are many more right $F$-acyclic modules.

\begin{dfn} \label{dfn:1300}
Let $\cat{M}$ and $\cat{N}$ be additive categories that admit infinite direct 
sums, and let $F : \cat{M} \to \cat{N}$ be an additive functor.
The functor $F$ is called {\em quasi-compact} if it commutes with infinite 
direct sums. Namely, for every collection 
$\{ M_x \}_{x \in X}$ of objects of $\cat{M}$, indexed by some set $X$, 
the canonical morphism
\[ \bigoplus\nolimits_{x \in X} F(M_x) \to 
F \bigl( \bigoplus\nolimits_{x \in X} M_x \bigr) \]
in $\cat{N}$ is an isomorphism. 
\end{dfn}

The name ``quasi-compact functor'' is inspired by the property of pushforward 
of 
quasi-coherent sheaves along a quasi-compact map of schemes. 

\begin{lem} \label{lem:1056}
Let  $F : \cat{M}(A) \to \cat{M}(B)$ be a left exact additive functor, and 
assume that all the functors $\mrm{R}^q F$ are quasi-compact. Let 
$\{ I_x \}_{x \in X}$ be a collection of right $F$-acyclic $A$-modules.
Then the $A$-module $I := \bigoplus_{x \in X} I_x$ is right $F$-acyclic.
\end{lem}

\begin{proof}
Take any $q > 0$. Because $\mrm{R}^qF$ is quasi-compact, the canonical 
homomorphism 
\[ \bigoplus_{x \in X} \, \mrm{R}^q F(I_x) \to \mrm{R}^q F(I)  \]
in $\cat{M}(B)$ is an isomorphism. But by assumption, $\mrm{R}^q F(I_x) = 0$ 
for 
all $x$. 
\end{proof}

\begin{dfn} \label{dfn:1301}
Let $F : \cat{M}(A) \to \cat{M}(B)$ be an additive functor. 
A complex $I \in \cat{C}(A)$ is called a {\em right $F$-acyclic complex} 
if the morphism $\xi^{\mrm{R}}_I : F(I) \to \mrm{R} F(I)$
in $\cat{D}(B)$ is an isomorphism. 
\end{dfn}

Of course every K-injective complex is right $F$-acyclic, but often there 
are others.

In case $F$ is a left exact functor, so that Definition \ref{dfn:1055} applies, 
it is easy to see that an $A$-module $I$ is right $F$-acyclic if and only if 
it is right $F$-acyclic as a complex; i.e.\ Definitions \ref{dfn:1301} and 
\ref{dfn:1055} agree in this case. 

\begin{lem} \label{lem:1057}
Let $F : \cat{M}(A) \to \cat{M}(B)$ be a left exact additive functor, and let 
$I \in \cat{C}(A)$ be a complex such that each of the modules $I^q$ is right 
$F$-acyclic. If $I$ is a bounded below complex, or if $F$ has finite right 
cohomological dimensional, then $I$ is a right $F$-acyclic complex. 
\end{lem}

\begin{proof}
A proof of this assertion is sketched within the proof of 
\cite[Corollary I.5.3]{RD}. A detailed argument can be found in 
\cite[version 3, Lemma 16.1.5]{Ye6}.
\end{proof}

Let $F : \cat{D}(A) \to \cat{D}(B)$
be a triangulated functor, and let $\cat{E} \subseteq \cat{D}(A)$
be a class of objects. The {\em cohomological displacement} of $F$ relative 
to $\cat{E}$ is the smallest integer interval $S$ such that 
\[ \opn{con}(\opn{H}(F(M))) \subseteq \opn{con}(\opn{H}(M)) + S \]
for every $M \in \cat{E}$. 
The {\em cohomological dimension} of $F$ relative 
to $\cat{E}$ is the amplitude of its cohomological displacement.
Note that for the zero functor $F$, its cohomological displacement is the empty 
interval, and its cohomological dimension is $-\infty$. 
It is clear that if $\cat{E} \subseteq \cat{E}'$, then the  cohomological 
dimension of $F$ relative to $\cat{E}'$ is greater than or equal to its 
dimension relative to $\cat{E}$. 

\begin{dfn} \label{dfn:1000}
Let $F : \cat{D}(A) \to \cat{D}(B)$ be a triangulated functor.
The {\em cohomological dimension} of $F$ is its cohomological dimension 
relative to $\cat{D}(A)$.
\end{dfn}

Note that for a left exact additive functor 
$F : \cat{M}(A) \to \cat{M}(B)$,
the cohomological dimension of the triangulated functor 
$\mrm{R} F : \cat{D}(A) \to \cat{D}(B)$ relative to the subclass
$\cat{M}(A) \subseteq \cat{D}(A)$ equals the 
right cohomological dimension of the functor $F$. As mentioned above, this 
number is less than or equal to the cohomological dimension of $\mrm{R} F$. But 
it turns out that these two generalized integers are equal:

\begin{prop} \label{prop:1160}
Let $F : \cat{M}(A) \to \cat{M}(B)$ be a left exact additive functor. 
Then the cohomological dimension of 
$\mrm{R} F : \cat{M}(A) \to \cat{M}(B)$ equals the 
right cohomological dimension of $F$. 
\end{prop}

\begin{proof}
We can assume that $F \neq 0$, and it has finite right 
cohomological dimension, say $d \in \N$. 
We shall prove that the cohomological displacement of $\mrm{R} F$ is 
$[0, d]$. 

By considering a module $M$ for which $F(M) \neq 0$, 
and a module $N$ for which $\mrm{R}^d F(M) \neq 0$, we see that 
the cohomological displacement of $\mrm{R} F$ contains the integer interval 
$[0, d]$. 

We must prove the converse, namely that for every  $M \in \cat{D}(A)$
the interval $\opn{con}(\opn{H}(\mrm{R} F(M)))$ is contained in the interval 
$\opn{con}(\opn{H}(M)) + [0, d]$. This is done by cases, and we can assume that 
$M \neq 0$. 

\medskip \noindent 
Case 1. If $M$ is not in $\cat{D}^{-}(A)$, 
i.e.\ $\opn{con}(\opn{H}(M)) = [d_0, \infty]$ for some 
$d_0 \in \Z \cup \{ -\infty \}$, then we can take a K-injective resolution 
$M \to I$ such that $\opn{inf}(I) = d_0$. Then 
$\opn{H}(\mrm{R} F(M)) = \opn{H}(F(I))$
is concentrated in $[d_0, \infty]$. 

\medskip \noindent 
Case 2. Here $M \in \cat{D}^{-}(A)$, so that 
$\opn{con}(\opn{H}(M)) = [d_0, d_1]$ for some 
$d_0 \in \Z \cup \{ -\infty \}$ and $d_1 \in \Z$. 
We now take a  K-injective resolution $M \to I$ such that $\opn{inf}(I) = d_0$
and $I$ is made up of injective $A$-modules. We must prove that 
$\opn{H}^q(F(I)) = 0$ for all $q > d_1 + d$. This is done like in the 
proof of Lemma \ref{lem:1057}. 
Let $N := \opn{Z}^{d_1}(I)$, and let 
\[ J := (\cdots \to 0 \to I^{d_1} \to  I^{d_1 + 1} \to \cdots) , \]
the complex with $I^{d_1}$ placed in degree $0$. 
So there is a quasi-isomorphism $N \to J$, and it is an injective resolution of 
$N$. Hence for every $p > d$, letting $q := d_1 + p$, we have
\[ \opn{H}^{q}(F(I))\cong \opn{H}^{p}(F(J))
\cong \opn{H}^p(\mrm{R} F(N)) = 0 . \]
\end{proof}

It is known that the category $\cat{D}(A)$ has infinite direct sums, and they 
are the same as the direct sums in $\cat{C}(A)$. See \cite[Lemma 1.5]{BN} or 
\cite[Theorem 10.1.15]{Ye6}.
Therefore Definition \ref{dfn:1300} applies to triangulated functors 
$F : \cat{D}(A) \to \cat{D}(B)$, 
and we can talk about quasi-compact triangulated functors. 

\begin{exa} \label{exa:1220}
Suppose $B = \Z$, and $F : \cat{D}(A) \to \cat{D}(\Z)$ 
is the triangulated functor $F = \opn{RHom}_A(P, -)$ for some complex 
$P \in \cat{D}(A)$. Recall that the complex $P$ is called a {\em compact 
object of $\cat{D}(A)$} if the functor $F$ is quasi-compact, in the sense of 
Definition \ref{dfn:1300}. It is known that $P$ is a compact object of 
$\cat{D}(A)$ if and only if $P$ is isomorphic to a bounded complex of finitely 
generated projective $A$-modules; see \cite{BV} or
 \cite[version 3, Theorems 14.1.23 and 14.1.26]{Ye6}. Therefore in this case, 
if $F$ is a quasi-compact functor then it is also a finite dimensional functor, 
in the sense of Definition \ref{dfn:1000}. (In general these attributes are 
independent of each other.) 
\end{exa}

\begin{exa} \label{exa:1221}
Assume $F = P \ot_A^{\mrm{L}} (-)$ for some complex 
$P \in \cat{D}(A^{\mrm{op}})$; here $A^{\mrm{op}}$ is the opposite ring. Again 
the target of $F$ is $\cat{D}(\Z)$.
Then $F$ is a quasi-compact functor. It is a finite dimensional functor if and 
only if the complex $P$ has finite flat dimension, i.e.\ if $P$ is isomorphic 
to 
a bounded complex of flat $A^{\mrm{op}}$-modules. 
\end{exa}

\begin{lem} \label{lem:1058}
Let $F : \cat{M}(A) \to \cat{M}(B)$ be a left exact additive functor, such that 
all the functors $\mrm{R}^q F$ are quasi-compact.
Let $\{ I_x \}_{x \in X}$ be a collection of right $F$-acyclic 
complexes in $\cat{C}(A)$, and define 
$I := \bigoplus_{x \in X} I_x$. 
If $I$ is a bounded below complex, or if $F$ has finite right cohomological 
dimension, then $I$ is a right $F$-acyclic complex.
\end{lem}

\begin{proof}
For every index $x$ let choose a quasi-isomorphism 
$\phi_x : I_x \to J_x$, where $J_x$ is a K-injective complex consisting of 
injective $A$-modules, and 
$\opn{inf}(J_x) \geq \opn{inf}(I_x)$. Define 
$J := \bigoplus_{x \in X} J_x$.
We get a quasi-isomorphism 
\[ \phi := \bigoplus_{x \in X} \, \phi_x : I \to J. \]
By construction, if $I$ is a bounded below complex, then so is $J$. 

For every $x$ there is a commutative diagram 
\[ \UseTips \xymatrix @C=9ex @R=6ex {
F(I_x)
\ar[r]^{F(\phi_x)}
\ar[d]_{\xi^{\mrm{R}}_{I_x}}
& 
F(J_x)
\ar[d]^{\xi^{\mrm{R}}_{J_x}}
\\
\mrm{R} F(I_x)
\ar[r]^(0.5){\mrm{R} F(\phi_x)}
& 
\mrm{R} F(J_x)
} \]
in $\cat{D}(A)$. The vertical arrows are isomorphisms because both $I_x$ and 
$J_x$ are right $F$-acyclic complexes. The morphism $\mrm{R} F(\phi_x)$ is 
also an isomorphism. It follows that $F(\phi_x)$ is an isomorphism in 
$\cat{D}(A)$; and therefore it is a quasi-isomorphism in $\cat{C}(A)$. 

Next consider this commutative diagram in $\cat{C}(A)$~:
\[ \UseTips \xymatrix @C=14ex @R=6ex {
\bigoplus_{x \in X} \, F(I_x)
\ar[r]^{ \bigoplus_{x \in X} \, F(\phi_x) }
\ar[d]_{}
& 
\bigoplus_{x \in X} \, F(J_x)
\ar[d]
\\
F(I)
\ar[r]^{ F(\phi) }
& 
F(J)
} \]
The previous paragraph, and the fact that cohomology commutes with infinite 
direct sums, tell us that the top horizontal arrow is a quasi-isomorphism.
Because the functor $F = \mrm{R}^0 F$ is quasi-compact, the vertical arrows are 
isomorphisms. We conclude that $F(\phi)$ is a quasi-isomorphism. 

Finally we look at this  commutative diagram in $\cat{D}(A)$~:
\[ \UseTips \xymatrix @C=8ex @R=6ex {
F(I)
\ar[r]^{F(\phi)}
\ar[d]_{\xi^{\mrm{R}}_I}
& 
F(J)
\ar[d]^{\xi^{\mrm{R}}_J}
\\
\mrm{R} F(I)
\ar[r]^{ \mrm{R} F(\phi) }
& 
\mrm{R} F(J)
} \]
We know that the morphisms $F(\phi)$ and $\mrm{R} F(\phi)$ are isomorphisms. 
For every $p$, the module 
$J^p = \bigoplus_{x \in X} \, J_x^p$ 
is a direct sum of injective modules. So according to Lemma \ref{lem:1056}, 
$J^p$ is a right $F$-acyclic module.
By Lemma \ref{lem:1057} (in either of the two cases) the complex $J$ is a 
right $F$-acyclic complex.
This says that the morphism $\xi^{\mrm{R}}_J$ is an isomorphism.
We conclude that the morphism $\xi^{\mrm{R}}_I$ is an isomorphism too, and this 
says that $I$ is a right $F$-acyclic complex.
\end{proof}

\begin{thm} \label{thm:1001}
Let $F : \cat{M}(A) \to \cat{M}(B)$ be a left exact additive functor.
Assume that $F$ has finite right cohomological dimension, and each 
$\mrm{R}^q F$ is quasi-compact. Then the triangulated functor
$\mrm{R} F : \cat{D}(A) \to \cat{D}(B)$ 
is quasi-compact and finite dimensional.
\end{thm}

\begin{proof}
By Proposition \ref{prop:1160} the functor $\mrm{R} F$ is finite 
dimensional. We need to  prove that $\mrm{R} F$ commutes with 
infinite direct sums. Namely, consider a collection \lb $\{ M_x \}_{x \in X}$ 
of complexes, and let $M := \bigoplus_{x \in X} M_x$. We have to prove that the 
canonical morphism 
\[ \bigoplus_{x \in X} \, \mrm{R} F(M_x) \to  \mrm{R} F(M) \]
in $\cat{D}(A)$ is an isomorphism. 

For every index $x$ we choose a quasi-isomorphism 
$\phi_x : M_x \to I_x$ in $\cat{C}(A)$,
where $I_x$ is a K-injective complex. 
Define $I := \bigoplus_{x \in X} I_x$, so there is a quasi-isomorphism 
$\phi : M \to I$ in $\cat{C}(A)$. We get a commutative diagram 
\[ \UseTips \xymatrix @C=14ex @R=6ex {
\bigoplus_{x \in X} \mrm{R} F(M_x)
\ar[r]^{ \bigoplus \mrm{R} F(\phi_x) } 
\ar[d]_{}
& 
\bigoplus_{x \in X} \mrm{R} F(I_x)
\ar[d]
\\
\mrm{R} F(M)
\ar[r]^{ \mrm{R} F(\phi) }
& 
\mrm{R} F(I)
} \]
in $\cat{D}(A)$, in which the horizontal arrows are isomorphisms. 
Therefore it suffices to prove that the canonical morphism 
\begin{equation} \label{eqn:1060}
\bigoplus_{x \in X} \, \mrm{R} F(I_x) \to \mrm{R} F(I) 
\end{equation}
in $\cat{D}(A)$ is an isomorphism. 

Consider this commutative diagram in $\cat{D}(A)$~:
\[ \UseTips \xymatrix @C=12ex @R=6ex {
\bigoplus_{x \in X} \, F(I_x)
\ar[r]^{ \bigoplus \xi^{\mrm{R}}_{I_x} }
\ar[d]_{}
& 
\bigoplus_{x \in X} \, \mrm{R} F(I_x)
\ar[d]
\\
F(I)
\ar[r]^{ \xi^{\mrm{R}}_I }
& 
\mrm{R} F(I)
} \]
For each $x$ the morphism $\xi^{\mrm{R}}_{I_x}$ is an isomorphism; and hence 
the 
top 
horizontal arrow is an isomorphism. 
Because $F = \mrm{R}^0 F$ is quasi-compact, the left vertical arrow is an 
isomorphism (in $\cat{C}(A)$, and so also in $\cat{D}(A)$). By Lemma 
\ref{lem:1058} the complex $I$ is right $F$-acyclic, and therefore 
$\xi^{\mrm{R}}_I$ is 
an isomorphism. Hence the remaining arrow is an isomorphism; but this is the 
morphism (\ref{eqn:1060}).
\end{proof}

%\cleardoublepage
\section{Weakly Stable and Idempotent Copointed Functors} 
\label{sec:idem-cop-func}

Again $A$ is a ring. The category of left $A$-modules is $\cat{M}(A)$. In this 
section we introduce a new property of additive functors from 
$\cat{M}(A)$ to itself. 

\begin{dfn} \label{dfn:1225}
Let $F : \cat{M}(A) \to \cat{M}(A)$ be a left exact additive functor.
\begin{enumerate}
\item The functor $F$ is called {\em weakly stable} if for every injective 
$A$-module $I$, the $A$-module $F(I)$ is right $F$-acyclic. 

\item The functor $F$ is called {\em stable} if for every injective $A$-module 
$I$, the $A$-module $F(I)$ is injective.  
\end{enumerate}
\end{dfn}

When we say that a functor $F$ is stable or weakly stable, it is always 
implied that $F$ is a left exact additive functor. 
Clearly stable implies weakly stable. The reason for these names will become 
apparent in the next section, when we talk about torsion functors. 

\begin{dfn} \label{dfn:1003} 
Let $\cat{N}$ be an additive category (e.g.\ $\cat{M}(A)$ or $\cat{D}(A)$), 
with identity functor $\opn{Id}_{\cat{N}}$.
\begin{enumerate}
\item A {\em copointed additive functor} on $\cat{N}$ is a pair 
$(F, \si)$, consisting of an additive functor 
$F : \cat{N} \to \cat{N}$, and morphism of functors 
$\si : F \to \opn{Id}_{\cat{N}}$.

\item The copointed additive functor $(F, \si)$ is called {\em idempotent} if 
the morphisms 
\[ \si_{F(N)}, \, F(\si_N) : F(F(N)) \to F(N) \]
are isomorphisms for all objects $N \in \cat{N}$.

\item If $\cat{N}$ is a triangulated category, $F$ is a triangulated functor, 
and $\si$ is a morphism of triangulated functors, then we call 
$(F, \si)$ a {\em copointed triangulated functor}. 
\end{enumerate}
\end{dfn}

See Remark \ref{rem:1415} regarding the name ``copointed functor''. 

Weak stability and idempotence together have the following effect. 

\begin{prop} \label{prop:1162}
Let $(F, \si)$ be an idempotent copointed additive functor on 
$\cat{M}(A)$, and assume that $F$ is weakly stable. If $I$ is a 
right $F$-acyclic module, then $F(I)$ is also a right $F$-acyclic module.
\end{prop}

\begin{proof}
Choose an injective resolution $\eta : I \to J$; i.e.\ $J$ is a complex of 
injectives concentrated in nonnegative degrees, and $\eta$ is a 
quasi-isomorphism. Since 
$\opn{H}^q(F(J)) \cong \mrm{R}^q F(I)$, and since $I$ is a right $F$-acyclic 
module, we see that the homomorphism of complexes
$F(\eta) : F(I) \to F(J)$ is a quasi-isomorphism. 
Therefore both $F(\eta)$ and $\mrm{R} F(F(\eta))$ are isomorphisms in 
$\cat{D}(A)$. 
The weak stability of $F$ implies that $F(J)$ is a bounded below complexes of 
right $F$-acyclic modules. According to Lemma \ref{lem:1057} the 
complex $F(J)$ is right $F$-acyclic. This means that the morphism 
$\xi^{\mrm{R}}_{F(J)}$ is an isomorphism. 
The idempotence of $F$ says that the morphisms 
$\si_{F(I)}$ and $\si_{F(J)}$ are isomorphisms (already in $\cat{C}(A)$). 
We get a commutative diagram in $\cat{D}(A)$: 
\[ \UseTips \xymatrix @C=8ex @R=8ex {
F(I)
\ar[d]_{F(\eta)}^{\cong}
&
F(F(I))
\ar[l]_{\si_{F(I)}}^{\cong}
\ar[r]^{\xi^{\mrm{R}}_{F(I)}}
\ar[d]_{F(F(\eta))}
&
\mrm{R} F(F(I))
\ar[d]_{\mrm{R} F(F(\eta))}^{\cong}
\\
F(J)
&
F(F(J))
\ar[l]_{\si_{F(J)}}^{\cong}
\ar[r]^{\xi^{\mrm{R}}_{F(J)}}_{\cong}
&
\mrm{R} F(F(J))
} \]
We conclude that $\xi^{\mrm{R}}_{F(I)}$ is an isomorphism; and this means that 
$F(I)$ is right $F$-acyclic. 
\end{proof}

The next lemma is a generalization of \cite[Proposition 3.10]{PSY1}.

\begin{lem} \label{lem:1063}
Suppose we are given a copointed additive functor $(F, \si)$ on $\cat{M}(A)$. 
Then there is a unique morphism 
\[ \si^{\mrm{R}} : \mrm{R} F \to \opn{Id}_{\cat{D}(A)} \]
of triangulated functors from $\cat{D}(A)$ to itself, satisfying this 
condition\tup{:} for every $M \in \cat{D}(A)$ there is equality 
$\si^{\mrm{R}}_M \circ \xi^{\mrm{R}}_M = \si_M$
of morphisms $F(M) \to M$ in $\cat{D}(A)$. 
\end{lem}

In a commutative diagram:
\begin{equation} \label{eqn:1455}
\UseTips \xymatrix @C=6ex @R=6ex {
F(M)
\ar[r]^{ \xi^{\mrm{R}}_M }
\ar[dr]_{ \si_M }
&
\mrm{R} F(M)
\ar[d]^{ \si^{\mrm{R}}_M }
\\
&
M
} 
\end{equation}

\begin{proof}
The existence of the morphism $\si^{\mrm{R}}$ comes for free from the universal 
property of the right derived functor. Still, for later reference, we give the 
construction. 

For a K-injective complex $I$ the morphism 
$\xi^{\mrm{R}}_I : F(I) \to \mrm{R} F(I)$ in $\cat{D}(A)$ is an isomorphism, 
and we define 
$\si^{\mrm{R}}_I : \mrm{R} F(I) \to I$
to be 
$\si^{\mrm{R}}_I := \si_I \circ (\xi^{\mrm{R}}_I)^{-1}$.
For an arbitrary complex $M$ we choose a quasi-isomorphism 
$\eta : M \to I$ into a K-injective complex, and then we let 
\[ \si^{\mrm{R}}_M := \eta^{-1} \circ \si^{\mrm{R}}_I \circ \mrm{R} F(\eta)  \]
in $\cat{D}(A)$. The corresponding commutative diagram in 
$\cat{D}(A)$ is this: 
\[ \UseTips \xymatrix @C=8ex @R=8ex {
F(M)
\ar[r]^{\xi^{\mrm{R}}_M}
\ar[d]_{F(\eta)}
\ar@(u,u)[rr]^{\si_M}
&
\mrm{R} F(M)
\ar[d]_{\mrm{R} F(\eta)}^{\cong}
\ar[r]^{\si^{\mrm{R}}_M}
&
M
\ar[d]_{\eta}^{\cong}
\\
F(I)
\ar[r]^{\xi^{\mrm{R}}_I}_{\cong}
\ar@(d,d)[rr]_{\si_I}
&
\mrm{R} F(I)
\ar[r]^{\si^{\mrm{R}}_I}
&
I
} \]

It is easy to see that the collection of morphisms 
$\{ \si^{\mrm{R}}_M \}_{M \in \cat{D}(A)}$ 
has the desired properties. 
\end{proof}

In this way we obtain a copointed triangulated functor 
$(\mrm{R} F, \si^{\mrm{R}})$ on $\cat{D}(A)$.

\begin{thm} \label{thm:1004}
Let $(F, \si)$ be an  idempotent copointed additive functor on $\cat{M}(A)$.
Assume that $F$ is weakly stable and has finite right cohomological dimension.  
Then the copointed triangulated 
functor $(\mrm{R} F, \si^{\mrm{R}})$ on $\cat{D}(A)$
is idempotent. 
\end{thm}

\begin{proof}
Let $M$ be a complex of $A$-modules. Choose an isomorphism 
$M \iso I$ in $\cat{D}(A)$, where $I$ is a complex consisting of injective 
$A$-modules (e.g.\ take a K-injective resolution $M \to I$, such that $I$ 
consists of injective modules). It suffices to prove that 
\[ \si^{\mrm{R}}_{\mrm{R} F(I)}, \, \mrm{R} F(\si^{\mrm{R}}_I) : 
\mrm{R} F(\mrm{R} F(I)) \to \mrm{R} F(I) \]
are isomorphisms in $\cat{D}(A)$.

Since each $I^q$ is an injective module, it is  right $F$-acyclic. 
Since $F$ is a weakly stable functor, each of the modules $F(I^q)$ 
is right $F$-acyclic too. The functor $F$ has finite right cohomological 
dimension, so according to Lemma \ref{lem:1057}, the complexes $I$ 
and $F(I)$ are both right $F$-acyclic complexes.  

Consider the diagram 
\begin{equation} \label{eqn:1163}
\UseTips \xymatrix @C=10ex @R=6ex {
F(F(I))
\ar[d]^{ F(\si_I) }
\ar[r]^{ \xi^{\mrm{R}}_{F(I)} }
&
\mrm{R}F(F(I))
\ar[d]^(0.55){ \mrm{R}F(\si_I) }
\ar[r]^{ \mrm{R}F(\xi^{\mrm{R}}_I) }
&
\mrm{R}F(\mrm{R}F(I))
\ar[d]^(0.55){ \mrm{R}F(\si^{\mrm{R}}_I) }
\\
F(I)
\ar[r]^{ \xi^{\mrm{R}}_{I} }
& 
\mrm{R}F(I)
\ar[r]^{ \opn{id} }
&
\mrm{R}F(I)
}
\end{equation}
in $\cat{D}(A)$. 
The left square is commutative: it is gotten from the vertical
morphism $\si_I : F(I) \to I$, to which we apply in the horizontal direction 
the morphism of functors $\xi^{\mrm{R}} : F \to \mrm{R} F$. The right square is 
also commutative: it comes from applying the functor $\mrm{R}F$ to the 
commutative diagram 
\[ \UseTips \xymatrix @C=8ex @R=6ex {
F(I)
\ar[r]^{\xi^{\mrm{R}}_I}
\ar[d]_{\si_I}
&
\mrm{R}F(I)
\ar[d]^{\si^{\mrm{R}}_I}
\\
I
\ar[r]^{ \opn{id} }
&
I
} \]
that characterizes $\si^{\mrm{R}}_I$. 
Because $I$ and $F(I)$ are right $F$-acyclic complexes, the morphisms 
$\xi^{\mrm{R}}_I$ and $\xi^{\mrm{R}}_{F(I)}$ are isomorphisms. Hence 
$\mrm{R}F(\xi^{\mrm{R}}_I)$ is an 
isomorphism. So the horizontal morphisms in the diagram \ref{eqn:1163} are all 
isomorphisms. We are given that $F$ is idempotent, and thus $F(\si_{I})$ is an
isomorphism.  The conclusion of this discussion is that 
$\mrm{R}F(\si^{\mrm{R}}_I)$ is an isomorphism.

Next, let $\phi : \mrm{R}F(I) \to J$ be an isomorphism in $\cat{D}(A)$ to a 
K-injective complex $J$. We know that 
$\xi^{\mrm{R}}_I : F(I) \to \mrm{R}F(I)$ is an isomorphism, and by composing 
them we get the isomorphism 
$\phi \circ \xi^{\mrm{R}}_I : F(I) \to J$ in $\cat{D}(A)$. 
Let $\psi : F(I) \to J$ be a 
quasi-isomorphism in $\cat{C}(A)$ representing $\phi \circ \xi^{\mrm{R}}_I$. 
Since these 
are both right $F$-acyclic complexes, it follows that 
$F(\psi) : F(F(I)) \to F(J)$ is a  quasi-isomorphism.
Consider the following commutative diagram 
\[ \UseTips \xymatrix @C=8ex @R=6ex {
F(F(I))
\ar[r]^{F(\psi)}
\ar[d]^{ \si_{F(I)} } 
&
F(J)
\ar[r]^{\xi^{\mrm{R}}_{J}}
\ar[d]^{\si_{J}}
& 
\mrm{R}F(J)
\ar[d]^(0.55){ \si^{\mrm{R}}_J }
&
\mrm{R}F(\mrm{R}F(I))
\ar[d]^(0.55){ \si^{\mrm{R}}_{\mrm{R}F(I)} }
\ar[l]_{ \mrm{R}F(\phi) }
\\
F(I)
\ar[r]^{\psi}
& 
J
\ar[r]^{\opn{id}}
& 
J
& 
\mrm{R}F(I)
\ar[l]_{\phi} 
} \]
in $\cat{D}(A)$. All 
horizontal arrows here are isomorphisms. 
We are given that $F$ is idempotent, and thus $\si_{F(I)}$ is an
isomorphism. The conclusion is that 
$\si^{\mrm{R}}_{\mrm{R}F(I)}$ is an isomorphism.
\end{proof}

Later on we shall need a notion dual to ``copointed functor''. 

\begin{dfn} \label{dfn:1210} 
Let $\cat{N}$ be an additive category (e.g.\ $\cat{M}(A)$ or $\cat{D}(A)$), 
with identity functor $\opn{Id}_{\cat{N}}$.
\begin{enumerate}
\item A {\em pointed additive functor} on $\cat{N}$ is a pair 
$(G, \tau)$, consisting of an additive functor 
$G : \cat{N} \to \cat{N}$, and morphism of functors 
$\tau : \opn{Id}_{\cat{N}} \to G$.

\item The pointed additive functor $(G, \tau)$ is called {\em idempotent} if 
the morphisms 
\[ \tau_{G(N)}, \, G(\tau_N) : G(N) \to G(G(N)) \]
are isomorphisms for all objects $N \in \cat{N}$. 

\item If $\cat{N}$ is a triangulated category, $G$ is a triangulated functor, 
and $\tau$ is a morphism of triangulated functors, then we call 
$(G, \tau)$ a {\em pointed triangulated functor}. 
\end{enumerate}
\end{dfn}

\begin{rem} \label{rem:1415}
Idempotent copointed functors already appeared in the literature under another 
name: idempotent comonads. Another name for (nearly) the same notion is 
a (Bousfield) colocalization functor, see  e.g.\ \cite{Kr}. 
Dually, idempotent pointed functors are the same 
thing as idempotent monads. See \cite{nLab} for a discussion of these concepts. 

In \cite[Section 4.1]{KS2}, what we call an idempotent pointed functor is 
called a projector. It is proved there that for an idempotent 
pointed functor $(G, \tau)$ and an object $N$,  
the morphisms $\tau_{G(N)}$ and $G(\tau_N)$ are equal. 
The same proof (with arrows reversed) shows that for an idempotent 
copointed functor $(F, \si)$ and an object $N$,  
the morphisms $\si_{F(N)}$ and $F(\si_N)$ are equal. We shall not require these 
facts. 
\end{rem}

%\cleardoublepage
\section{Torsion Classes} \label{sec:tors-cls}

Again $A$ is a ring. The category of left $A$-modules is $\cat{M}(A)$. 

\begin{dfn} \label{dfn:1415}
A  {\em torsion class} in $\cat{M}(A)$ is a class of objects 
$\catt{T} \subseteq \cat{M}(A)$ 
that is closed under taking submodules, quotients, extensions and 
infinite direct sums. 
\end{dfn}

\begin{rem} \label{rem:1465}
Many texts (including \cite[Chapter VI]{St}) use the name ``hereditary torsion 
class'', and the extra adjective ``hereditary'' indicates that $\catt{T}$ is 
closed under taking submodules. Since this distinction never shows up is our 
work (all our torsion classes are hereditary), and since we have 
plenty of other attributes to attach to a torsion class, we decided to 
allow ourselves to simplify the naming. 
\end{rem}

Note that the full subcategory on a torsion class $\catt{T}$ is called a 
{\em localizing subcategory} of $\cat{M}(A)$. 

A  torsion class $\catt{T} \subseteq \cat{M}(A)$ induces a left exact 
additive functor 
$\Ga_{\catt{T}}$ from $\cat{M}(A)$ to itself, called the {\em torsion functor}. 
The formula is this: 
for a module $M$, $\Ga_{\catt{T}}(M)$ is the largest submodule of $M$ that 
belongs to $\catt{T}$. For this reason, a module $M$ that belongs to $\catt{T}$ 
is called a {\em $\catt{T}$-torsion module}. 
As $M$ varies, the inclusions  
$\si_M : \Ga_{\catt{T}}(M) \to M$ become a morphism of functors 
$\si : \Ga_{\catt{T}} \to \opn{Id}_{\cat{M}(A)}$,
and thus the pair $(\Ga_{\catt{T}}, \si)$ is a copointed additive functor. 
This is in fact an idempotent copointed functor, since 
$\Ga_{\catt{T}}(\Ga_{\catt{T}}(M)) =  \Ga_{\catt{T}}(M)$
as submodules of $M$. The class $\catt{T}$ can be recovered from the 
copointed additive functor $(\Ga_{\catt{T}}, \si)$, as follows: 
$M \in \catt{T}$ if and only if $\si_M : \Ga_{\catt{T}}(M) \to M$
is an isomorphism.

A  torsion class $\catt{T}$ determines a  set of left 
ideals $\opn{Filt}(\catt{T})$, called the {\em  Gabriel filter} of $\catt{T}$.
By definition, a left ideal $\a \subseteq A$ belongs to 
$\opn{Filt}(\catt{T})$ if the left module $A / \a$ belongs to $\catt{T}$.
The functor $\Ga_{\catt{T}}$ can be recovered from the Gabriel filter, as 
follows. We view $\opn{Filt}(\catt{T})$ as a partially ordered set by 
inclusion. Then for every module $M$ there is equality 
\[ \Ga_{\catt{T}}(M) = 
\lim_{\substack{ \xar{ } \\ \a \in \opn{Filt}(\catt{T}) }} \, 
\opn{Hom}_A(A / \a, M) \]
of submodules of $M$. In other words, an element $m \in M$ lies inside 
$\Ga_{\catt{T}}(M)$ if and only if $m$ is annihilated by some left ideal 
$\a \in \opn{Filt}(\catt{T})$. 

Here is a source of torsion classes. 

\begin{dfn} \label{dfn:1065}
Let $A$ be a ring, and let $\a$ be a two-sided ideal in $A$ that is finitely 
generated as a left ideal. For an $A$-module 
$M$, we define its {\em $\a$-torsion submodule} to be 
\[ \Ga_{\a}(M) := \lim_{i \to } \, \opn{Hom}_A(A / \a^i, M) \subseteq M . \]
The class of objects 
$\catt{T}_{\a} \subseteq \cat{M}(A)$
is defined to be 
\[ \catt{T}_{\a} := \bigl\{ M \mid \Ga_{\a}(M) = M \bigr\} . \]
\end{dfn}

Note that an element $m \in M$ belongs to $\Ga_{\a}(M)$ if and only if 
$\a^i \cd m = 0$ for $i \gg 0$. 
We require $\a$ to be finitely generated as a left ideal to ensure that 
$\catt{T}_{\a}$ is closed under extensions. 

\begin{dfn} \label{dfn:1037}
Let $\catt{T}$ be a  torsion class in $\cat{M}(A)$.
\begin{enumerate}
\item We call $\catt{T}$ a {\em weakly stable} torsion class if the functor 
$\Ga_{\catt{T}}$ is weakly stable, as in Definition \ref{dfn:1225}(1). 

\item  We call $\catt{T}$ a {\em stable} torsion class if the functor 
$\Ga_{\catt{T}}$ is stable, as in Definition \ref{dfn:1225}(2).

\item We call $\catt{T}$ a {\em quasi-compact} torsion class if for every 
$q \geq 0$ the functor $\mrm{R}^q \Ga_{\catt{T}}$ is quasi-compact,
as in Definition \ref{dfn:1300}.

\item The {\em dimension} of $\catt{T}$ is defined to be the right 
cohomological 
dimension of the functor $\Ga_{\catt{T}}$. 
\end{enumerate}
\end{dfn}

Item (2) in the definition above, i.e.\ the notion of stable torsion class, is 
standard -- see \cite[Section VI.7]{St}. The rest of the definition is new. 

Example \ref{exa:1650} presents a torsion class that is weakly stable but not 
stable. 

\begin{dfn} \label{dfn:1038} 
Let $\catt{T}$ be a torsion class in $\cat{M}(A)$.
\begin{enumerate}
\item A module $I \in \cat{M}(A)$ is called a {\em $\catt{T}$-flasque module} 
if it is right $\Ga_{\catt{T}}$-acyclic, as in Definition \ref{dfn:1055}. 

\item A complex $I \in \cat{C}(A)$ is called a {\em $\catt{T}$-flasque 
complex} if it is right $\Ga_{\catt{T}}$-acyclic, as in Definition 
\ref{dfn:1301}.
\end{enumerate}
\end{dfn}

Item (1) in the definition above is copied from \cite{YZ}. 
Note that for a module $I$, considered also as a complex, the definition is 
consistent. 

\begin{prop} \label{prop:1165}
It $\catt{T}$ is a weakly stable torsion class in $\cat{M}(A)$, then every 
module $M \in \catt{T}$ is $\catt{T}$-flasque. 
\end{prop}

\begin{proof}
Take any $M \in \catt{T}$. We can embed $M$ in an injective module $I^0$. But 
then $M \subseteq \Ga_{\catt{T}}(I^0)$, so we have an exact sequence 
$0 \to M \xar{\eta} \Ga_{\catt{T}}(I^0)$. The cokernel of $\eta$ is in 
$\catt{T}$, so the process can be continued, to give an exact sequence
\[ 0 \to M \xar{\eta} \Ga_{\catt{T}}(I^0) \to \Ga_{\catt{T}}(I^1) \to \cdots , 
\]
where the modules $I^q$ are injective. Writing $J^q := \Ga_{\catt{T}}(I^q)$, we 
get a quasi-iso\-morphism of complexes $\eta : M \to J$.

Because $\catt{T}$ is weakly stable, each $J^q = \Ga_{\catt{T}}(I^q)$ is a right
$\Ga_{\catt{T}}$-acyclic module (i.e.\ a $\catt{T}$-flasque module). 
According to Lemma \ref{lem:1057}, $J$ is a right $\Ga_{\catt{T}}$-acyclic 
complex (i.e.\ a $\catt{T}$-flasque complex). This accounts for the second 
isomorphism in:
\[ \opn{H}^q(\mrm{R}\Ga_{\catt{T}}(M)) \cong 
\opn{H}^q(\mrm{R}\Ga_{\catt{T}}(J)) \cong 
\opn{H}^q(\Ga_{\catt{T}}(J)) \cong 
\opn{H}^q(J) \cong \opn{H}^q(M) . \]
The first and fourth isomorphisms come from the 
quasi-isomorphism $\eta : M \to J$. And 
$\Ga_{\catt{T}}(J) = J$ because $\Ga_{\catt{T}}$ is idempotent.

The conclusion is that 
$\opn{H}^q(\mrm{R}\Ga_{\catt{T}}(M)) = 0$
for $q > 0$. 
\end{proof}

As in Section \ref{sec:idem-cop-func}, the copointed additive functor 
$(\Ga_{\catt{T}}, \si)$ on $\cat{M}(A)$ gives rise to a  copointed triangulated 
functor $(\mrm{R} \Ga_{\catt{T}}, \si^{\mrm{R}})$ on $\cat{D}(A)$.

\begin{dfn} \label{dfn:2025}
Let $\catt{T}$ be a torsion class in $\cat{M}(A)$. 
\begin{enumerate}
\item The category of {\em cohomologically $\catt{T}$-torsion complexes},
denoted by $\cat{D}_{\catt{T}}(A)$, is the full subcategory  of $\cat{D}(A)$
on the complexes $M$ such that $\opn{H}^i(M) \in \catt{T}$ for all $i$.

\item The category of {\em derived $\catt{T}$-torsion complexes},
denoted by $\cat{D}(A)_{\catt{T} \tup{-tor}}$, is the full subcategory  of 
$\cat{D}(A)$ on the complexes $M$ such that
\[ \si^{\mrm{R}}_{M} : \mrm{R} \Ga_{\catt{T}}(M) \to M \]
is an isomorphism.
\end{enumerate}
\end{dfn}

Clearly there is an inclusion 
$\cat{D}(A)_{\catt{T} \tup{-tor}} \sub \cat{D}_{\catt{T}}(A)$,
and these are triangulated subcategories of $\cat{D}(A)$. 

\begin{thm} \label{thm:2031}   %{prop:1465}
Let $\catt{T}$ be a weakly stable finite 
dimensional torsion class in $\cat{M}(A)$. 
The following conditions are equivalent for a complex 
$M \in \cat{D}(A)$. 
\begin{enumerate}
\rmitem{i} $M$ belongs to $\cat{D}(A)_{\catt{T} \tup{-tor}}$.

\rmitem{ii}  $M$ is in the essential image of the functor 
$\mrm{R} \Ga_{\catt{T}}$. 

\rmitem{iii}  $M$ belongs to $\cat{D}_{\catt{T}}(A)$.
\end{enumerate}
\end{thm}

\begin{proof}
The implications (i) $\Rightarrow$ (ii) and (ii) $\Rightarrow$ (iii)
are obvious. 

Let us prove that (iii) $\Rightarrow$ (i).  Recall that 
$\si^{\mrm{R}}_M \circ \xi^{\mrm{R}}_M = \si_M$.
First suppose $M \in \catt{T}$, so that $\si_M$ is an 
isomorphism.  By Proposition \ref{prop:1165} the morphism $\xi^{\mrm{R}}_M$ is 
also an isomorphism. Hence $\si^{\mrm{R}}_M$ is an isomorphism, and  
$M \in \cat{D}(A)_{\catt{T} \tup{-tor}}$.

Now to the general case. We know that  $\catt{T}$ is a thick abelian 
subcategory of $\cat{M}(A)$, the functors $\mrm{R} \Ga_{\catt{T}}$ and 
$\opn{Id}$ have finite cohomological dimensions, and $\si^{\mrm{R}}_M$ is an 
isomorphism for all $M \in \catt{T}$. 
Therefore we can apply \cite[Proposition I.7.1]{RD} (the way-out argument; see 
also \cite[Theorem 2.10]{Ye1}) to the morphism of functors 
$\si^{\mrm{R}} : \mrm{R} \Ga_{\catt{T}} \to \opn{Id}$. The conclusion is that 
$\si^{\mrm{R}}_M$ is an isomorphism for every 
$M \in \cat{D}_{\catt{T}}(A)$, which means that every such $M$ belongs to
$\cat{D}(A)_{\catt{T} \tup{-tor}}$.
\end{proof}

\begin{thm} \label{thm:1080}
Let $A$ be a ring, and let $\catt{T}$ be a weakly stable finite 
dimensional quasi-compact torsion class in $\cat{M}(A)$. 
Then the copointed triangulated functor 
$(\mrm{R} \Ga_{\catt{T}}, \si^{\mrm{R}})$ on $\cat{D}(A)$
is idempotent, and the triangulated functor $\mrm{R} \Ga_{\catt{T}}$ is 
finite dimensional and quasi-compact. 
\end{thm}

\begin{proof}
The functor $\Ga_{\catt{T}}$ is always left exact, and the copointed additive 
functor $(\Ga_{\catt{T}}, \si)$ is always idempotent.
Looking at Definition \ref{dfn:1037}, we see that moreover the functor 
$\Ga_{\catt{T}}$ is weakly stable and has finite right cohomological dimension, 
and the functors $\mrm{R}^q \Ga_{\catt{T}}$ are quasi-compact. So this theorem 
is a special case of Theorems \ref{thm:1001} and \ref{thm:1004}. 
\end{proof}

Here are several examples of weakly stable torsion classes. 

\begin{exa} \label{exa:1315}
Let $A$ be a noetherian commutative ring, and let $\a$ be any ideal in $A$.
Then $\catt{T}_{\a}$ is a stable torsion class in $\cat{M}(A)$. One way to 
prove 
this is by the structure theory of injective $A$-modules (see the proof of 
\cite[Theorem 4.34]{PSY1}). By Proposition \ref{prop:1165}, every module 
$M \in \catt{T}_{\a}$ is $\catt{T}_{\a}$-flasque.
\end{exa}

\begin{exa} \label{exa:1316}
Let $A$ be a commutative ring, and let $\a$ be an ideal in $A$. 
In Section \ref{sec:comm-rings} we will show that the ideal $\a$ is {\em weakly 
proregular} if and only if the torsion class $\catt{T}_{\a}$ is weakly stable.
\end{exa}

\begin{exa} \label{exa:1650}
Here we present a torsion class that is {\em weakly stable 
but not stable}. Let $A$ be a commutative ring that is not noetherian, 
for instance 
\[ A = \Z[s_0, s_1, \ldots ] , \]
the polynomial ring in countably many variables. Let 
\[ B := A[t] \cong A \ot_{\Z} \Z[t] , \]
the polynomial ring in a variable $t$, and let 
$\b := (t) \sub B$. The torsion class 
$\catt{T} := \catt{T}_{\b} \sub \cat{M}(B)$
is weakly stable, by Example \ref{exa:1325} and Theorem \ref{thm:47}.

Now let us suppose, for the sake of contradiction, that $\catt{T}$ is stable. 
Consider any countable collection 
$\{ I_p \}_{p \in \N}$ of injective $A$-modules. Then
$I := \prod_p I_p$ is injective over $A$, so 
$\opn{Hom}_{A}(B, I)$ is injective over $B$. Stability of $\catt{T}$ says that 
$J := \Ga_{\b} \bigl( \opn{Hom}_{A}(B, I) \bigr)$
is also injective over $B$. Since $A \to B$ is flat, it follows that $J$ is 
injective over $A$. But as $A$-modules there is an isomorphism 
$B \cong A^{\oplus \N}$, and thus 
$\opn{Hom}_{A}(B, I) \cong I^{\times \N}$ and 
$J \cong I^{\oplus \N}$. Each $I_p$ is a direct summand of $I$, and therefore 
$\bigoplus_{p \in \N} I_p$ is a direct summand of $J$. 
The conclusion is that every countable direct sum of injective $A$-modules is 
injective. By the Bass-Papp Theorem (see \cite[Theorem 3.46(3)]{La}) this 
implies that the ring $A$ is noetherian, contradicting our assumptions. 
\end{exa}

\begin{exa} \label{exa:1065}
Let $A$ be a ring, and let $\a$ be a two-sided ideal in $A$ that is finitely 
generated as a left ideal. In Definition \ref{dfn:1065} we introduced the 
torsion class $\catt{T}_{\a}$. 
It is possible to show that $\catt{T}_{\a}$ is a stable torsion class if and 
only if the ideal $\a$ has the {\em left Artin-Rees property}, in the 
sense of \cite[Chapter 13]{GW}.

If $A$ happens to be a noetherian commutative ring, as in Example 
\ref{exa:1315}, then the original Artin-Rees property holds. This furnishes 
another proof for that case. 
\end{exa}

\begin{exa} \label{exa:1416}
Let $A$ be a ring and let $S$ be a left denominator set in it. Then the left 
Ore localization $A_S = A[S^{-1}]$ exists. 
Define 
\[ \catt{T}_S := \{ M \in \cat{M}(A) \mid A_S \ot_A M = 0 \} . \]
This is a torsion class in $\cat{M}(A)$. 
It can be shown that the torsion class $\catt{T}_S$ is weakly stable if and 
only if it has dimension $\leq 1$. This is \cite[Theorem 4.4]{Vy}. 

In the special case where $A$ is a commutative integral domain, it can be shown 
that this condition holds, and thus $\catt{T}_S$ is weakly stable in this case. 
\end{exa}

\begin{exa} \label{exa:1417}
Let $A$ be a ring, and let $\aa = (a_1, \ldots, a_n)$ be a sequence of 
elements in it. Let $\a \sub A$ be the two-sided ideal generated by $\aa$.

If the elements $a_i$ are contained in a central subring 
$C \sub A$ such that $A$ is flat over $C$, and moreover the sequence $\aa$ is 
weakly proregular in $C$ (Definition \ref{dfn:48}), then we are basically back 
in Example \ref{exa:1316}, and the torsion class $\catt{T}_{\a}$ is weakly 
stable. 

If $n = 1$ and $a_1$ is a regular normalizing element, then 
$\catt{T}_{\a}$ is weakly stable. The proof is very much like the commutative 
case, as in Section \ref{sec:comm-rings}. See \cite[Lemma 6.4]{Vy}. 

A more complicated situation is when the sequence $\aa$ is regular normalizing 
in $A$, and $n > 1$. We believe that in this case too the torsion class 
$\catt{T}_{\a}$ is weakly stable, but this has not been verified. 
\end{exa}

\begin{exa} \label{exa:1066}
Here is a graded variant, due to M. Van den Bergh in his seminal paper 
\cite{VdB}. Let $\K$ be a field, and let $A = \bigoplus_{i \geq 0} A_i$ be a 
noetherian connected graded central $\K$-ring. Recall that ``connected'' means 
that $A_0 = \K$, and each $A_i$ is a finitely generated $\K$-module. 
The augmentation ideal of $A$ is $\m := \bigoplus_{i > 0} A_i$.
We view $\K$ as an $A$-bimodule using the isomorphism 
$A / \m \cong \K$. 

Let $\cat{M}(A, \mrm{gr})$ be the abelian category of {\em graded} left 
$A$-modules, with degree preserving homomorphisms. 
Inside $\cat{M}(A, \mrm{gr})$ there is the torsion class $\catt{T}_{\m}$ of 
$\m$-torsion modules; cf.\ Definition \ref{dfn:1065}. Let 
$A^* := \opn{Hom}_{\K}(A, \K)$, the graded $\K$-linear dual of $A$.
It is $\m$-torsion, and injective in the category $\cat{M}(A, \mrm{gr})$. 
Every $\m$-torsion graded module $M$ can be embedded in a sufficiently large 
direct sum $I = \bigoplus_{x \in X} \, A^*$.
Because $A$ is noetherian, this module $I$ is injective in the category 
$\cat{M}(A, \mrm{gr})$. This implies that 
$\catt{T}_{\m} \sub \cat{M}(A, \mrm{gr})$ is a stable torsion 
class. The noetherian property also guarantees that $\catt{T}_{\m}$ is 
quasi-compact. 

Now let us consider the connected graded central $\K$-ring 
$A^{\mrm{en}} := A \ot_{\K} A^{\mrm{op}}$. 
This is often not a noetherian ring. 
Inside the abelian category $\cat{M}(A^{\mrm{en}}, \mrm{gr})$
we have the bimodule torsion class 
$\catt{T} = \catt{T}_{\m}$, 
consisting of the bimodules $M$ that are $\m$-torsion as left $A$-modules. 
There is a mirror-image bimodule torsion class 
$\catt{T}^{\mrm{op}} = \catt{T}_{\m^{\mrm{op}}} \subseteq 
\cat{M}(A^{\mrm{en}}, \mrm{gr})$,
consisting of the bimodules $M$ that are $\m^{\mrm{op}}$-torsion, 
or in other words, that are $\m$-torsion as right modules.
Van den Bergh proved (not in this terminology of course) that 
the torsion classes 
$\catt{T}, \catt{T}^{\mrm{op}} \subseteq \cat{M}(A^{\mrm{en}}, \mrm{gr})$
are weakly stable and quasi-compact.
Furthermore, he proved that if 
$\catt{T} \sub \cat{M}(A, \mrm{gr})$ is finite dimensional, 
then so is $\catt{T} \sub \cat{M}(A^{\mrm{en}}, \mrm{gr})$. 
The same is true for the torsion class $\catt{T}^{\mrm{op}}$. 

One of the main technical result in \cite{VdB} is this (in our terminology): 
Assume that the torsion classes $\catt{T}$ and $\catt{T}^{\mrm{op}}$ are both 
finite dimensional. If $M \in \cat{D}(A^{\mrm{en}}, \mrm{gr})$
satisfies 
\[ \mrm{R} \Ga_{\catt{T}}(M) \, , \, \mrm{R} \Ga_{\catt{T}^{\mrm{op}}}(M)
\, \in \, \catt{T} \cap \catt{T}^{\mrm{op}} , \]
then 
$\mrm{R} \Ga_{\catt{T}}(M) \cong \mrm{R} \Ga_{\catt{T}^{\mrm{op}}}(M)$
in $\cat{D}(A^{\mrm{en}}, \mrm{gr})$. 
This is the prototype for our Theorem \ref{thm:1379} on {\em symmetric derived 
torsion}. 
\end{exa}

%\cleardoublepage
\section{Weakly Proregular Ideals in Commutative Rings} \label{sec:comm-rings}

In this section we compare the noncommutative torsion picture from Section 
\ref{sec:tors-cls} to the commutative picture that was studied in
\cite{LC}, \cite{SGA2}, \cite{Ma}, \cite{GM}, \cite{AJL}, \cite{Sn}, \cite{KS3} 
and \cite{PSY1}. This comparison will provide motivation for the subsequent 
sections of our paper. 

Throughout this section $A$ is a nonzero commutative ring. 
We are going to follow the notation and definitions of \cite{PSY1}.
Moreover, we shall quote the necessary results from \cite{PSY1}; this is done 
for the sake of convenience, even though some of the results have been proved 
(wholly or partially) in earlier tests. Note that some of the notation of 
\cite{PSY1} differs from that of earlier texts, such as \cite{LC} and \cite{Bo}. 
 
Let $\aa = (a_1, \ldots, a_n)$ be a sequence of elements of $A$. 
To this sequence we associate the {\em Koszul complex} 
$\opn{K}(A; \aa)$, which is a bounded complex of free $A$-modules,
concentrated in degrees $-n, \ldots, 0$.
The Koszul complex of a single element $a \in A$ is 
\[ \opn{K}(A; a) := \bigl( \cdots \to 0 \to A \xar{\cd a} A \to 0 \to 
\cdots \bigr) , \]
concentrated in degrees $-1, 0$. The Koszul complex of the sequence $\aa$ is
\[ \opn{K}(A; \aa) := \opn{K}(A; a_1) \ot_{A} \cdots  \ot_{A}
\opn{K}(A; a_n) . \]
Thus $\opn{K}(A; \aa)^0 = A$, and $\opn{K}(A; \aa)^{-1}$ is a free $A$-module 
of rank $n$. 

Note that $\opn{K}(A; \aa)$ is actually a commutative DG (differential graded) 
ring,  and there is a canonical ring isomorphism 
\[ \opn{H}^0(\opn{K}(A; \aa)) \cong A / \a , \]
where $\a \subseteq A$ is the ideal generated by the sequence $\aa$. 

For each $i \geq 1$ let $\aa^i$ be the sequence 
$(a_1^i, \ldots, a_n^i)$. There is a corresponding Koszul complex 
$\opn{K}(A; \aa^i)$. For $j \geq i$ there is a DG ring homomorphism 
\[ \opn{K}(A; \aa^j) \to \opn{K}(A; \aa^i) , \]
that in degree $0$ is the identity of $A$, and in degree $-1$ is multiplication 
by the sequence $\aa^{j - i}$. In this way the collection of complexes of 
$A$-modules (or of DG rings) 
$\{ \opn{K}(A; \aa^i) \}_{i \geq 1}$
becomes an inverse system.  

Recall that an inverse system of modules 
$\{ M_i \}_{i \geq 1}$ is called {\em pro-zero} if for every $i$ 
there is some $j \geq i$ such that the homomorphism 
$M_j \to M_i$ is zero. In the literature, this condition is occasionally 
called the {\em trivial Mittag-Leffler condition}.

\begin{dfn} \label{dfn:48}
A finite sequence $\aa = (a_1, \ldots, a_n)$ in $A$ is called {\em weakly 
proregular} if for every $p < 0$ the inverse system of $A$-modules 
\[ \bigl\{ \opn{H}^p(\opn{K}(A; \aa^i)) \bigr\}_{i \geq 1} \]
is pro-zero.
\end{dfn}

\begin{dfn} \label{dfn:49}
An ideal $\a \subseteq A$ is called {\em weakly proregular}
if it is generated by some weakly proregular sequence $\aa$. 
\end{dfn}

Definition \ref{dfn:48} was first considered by Grothendieck in 
\cite{SGA2} and \cite{LC}. The name ``weakly proregular'' was given in 
\cite[Correction]{AJL}. 

The {\em dual Koszul complex} of $\aa$ is the complex 
\[ \opn{K}^{\vee}(A; \aa) := 
\opn{Hom}_A \bigl( \opn{K}(A; \aa), A \bigr) . \]
The  collection of complexes
$\{ \opn{K}^{\vee}(A; \aa^i) \}_{i \geq 1}$
is a direct system, and in the limit we obtain the 
{\em infinite dual Koszul complex}
\[ \opn{K}^{\vee}_{\infty}(A; \aa) := 
\lim_{i \to} \opn{K}^{\vee}(A; \aa^i) . \]
It is a complex of flat $A$-modules, concentrated in degrees $0, \ldots, n$. 
The infinite dual Koszul complex of a single element $a \in A$ looks like this:
\begin{equation} \label{eqn:1325}
\opn{K}^{\vee}_{\infty}(A; a) \cong  
\bigl( \cdots \to 0 \to A \to A[a^{-1}] \to 0 \to \cdots \bigr),
\end{equation}
where $A$ sits in degree $0$, and the differential 
$A \to A[a^{-1}]$ is the ring homomorphism.
For the sequence $\aa$ we have an isomorphism of complexes
\begin{equation} \label{eqn:1326}
\opn{K}^{\vee}_{\infty}(A; \aa) \cong 
\opn{K}^{\vee}_{\infty}(A; a_1) \ot_A \cdots \ot_A 
\opn{K}^{\vee}_{\infty}(A; a_n) .
\end{equation}
Let $\a$ be the ideal generated by the sequence $\aa$. 
From formulas (\ref{eqn:1325}) and (\ref{eqn:1326}) it is clear that 
for every $A$-module $M$ there is a canonical isomorphism 
\begin{equation} \label{eqn:1335}
\opn{H}^0 \bigl( \opn{K}^{\vee}_{\infty}(A; \aa) \ot_A M \bigr) \cong 
\Ga_{\a}(M) .
\end{equation}

\begin{thm} \label{thm:45}
Let $\aa$ be a finite sequence in $A$. The following two conditions are 
equivalent\tup{:} 
\begin{enumerate}
\rmitem{i} The sequence $\aa$ is weakly proregular. 

\rmitem{ii} For every injective $A$-module $I$ and every positive integer $p$, 
the module 
\[ \opn{H}^p \bigl (\opn{K}^{\vee}_{\infty}(A; \aa) \ot_A I \bigr) \]
is zero. 
\end{enumerate}
\end{thm}

\begin{thm} \label{thm:51}
If $A$ is noetherian, then every finite sequence $\aa$ in $A$ is weakly 
proregular. 
\end{thm}

Theorems \ref{thm:45} and \ref{thm:51} had first appeared in \cite{SGA2} 
and \cite{LC}, stated in different terminology of course. 
See \cite[Theorem 3.24]{PSY1} and \cite[Corollary 5.4]{PSY1},
that are stated in this language, and have full proofs. 

\begin{thm}[{\cite[Corollary 5.4]{PSY1}}] \label{thm:48}
Let $\aa$ and $\bsym{b}$ be finite sequences in $A$, and let $\a$ and $\b$ 
be the respective ideals that they generate. If the sequence $\aa$ 
is weakly proregular, and if there is equality $\sqrt{\a} = \sqrt{\b}$, then 
the sequence $\bsym{b}$ is  weakly proregular.
\end{thm}

A weaker version of Theorem \ref{thm:48} had previously appeared as 
\cite[Corollary 3.1.5]{AJL}.

Suppose $\a$ is a finitely generated ideal in $A$. 
The $\a$-torsion functor $\Ga_{\a}$ was discussed in Section 
\ref{sec:tors-cls}. 
It is part of an idempotent copointed additive functor 
$(\Ga_{\a}, \si)$ on $\cat{M}(A)$. As explained in Section 
\ref{sec:idem-cop-func}, there is a corresponding copointed triangulated 
functor $(\mrm{R} \Ga_{\a}, \si^{\mrm{R}})$ on $\cat{D}(A)$. 
A complex $M \in \cat{D}(A)$ is called a {\em derived $\a$-torsion complex}
if $\si^{\mrm{R}}_M : \mrm{R} \Ga_{\a}(M) \to M$ is an isomorphism. 
(In \cite{PSY1} the name for this property was {\em cohomologically 
$\a$-torsion}.) The full subcategory of $\cat{D}(A)$ on the derived 
$\a$-torsion complexes is denoted by 
$\cat{D}(A)_{\a\tup{-tor}}$. It is a triangulated category. 

The $\a$-adic completion is the additive functor 
$\La_{\a} : \cat{M}(A) \to \cat{M}(A)$
defined by 
\[ \La_{\a}(M) := \lim_{\leftarrow i} \, (M / \a^i \cd M) . \]
There is a morphism of functors 
$\tau : \opn{Id} \to \La_{\a}$, and the pair 
$(\La_{\a}, \tau)$ is an idempotent pointed additive functor on $\cat{M}(A)$.
There is a corresponding pointed triangulated 
functor $(\mrm{L} \La_{\a}, \tau^{\mrm{L}})$ on $\cat{D}(A)$. 
A complex $M \in \cat{D}(A)$ is called a
{\em derived $\a$-adically complete complex}
if $\tau^{\mrm{L}}_M : M \to \mrm{L} \La_{\a}(M)$ is an isomorphism. 
(In \cite{PSY1} the name for this property was {\em cohomologically 
$\a$-adically complete}.)  The full subcategory of $\cat{D}(A)$ on the 
derived $\a$-adically complete complexes is denoted by 
$\cat{D}(A)_{\a\tup{-com}}$. It is a triangulated category. 

It turns out that if $\a$ is weakly proregular, then the triangulated 
(co)pointed functors 
$(\mrm{R} \Ga_{\a}, \si^{\mrm{R}})$ and 
$(\mrm{L} \La_{\a}, \tau^{\mrm{L}})$ are idempotent. 
This is partly proved in \cite[Corollary 4.30 and Proposition 7.10]{PSY1} --
these results only state that 
$\si^{\mrm{R}}_{\mrm{R} \Ga_{\a}(M)}$ and 
$\tau^{\mrm{L}}_{\mrm{L} \La_{\a}(M)}$
are isomorphisms. But the full idempotence is an immediate consequence of
\cite[Lemma 7.9]{PSY1}. Moreover, the MGM Equivalence (Theorem 
\ref{thm:1375} in the Introduction) is proved in \cite{PSY1}. 

If $\aa$ is a finite sequence of elements that generates $\a$ (so it is 
a weakly proregular sequence by Theorem \ref{thm:48}), 
then there is a better representative for 
$\mrm{R} \Ga_{\a}(A)$ in $\cat{D}(A)$ than the infinite Koszul complex
$\opn{K}^{\vee}_{\infty}(A; \aa)$  -- it is the {\em telescope complex}
$\opn{Tel}(A; \aa)$. This is a bounded complex of countable rank free $A$ 
modules, so in particular it is K-projective. The derived functors take these 
nice forms:
\[ \mrm{R} \Ga_{\a}(M) \cong \opn{Tel}(A; \aa) \ot_A M \]
and
\[ \mrm{L} \La_{\a}(M) \cong \opn{Hom}_A \bigl( \opn{Tel}(A; \aa), M \bigr) . \]
In the terminology of the Section \ref{sec:cop-obj}, 
$\opn{Tel}(A; \aa)$ is an idempotent copointed object in the monoidal category 
$\cat{D}(A)$. 
 
\begin{rem} \label{rem:2035}
We should mention that the notion of telescope complex is familiar in algebraic 
topology, usually as an abstract homotopy colimit; cf.\ \cite{GM}. The 
explicit construction of the telescope complex $\opn{Tel}(A; \aa)$ 
can be found in \cite{Sn} and \cite{PSY1}. 
\end{rem}

Here are several examples of weakly proregular ideals. 

\begin{exa} \label{exa:1325}
Let $A$ be a commutative ring. If $\aa =  (a_1, \ldots, a_n)$ is a regular 
sequence in $A$, then it it weakly proregular.
\end{exa}
  
\begin{exa} \label{exa:1327}
Let $A$ and $B$ be the rings from Example \ref{exa:1650}. So $A$ is a 
non-noetherian commutative ring, $B = A[t]$ and
$\b = (t) \sub B$.  
Then $B$ is not noetherian, but the ideal $\b \sub B$ is weakly proregular, by 
Example \ref{exa:1325}. 
\end{exa}

\begin{exa} \label{exa:1326}
Let $\K$ be a field of characteristic $0$, and let $\K[[t_1]]$ and $\K[[t_2]]$ 
be the power series rings. Let $A$ be the ring
\[ A := \K[[t_1]] \ot_{\K} \K[[t_2]] , \]
and let $\a \subseteq A$ be the ideal generated by $t_1$ and $t_2$. 
As shown in \cite[Theorem 0.9]{Ye5} the ring $A$ is not noetherian. 
By Example \ref{exa:1325} the ideal $\a$ is weakly proregular.
What is remarkable in this example is that the $\a$-adic completion 
$\what{A}$ is noetherian, and it is not flat over $A$. 
\end{exa}

Now for the main result in this section.
Recall that a finitely generated ideal $\a \subseteq A$ gives rise to a 
torsion class $\catt{T}_{\a} \subseteq \cat{M}(A)$; see Definition 
\ref{dfn:1065}.
If $\b \subseteq A$ is another finitely generated ideal, and 
$\catt{T}_{\a} = \catt{T}_{\b}$, then $\sqrt{\a} = \sqrt{\b}$. 
Theorem \ref{thm:48} tells us that weak proregularity is a property of the 
torsion class $\catt{T}_{\a}$. The next theorem goes one step more: it lets us 
characterize weak proregularity in terms of a ``noncommutative'' property of 
$\catt{T}_{\a}$.

\begin{thm} \label{thm:47}
Let $A$ be a commutative ring, let $\aa$ be a finite sequence of elements of 
$A$, and let $\a$ be the ideal generated by $\aa$. The following two conditions 
are equivalent\tup{:} 
\begin{enumerate}
\rmitem{i} The sequence $\aa$ is weakly proregular. 

\rmitem{ii} The torsion class $\catt{T}_{\a}$ is weakly stable.
\end{enumerate}
\end{thm}

We need a lemma first. 

\begin{lem} \label{lem:1335}
Let $M$ be an $A$-module, and consider the complex 
\[ K := \opn{K}^{\vee}_{\infty}(A; \aa) \ot_A M . \]
\begin{enumerate}
\item For every $p \geq 1$ we have 
$\Ga_{\a}(K^p) = 0$, and the module $K^p$ is  right $\Ga_{\a}$-acyclic.

\item For every $p$ the module $\opn{H}^p(K)$ is $\a$-torsion. 
\end{enumerate}
\end{lem}

\begin{proof}
These assertions can be found inside the proof of 
\cite[Corollary 3.1.5]{AJL}. For the benefit of the reader, here is a full 
stand-alone proof. 

\medskip \noindent
(1) Let us write the sequence in full: $\aa =  (a_1, \ldots, a_n)$. 
Fix $p \geq 1$. By formulas (\ref{eqn:1325}) and (\ref{eqn:1326}) we know 
that there is an isomorphism 
$K^p \cong \bigoplus_{j = 1}^n K^p_j$, 
where $K^p_j = A[a_j^{-1}] \ot_A L_j$
for certain $A$-modules $L_j$. For every $q$ the action of the element $a_j$ on 
the module $\mrm{R}^q \Ga_{\a}(K^p_j)$
is bijective; yet the action of $a_j$ on any finitely generated submodule
of $\mrm{R}^q \Ga_{\a}(K^p_j)$ is 
nilpotent. This implies that $\mrm{R}^q \Ga_{\a}(K^p_j) = 0$ for all $q$.
We see that $\Ga_{\a}(K^p_j) = 0$, and that  $K^p_j$ is  right 
$\Ga_{\a}$-acyclic. Hence the same is true for $K^p$. 

\medskip \noindent 
(2) By formulas (\ref{eqn:1325}) and (\ref{eqn:1326}), for every $j$ the complex
\[ A[a_j^{-1}] \ot_A \opn{K}^{\vee}_{\infty}(A; \aa) \]
is contractible. Therefore for every $p$
\[  A[a_j^{-1}] \ot_A \opn{H}^p(K) \cong 
\opn{H}^p \bigl( A[a_j^{-1}] \ot_A K \bigr) = 0 . \]
In view of equation (\ref{eqn:1335}), this shows that the module
$\opn{H}^p(K)$ is $a_j$-torsion for all $j$. Hence it is $\a$-torsion. 
\end{proof}

\begin{proof}[Proof of Theorem \tup{\ref{thm:47}}]
(i) $\Rightarrow$ (ii): By \cite[Corollary 4.30]{PSY1}, for every 
$M \in \cat{D}(A)$ the morphism 
\begin{equation} \label{eqn:1328}
\si^{\mrm{R}}_{\mrm{R} \Ga_{\a}(M)} : 
\mrm{R} \Ga_{\a} (\mrm{R} \Ga_{\a} (M)) \to \mrm{R} \Ga_{\a} (M)
\end{equation}
in $\cat{D}(A)$ is an isomorphism. Take any injective $A$-module $I$. Then 
$\mrm{R} \Ga_{\a} (I) \cong \Ga_{\a} (I)$, and formula (\ref{eqn:1328}) says 
that 
$\mrm{R} \Ga_{\a} (\Ga_{\a} (I)) \cong \Ga_{\a}(I)$
in $\cat{D}(A)$. 
Therefore 
\[ \mrm{R}^q \Ga_{\a} (\Ga_{\a} (I)) \cong 
\opn{H}^q(\mrm{R} \Ga_{\a} (\Ga_{\a} (I)) \cong 
\opn{H}^q(\Ga_{\a} (I)) , \]
which is zero for all $q > 0$. 
We see that the module $\Ga_{\a} (I)$ is right $\Ga_{\a}$-acyclic 
(Definition \ref{dfn:1055}). We conclude that the 
functor $\Ga_{\a}$ is weakly stable (Definition \ref{dfn:1225}), and that the 
torsion class $\catt{T}_{\a}$ is weakly stable (Definition \ref{dfn:1037}). 

\medskip \noindent
(ii) $\Rightarrow$ (i): 
According to Theorem \ref{thm:45}, it suffices to show 
that for every injective $A$-module $I$ and every positive integer $p$, the 
module
$\opn{H}^p \bigl (\opn{K}^{\vee}_{\infty}(A; \aa) \ot_A I \bigr)$
is zero. 

The proof is by contradiction. Assume, for the sake of contradiction, 
that for some injective $A$-module $I$, and for some $p \geq 1$, 
we have 
\[ \opn{H}^p \bigl (\opn{K}^{\vee}_{\infty}(A; \aa) \ot_A I \bigr) \neq 0 . \]

Let us write 
\[ K := \opn{K}^{\vee}_{\infty}(A; \aa) \ot_A I . \]
We may assume that $p$ is minimal:
$\opn{H}^p (K) \neq 0$, but $\opn{H}^q (K) = 0$ for all $1 \leq q < p$. 
Thus we have an exact sequence 
\begin{equation} \label{eqn:55}
0 \to \Ga_{\a}(I) \to K^0 \to \cdots \to K^{p - 1} \to K^p .
\end{equation}
Since the module $K^0 \cong I$ is injective, it is right $\Ga_{\a}$-acyclic. 
For $1 \leq q \leq p$ the module $K^q$ is right $\Ga_{\a}$-acyclic
by Lemma \ref{lem:1335}(1). 

Let us now write $J^q := K^q$ for $q \leq p$, and 
continue the exact sequence (\ref{eqn:55}), with its new notation, to an exact 
sequence 
\begin{equation} \label{eqn:56}
0 \to \Ga_{\a}(I) \to J^0 \to \cdots \to J^{p - 1} \to J^p \to 
J^{p + 1} \to \cdots
\end{equation}
in which the modules $J^q$, for $q > p$, are also right $\Ga_{\a}$-acyclic 
(e.g.\ we can take injective $A$-modules). 
So we get a complex $J$ of right $\Ga_{\a}$-acyclic modules that is 
concentrated in nonnegative degrees, and a quasi-isomorphism 
$\Ga_{\a}(I) \to J$. By Lemma \ref{lem:1057} we know that $J$ is a right
$\Ga_{\a}$-acyclic complex, and thus 
\begin{equation} \label{eqn:1334}
\mrm{R}^q \Ga_{\a}(\Ga_{\a}(I)) \cong \opn{H}^q(\Ga_{\a}(J))  
\end{equation}
for all $q$. 

There is an embedding
\[ \opn{H}^{p}(K) \subseteq \opn{Coker}(K^{p - 1} \to K^p)
= \opn{Coker}(J^{p - 1} \to J^p) . \]
On the other hand, due to the exactness of (\ref{eqn:56}), we have 
\[ \opn{Coker}(J^{p - 1} \to J^p) \cong 
\opn{Ker}(J^{p + 1} \to J^{p + 2}) . \]
By Lemma \ref{lem:1335}(2) the module $\opn{H}^{p}(K)$ is $\a$-torsion, 
and therefore we get an embedding 
\[ \opn{H}^{p}(K) \subseteq
\Ga_{\a} \bigl( \opn{Ker}(J^{p + 1} \to J^{p + 2}) \bigr) . \]
But 
\[ \Ga_{\a} \bigl( \opn{Ker}(J^{p + 1} \to J^{p + 2}) \bigr) \cong 
\opn{Ker}\bigl( \Ga_{\a}(J^{p + 1}) \to \Ga_{\a}(J^{p + 2}) \bigr) . \]
Since $ \opn{H}^{p}(K) \neq 0$, we conclude that 
\begin{equation} \label{eqn:1333}
\opn{Ker}\bigl( \Ga_{\a}(J^{p + 1}) \to \Ga_{\a}(J^{p + 2}) \bigr) \neq 0 . 
\end{equation}

The module $\Ga_{\a}(I)$ is right $\Ga_{\a}$-acyclic; this is because $I$ 
is injective and $\catt{T}_{\a}$ is weakly stable.
From equation (\ref{eqn:1334}) we conclude that
\begin{equation} \label{eqn:1337}
\opn{H}^q(\Ga_{\a}(J)) = 0 \ \tup{ for all } q > 0 .
\end{equation}
Also $\Ga_{\a}(\Ga_{\a}(I)) = \Ga_{\a}(I)$. Therefore the sequence
\begin{equation} \label{eqn:57}
0 \to \Ga_{\a}(I) \to \Ga_{\a}(J^0) \to \Ga_{\a}(J^1) \to \cdots 
\to \Ga_{\a}(J^p) \to \Ga_{\a}(J^{p + 1}) \to \cdots ,
\end{equation}
gotten from (\ref{eqn:56}) by applying the functor $\Ga_{\a}$, is exact.
 
Finally, by Lemma \ref{lem:1335}(1) we know that 
$\Ga_{\a}(J^q) = 0$ for all $1 \leq q \leq p$. So the sequence 
(\ref{eqn:57}) looks like this: 
\[ 0 \to \Ga_{\a}(I) \to \Ga_{\a}(J^0) \to 0 \to \cdots 
\to 0 \to \Ga_{\a}(J^{p + 1}) \to \Ga_{\a}(J^{p + 2}) \to  \cdots , \]
with at least one $0$ occurring between $\Ga_{\a}(J^0)$
and $\Ga_{\a}(J^{p + 1})$. 
Combining this with equation (\ref{eqn:1333}) we conclude that 
$\opn{H}^{p + 1}(\Ga_{\a}(J)) \neq 0$. This contradicts (\ref{eqn:1337}).
\end{proof}
 
\begin{cor} \label{cor:1320}
Let $A$ be a commutative ring, and let $\a$ be a finitely generated ideal of 
$A$. If the torsion class $\catt{T}_{\a}$ is weakly stable, then it is also 
quasi-compact and finite dimensional.
\end{cor}

\begin{proof}
Let $\aa$ be a finite sequence, say of length $n$, that generates the ideal 
$\a$. By the theorem we know that the sequence $\aa$ is weakly proregular. 
According to \cite[Corollary 3.26]{PSY1} there is an isomorphism of 
triangulated 
functors 
\[ \mrm{R} \Ga_{\a} \cong \opn{K}^{\vee}_{\infty}(A; \aa) \ot_A (-) . \]
This shows that the functor $\mrm{R} \Ga_{\a}$ is quasi-compact, and that its 
cohomological dimension is at most $n$. Since the triangulated functor 
$\mrm{R} \Ga_{\a}$ is quasi-compact, then so are the functors
$\mrm{R}^q \Ga_{\a}$. And $\mrm{R}^q \Ga_{\a} = 0$ for $q > n$.
\end{proof}

\begin{rem} \label{rem:1116}
Corollary \ref{cor:1320} could be false for a torsion class $\catt{T}$ that is 
not defined by a finitely generated ideal. 
Indeed, in \cite[Example 4.12]{Vy} there is a weakly stable torsion class 
which is not quasi-compact (but is finite dimensional).
\end{rem}

%\cleardoublepage
\section{Derived Categories of Bimodules}
\label{sec:bimods}

In this section we explain how to form derived tensor and Hom functors 
between derived categories of bimodules over noncommutative rings, and study 
some of their properties. We do this in a somewhat narrow setting: we work over 
a commutative base ring $\K$, and some of the rings are assumed to be flat over 
it. See Remark \ref{rem:1420} regarding the use of DG rings to remove the 
flatness assumption. 

From now on in the paper we adopt the following convention. 

\begin{conv} \label{conv:1425}
There is a fixed nonzero commutative base ring $\K$. 
All additive operations and categories are by default 
$\K$-linear. All rings and bimodules are central over $\K$. We use the 
abbreviation $\ot$ for $\ot_{\K}$. By default all modules are left 
modules. We say that $A$ is a {\em flat ring} if $A$ is flat as a $\K$-module.
\end{conv}

\begin{exa} \label{exa:1115}
If $\K$ is a field, then every central $\K$-ring $A$ is flat.
If $\K = \Z$, then any ring $A$ is $\K$-central, and $A$ is flat if and only if 
it is torsion free (as an abelian group). 
\end{exa}

Let $A$ and $B$ be rings. Recall that the {\em opposite} ring of $A$ is the 
ring $A^{\mrm{op}}$, that has the same underlying $\K$-module structure as $A$, 
but the multiplication is 
\[ a_1 \cd^{\mrm{op}} \, a_2 := a_2 \cd a_1 \]
for $a_1, a_2 \in A$.
The tensor product $A \ot B$ is a ring with the usual multiplication 
\[ (a_1 \ot b_1) \cd (a_2 \ot b_2) :=
(a_1 \cd a_2) \ot (b_1 \cd b_2) \]
for $a_k \in A$ and $b_k \in B$.  The {\em enveloping} ring of $A$ is the ring 
$A^{\mrm{en}} := A \ot A^{\mrm{op}}$. 
We identify right $B$-modules with left $B^{\mrm{op}}$-modules, and 
$A$-$B$-bimodules with left modules over $A \ot B^{\mrm{op}}$. 

Given $M \in \cat{M}(A^{\mrm{op}})$ and 
$N \in \cat{M}(A)$, the tensor product $M \ot_A N$ is defined; and it is a 
$\K$-module. Note that $(A^{\mrm{op}})^{\mrm{op}} = A$. It is a bit 
confusing, but actually easy to check, that there is a canonical isomorphism
\begin{equation} \label{eqn:1236} 
N \ot_{A^{\mrm{op}}} M \cong M \ot_A N
\end{equation}
in $\cat{M}(\K)$, as quotients of $N \ot M \cong M \ot N$. 

A homomorphism of rings $f : A \to B$ induces a forgetful functor
\[ \opn{Rest}_f : \cat{M}(B) \to \cat{M}(A) \]
that we call {\em restriction}, which is exact. The restriction functor extends 
to complexes, and to a triangulated functor  
\begin{equation} \label{eqn:1030}
\opn{Rest}_f : \cat{D}(B) \to \cat{D}(A) .
\end{equation}

\begin{prop} \label{prop:1071}
Let $f : A \to B$ be a ring homomorphism.
The functor $\opn{Rest}_f$ is conservative. Namely a morphism 
$\phi : M \to N$ in $\cat{D}(B)$ is an isomorphism if and only if the morphism 
\[ \opn{Rest}_f(\phi) : \opn{Rest}_f(M) \to \opn{Rest}_f(N) \]
in  $\cat{D}(A)$ is an isomorphism.
\end{prop}

\begin{proof}
The morphism $\phi$ is an isomorphism in $\cat{D}(B)$ if and only if 
the $\K$-module homomorphisms 
$\opn{H}^p(\phi) : \opn{H}^p(M) \to \opn{H}^p(N)$ are isomorphisms for all $p$. 
But $\opn{H}^p(\phi) = \opn{H}^p(\opn{Rest}_f(\phi))$ 
as $\K$-module homomorphisms.
\end{proof}

It is well-known (see \cite{Sp}, \cite{BN}, \cite{Ke}, \cite[Chapter 
09JD]{SP} or \cite[Section 11]{Ye6}) 
that every complex $M \in \cat{C}(A)$ admits a quasi-isomorphism
$P \to M$, where $P$ is a K-projective complex, each $P^i$ is a projective 
$A$-module, and $\opn{sup}(P) =  \opn{sup}(\opn{H}(M))$.
If $P$ is a K-projective complex, then it is K-flat. 

Given rings $A$ and $B$, there is a canonical ring homomorphism
$A \to A \ot B$, that sends $a \mapsto a \ot 1_B$. The corresponding forgetful 
functor is denoted by 
\begin{equation} \label{eqn:1031}
\opn{Rest}_A : \cat{D}(A \ot B) \to \cat{D}(A) . 
\end{equation}

\begin{dfn} \label{dfn:1130}
Let $A$ and $B$ be rings, and let 
$M \in \cat{C}(A \ot B)$. If $\opn{Rest}_A(M) \in \cat{C}(A)$ 
is K-flat (resp.\  K-injective, resp.\ K-projective), then we say that $M$ is 
{\em K-flat} (resp.\ {\em K-injective}, resp.\ {\em K-projective})  
{\em over $A$}.
\end{dfn}

\begin{lem} \label{lem:1070}
Let $A$ and $B$ be rings, and assume $B$ is flat.
\begin{enumerate}
\item If $P \in \cat{C}(A \ot B)$ is K-flat, then $P$ is K-flat over $A$.

\item  If $I \in \cat{C}(A \ot B)$ is K-injective, then $I$ is K-injective over 
$A$.
\end{enumerate}
\end{lem}

\begin{proof}
These are direct consequences of the identities 
\[ M \ot_{A} P \cong (M \ot B) \ot_{A \ot B} P \]
and
\[ \opn{Hom}_A(N, I) \cong  \opn{Hom}_{A \ot B}(N \ot B, I)  \]
for $M \in \cat{M}(A^{\mrm{op}})$ and 
$N \in \cat{M}(A)$. 
\end{proof}

\begin{rem} \label{rem:1130}
If $B$ happens to be a projective module over $\K$, then the restriction to $A$ 
of a K-projective complex in $\cat{C}(A \ot B)$ is K-projective in 
$\cat{C}(A)$. The proof is similar. We will not need this fact here. 
\end{rem}

Left and right derived bifunctors are studied in detail in 
\cite[Section 9]{Ye6}. 

\begin{prop} \label{prop:1042}
Let $A$, $B$ and $C$ be rings, and assume $C$ is flat.
\begin{enumerate}
\item The bifunctor 
\[ (- \ot_{B} -) :  \cat{M}(A \ot B^{\mrm{op}}) \times 
\cat{M}(B \ot C^{\mrm{op}}) \to \cat{M}(A \ot C^{\mrm{op}})  \]
has a left derived bifunctor 
\[ (- \ot^{\mrm{L}}_{B} -) :  \cat{D}(A \ot B^{\mrm{op}}) \times 
\cat{D}(B \ot C^{\mrm{op}}) \to \cat{D}(A \ot C^{\mrm{op}}) . \]

\item Given $M \in \cat{C}(A \ot B^{\mrm{op}})$
and 
$N \in \cat{C}(B \ot C^{\mrm{op}})$,
such that $M$ is K-flat over $B^{\mrm{op}}$ or $N$ is K-flat over $B$, the 
morphism 
\[ \xi^{\mrm{L}}_{M, N} : M \ot_B^{\mrm{L}} N \to M \ot_{B} N \]
in $\cat{D}(A \ot C^{\mrm{op}})$ is an isomorphism. 

\item Suppose we are given ring homomorphisms 
$A' \to A$ and $C' \to C$, such that $C'$ is flat. Then the diagram 
\[ \UseTips \xymatrix @C=8ex @R=6ex {
\cat{D}(A \ot B^{\mrm{op}}) \times \cat{D}(B \ot C^{\mrm{op}})
\ar[r]^(0.6){ (- \ot^{\mrm{L}}_{B} -) }
\ar[d]_{\opn{Rest} \times \opn{Rest}}
&
\cat{D}(A \ot C^{\mrm{op}})
\ar[d]^{\opn{Rest}}
\\
\cat{D}(A' \ot B^{\mrm{op}}) \times \cat{D}(B \ot C'^{\, \mrm{op}})
\ar[r]^(0.6){ (- \ot^{\mrm{L}}_{B} -) }
&
\cat{D}(A' \ot C'^{\, \mrm{op}})
} \]
is commutative up to an isomorphism of triangulated bifunctors. 

\item Suppose $D$ is another flat ring. Then there is an isomorphism  
\[ (- \ot^{\mrm{L}}_{B} -) \ot^{\mrm{L}}_{C} (-) \cong 
(-) \ot^{\mrm{L}}_{B} (- \ot^{\mrm{L}}_{C} -) \]
of triangulated trifunctors 
\[ \cat{D}(A \ot B^{\mrm{op}}) \times \cat{D}(B \ot C^{\mrm{op}}) \times 
\cat{D}(C \ot D^{\mrm{op}}) \to \cat{D}(A \ot D^{\mrm{op}}) . \]
\end{enumerate}
\end{prop}

Note that item (3) includes the cases $A' = \K$ and $C' = \K$. 

\begin{proof}
(1) The bifunctor $(- \ot_{B} -)$ induces a triangulated bifunctor 
\[ (- \ot_{B} -) :  \cat{K}(A \ot B^{\mrm{op}}) \times 
\cat{K}(B \ot C^{\mrm{op}}) \to \cat{K}(A \ot C^{\mrm{op}})  \]
on the homotopy categories, in the obvious way. 
By Lemma \ref{lem:1070}(1), every complex 
$N \in \cat{C}(B \ot C^{\mrm{op}})$
admits a quasi-isomorphism 
$\ze_N : \til{N} \to N$, where 
$\til{N} \in \cat{C}(B \ot C^{\mrm{op}})$
is K-flat over $B$ (e.g.\ we can take $\til{N}$ to be K-projective in 
$\cat{C}(B \ot C^{\mrm{op}})$). Let us define the object
\[ M \ot^{\mrm{L}}_{B} N := M \ot_{B} \til{N} \in
\cat{D}(A \ot C^{\mrm{op}}) , \]
with the morphism 
\[ \xi^{\mrm{L}}_{M, N} := \opn{id}_M \ot_B \, \ze_N : 
M \ot_B^{\mrm{L}} N \to M \ot_{B} N . \]
The pair 
$\bigl( (- \ot^{\mrm{L}}_{B} -), \xi^{\mrm{L}} \bigr)$
is a left derived bifunctor of $(- \ot_{B} -)$. 

\medskip \noindent 
(2) Under either assumption the homomorphism 
$\opn{id}_M \ot_B \, \ze_N$
is a quasi-isomorphism.

\medskip \noindent 
(3) The resolutions $\ze_N : \til{N} \to N$ from item (1) become 
resolutions 
\[ \opn{Rest}(\ze_N) : \opn{Rest}(\til{N}) \to \opn{Rest}(N) \]
in $\cat{C}(B \ot C'^{\, \mrm{op}})$
that are K-flat over $B$. And for $M \in \cat{C}(A \ot B^{\mrm{op}})$ there is 
an obvious isomorphism 
\[ \opn{Rest}(M) \ot_B \opn{Rest}(\til{N}) \cong \opn{Rest}(M \ot_B \til{N}) \]
in $\cat{C}(A' \ot C'^{\, \mrm{op}})$.
 
\medskip \noindent 
(4) Given $M, N, \til{N}$ as above and 
$P \in \cat{C}(C \ot D^{\mrm{op}})$, 
we choose a quasi-isomorphism 
$\ze_P : \til{P} \to P$, where 
$\til{P} \in \cat{C}(C \ot D^{\mrm{op}})$
is K-flat over $C$. A small calculation shows that 
$\til{N} \ot_C \til{P}$ is K-flat over $B$. The desired isomorphism 
\[ (M \ot^{\mrm{L}}_{B} N) \ot^{\mrm{L}}_{C} P \cong  
M \ot^{\mrm{L}}_{B} (N \ot^{\mrm{L}}_{C} P) \]
in $\cat{D}(A \ot D^{\mrm{op}})$
comes from the obvious isomorphism 
\[ (M \ot_{B} \til{N}) \ot_{C} \til{P} \cong  
M \ot_{B} (\til{N} \ot_{C} \til{P}) \]
in $\cat{C}(A \ot D^{\mrm{op}})$.
\end{proof}

A matter of notation: suppose we are given morphisms $\phi : M' \to M$ in \lb 
$\cat{D}(A \ot B^{\mrm{op}})$
and $\psi : N' \to N$ in $\cat{D}(B \ot C^{\mrm{op}})$.
The result of applying the bifunctor $(- \ot_B^{\mrm{L}} -)$ is the morphism
\begin{equation} \label{eqn:1420}
 \phi \ot^{\mrm{L}}_B \psi : M' \ot_B^{\mrm{L}} N' \to M \ot_B^{\mrm{L}} N
\end{equation}
in $\cat{D}(A \ot C^{\mrm{op}})$. 

\begin{prop} \label{prop:1031}
Let $A$, $B$ and $C$ be rings, and assume $C$ is flat. 
\begin{enumerate}
\item The bifunctor 
\[ \opn{Hom}_B(-, -) :  \cat{M}(B \ot A^{\mrm{op}})^{\mrm{op}} \times 
\cat{M}(B \ot C^{\mrm{op}}) \to \cat{M}(A \ot C^{\mrm{op}})  \]
has a right derived bifunctor 
\[ \opn{RHom}_B(-, -) :  \cat{D}(B \ot A^{\mrm{op}})^{\mrm{op}} \times 
\cat{D}(B \ot C^{\mrm{op}}) \to \cat{D}(A \ot C^{\mrm{op}}) . \]

\item Given $M \in \cat{C}(B \ot A^{\mrm{op}})$
and $N \in \cat{C}(B \ot C^{\mrm{op}})$,
such that either $M$ is K-projective over $B$ or $N$ is K-injective over $B$, 
the morphism 
\[ \xi^{\mrm{R}}_{M, N} : \opn{Hom}_B(M, N) \to \opn{RHom}_B(M, N) \]
in $\cat{D}(A \ot C^{\mrm{op}})$ is an isomorphism. 

\item Suppose we are given ring homomorphisms 
$A' \to A$ and $C' \to C$, such that $C'$ is K-flat. Then the diagram 
\[ \UseTips \xymatrix @C=14ex @R=6ex {
\cat{D}(B \ot A^{\mrm{op}})^{\mrm{op}} \times \cat{D}(B \ot C^{\mrm{op}})
\ar[r]^(0.6){ \opn{RHom}_B(-, -) }
\ar[d]_{\opn{Rest} \times \opn{Rest}}
&
\cat{D}(A \ot C^{\mrm{op}})
\ar[d]^{\opn{Rest}}
\\
\cat{D}(B \ot A'^{\, \mrm{op}})^{\mrm{op}} \times 
\cat{D}(B \ot C'^{\, \mrm{op}})
\ar[r]^(0.6){ \opn{RHom}_{B}(-, -) }
&
\cat{D}(A' \ot C'^{\, \mrm{op}})
} \]
is commutative up to an isomorphism of triangulated bifunctors. 

\item Suppose $D$ is another flat ring. Then for 
$M \in \cat{D}(B \ot A^{\mrm{op}})$,
$N \in \cat{D}(C \ot B^{\mrm{op}})$ and 
$L \in \cat{D}(C \ot D^{\mrm{op}})$
there is an isomorphism 
\[ \opn{RHom}_B(M, \opn{RHom}_C(N, L)) \cong 
\opn{RHom}_C(N \ot_B^{\mrm{L}} M, L) \]
in $\cat{D}(A \ot D^{\mrm{op}})$. 
This isomorphism is functorial in the objects $M, N, L$.
\end{enumerate}

\end{prop}

\begin{proof}
(1-3) This is like the proof of Proposition \ref{prop:1042}, but now we rely on 
Lemma \ref{lem:1070}(2) and \cite[Proposition 2.6(2)]{Ye2}.

\medskip \noindent 
(4) Here we choose a quasi-isomorphism $L \to J$ in 
$\cat{C}(C \ot D^{\mrm{op}})$ into a complex $J$ that is K-injective over 
$C$, and a quasi-isomorphism $\til{N} \to N$ in 
$\cat{C}(C \ot B^{\mrm{op}})$ from a complex $\til{N}$ that is K-flat 
over $B^{\mrm{op}}$. A calculation shows that 
$\opn{Hom}_C(\til{N}, J)$ is K-injective over $B$. 
The desired isomorphism comes from the obvious adjunction isomorphism 
\[ \opn{Hom}_B(M, \opn{Hom}_C(\til{N}, J)) \cong 
\opn{Hom}_C(\til{N} \ot_B M, J) \]
in $\cat{C}(A \ot D^{\mrm{op}})$.         
\end{proof}

Suppose we are given morphisms $\phi : M \to M'$ in 
$\cat{D}(B \ot A^{\mrm{op}})$
and $\psi : N' \to N$ in $\cat{D}(B \ot C^{\mrm{op}})$.
The result of applying the bifunctor $\opn{RHom}_B(-, -)$
is the morphism
\begin{equation} \label{eqn:1425}
\opn{RHom}_B(\phi, \psi) : \opn{RHom}_B(M', N') \to \opn{RHom}_B(M, N) 
\end{equation}
in $\cat{D}(A \ot C^{\mrm{op}})$.

If $A$ and $B$ are flat rings, then according to Propositions \ref{prop:1042}
and \ref{prop:1031} we obtain triangulated bifunctors 
\begin{equation} \label{eqn:1426}
(- \ot^{\mrm{L}}_{A} -) : \cat{D}(A^{\mrm{en}}) \times 
\cat{D}(A^{\mrm{en}}) \to \cat{D}(A^{\mrm{en}}) ,
\end{equation}
\begin{equation} \label{eqn:1427}
(- \ot^{\mrm{L}}_{A} -) : \cat{D}(A^{\mrm{en}}) \times 
\cat{D}(A \ot B) \to \cat{D}(A \ot B) 
\end{equation}
and
\begin{equation} \label{eqn:1435}
\opn{RHom}_A(-, -) : \cat{D}(A^{\mrm{en}})^{\mrm{op}} \times 
\cat{D}(A \ot B) \to \cat{D}(A \ot B) .
\end{equation}

{\em Monoidal categories} are defined in \cite[Section XI.1]{Mc}.  

\begin{prop} \label{prop:1425}
Suppose $A$ and $B$ are flat rings.
\begin{enumerate}
\item The operation \tup{(\ref{eqn:1426})} is a monoidal structure on the 
category $\cat{D}(A^{\mrm{en}})$, with unit object $A$. 

\item The operation \tup{(\ref{eqn:1427})} is a left monoidal action of the 
monoidal category $\cat{D}(A^{\mrm{en}})$ on the category 
$\cat{D}(A \ot B)$.

\item The operation \tup{(\ref{eqn:1435})} is a right monoidal action of the 
monoidal category $\cat{D}(A^{\mrm{en}})$ on the category 
$\cat{D}(A \ot B)$.
\end{enumerate}
\end{prop}

\begin{proof}
Since it is enough to check the coherence axioms for K-flat complexes in 
$\cat{D}(A^{\mrm{en}})$ and K-injective complexes in 
$\cat{D}(A \ot B)$, these statements are immediate consequences of 
the corresponding coherence axioms for the ordinary tensor and Hom operations,
and the adjunctions between them.  
See \cite[Section XI.2]{Mc}. 
\end{proof}

Let us introduce the notation 
\begin{equation} \label{eqn:1428}
\opn{lu} : A \ot_A^{\mrm{L}} M \iso M 
\end{equation}
for the {\em left unitor} isomorphism in $\cat{D}(A^{\mrm{en}})$ or 
$\cat{D}(A \ot B)$; and 
\begin{equation} \label{eqn:1429}
\opn{ru} : N \ot_A^{\mrm{L}} A \iso N 
\end{equation}
for the {\em right unitor} isomorphism in $\cat{D}(A^{\mrm{en}})$. 
We shall also require the canonical isomorphism 
\begin{equation} \label{eqn:1436}
\opn{lcu} : M \iso \opn{RHom}_A(A, M) 
\end{equation}
coming from the action (\ref{eqn:1435}), that (for lack of pre-existing name) 
we call the {\em left co-unitor}.  

\begin{rem} \label{rem:1425}
Actually, in the setup of Proposition \ref{prop:1425},  $\cat{D}(A^{\mrm{en}})$ 
is a {\em biclosed monoidal category}. The two internal Hom operations are 
$\opn{RHom}_A(-, -)$ and \lb $\opn{RHom}_{A^{\mrm{op}}}(-, -)$.
See \cite[Section VII.7]{Mc} or \cite{nLab} regarding these concepts. 
\end{rem}

\begin{rem} \label{rem:1420}
Here is the way to handle derived categories of bimodules in the absence of 
flatness. The idea is to choose K-flat resolutions 
$\til{A} \to A$ and $\til{B} \to B$ in the category 
$\catt{DGRng}^{\leq 0} \centover \K$
of nonpositive central DG $\K$-rings. See \cite{Ye2} or \cite{Ye6}. Then the
triangulated category
$\cat{D}(\til{A} \ot \til{B}^{\mrm{op}})$,
the derived category of DG left modules over the DG ring
$\til{A} \ot \til{B}^{\mrm{op}}$, is called the
{\em derived category of $A$-$B$-bimodules}.  
Propositions \ref{prop:1042}, \ref{prop:1031} and \ref{prop:1425} have 
straightforward extensions to the DG setup. 

The fact that the category 
$\cat{D}(\til{A} \ot \til{B}^{\mrm{op}})$
is independent of the resolutions (up to a canonical equivalence of 
triangulated categories) is not easy to prove. This will be done in the future 
papers \cite{VY} and \cite{Ye3}; see also the lecture notes \cite{Ye4}. 
\end{rem}

%\cleardoublepage
\section{Idempotent Copointed Objects} \label{sec:cop-obj}

Recall that we are working over a commutative base ring $\K$, 
and Convention \ref{conv:1425} is in force. From this section onward we also  
assume the following convention:

\begin{conv} \label{conv:1445}
The rings $A$ and $B$ are flat over $\K$. 
\end{conv}

Consider the enveloping ring 
$A^{\mrm{en}} = A \ot A^{\mrm{op}}$ of $A$. 
According to Proposition \ref{prop:1425} the triangulated category
$\cat{D}(A^{\mrm{en}})$
has a monoidal structure $(- \ot_{A}^{\mrm{L}} -)$,
with unit object $A$. 
There are also a left monoidal action 
$(- \ot_{A}^{\mrm{L}} -)$, and a right monoidal action  
$\opn{RHom}_A(-, -)$, of $\cat{D}(A^{\mrm{en}})$ 
on $\cat{D}(A \ot B^{\mrm{op}})$.
Correspondingly there are unitor isomorphisms 
$\opn{lu}$, $\opn{ru}$ and $\opn{lcu}$, explained in formulas 
(\ref{eqn:1428}), (\ref{eqn:1429}) and (\ref{eqn:1436}). 

\begin{dfn} \label{dfn:1185} \mbox{}
\begin{enumerate}
\item A {\em copointed object} in the monoidal category $\cat{D}(A^{\mrm{en}})$ 
is a pair $(P, \rho)$, consisting of a complex 
$P \in \cat{D}(A^{\mrm{en}})$ and a morphism 
$\rho : P \to A$ in $\cat{D}(A^{\mrm{en}})$.

\item The copointed object $(P, \rho)$ is called 
{\em idempotent} if the morphisms 
\[  \opn{lu} \circ \, (\rho \ot^{\mrm{L}}_A \opn{id}) , \ 
\opn{ru} \circ \, (\opn{id}  \ot^{\mrm{L}}_A \, \rho) : 
\ P \ot_A^{\mrm{L}} P \to P \]
in $\cat{D}(A^{\mrm{en}})$ are both isomorphisms.
\end{enumerate}
\end{dfn}

\begin{dfn} \label{dfn:1075}
Let  $(P, \rho)$ be a copointed object in $\cat{D}(A^{\mrm{en}})$. 
\begin{enumerate}
\item Define the triangulated functors 
\[ F , G : \cat{D}(A \ot B^{\mrm{op}}) \to \cat{D}(A \ot B^{\mrm{op}}) \]
to be
\[ F := P \ot_{A}^{\mrm{L}} (-) \quad \tup{and} \quad
G := \opn{RHom}_{A}(P, - ) . \]
 
\item Let 
\[ \si : F \to \opn{Id}_{\cat{D}(A \ot B^{\mrm{op}})} \quad \tup{and} \quad 
\tau : \opn{Id}_{\cat{D}(A \ot B^{\mrm{op}})} \to G \]
be the morphisms of triangulated functors from $\cat{D}(A \ot B^{\mrm{op}})$ to 
itself 
that are induced by the morphism $\rho : P \to A$. Namely
\[ \si_M : F(M) = P \ot_{A}^{\mrm{L}} M \to M , \quad 
\si_M := \opn{lu} \circ \, (\rho \ot_{A}^{\mrm{L}} \opn{id}_M)  \]
and 
\[ \tau_M : M \to G(M) = \opn{RHom}_{A}(P, M) , \quad 
\tau := \opn{RHom}_A(\rho, \opn{id}_M)  \circ \opn{lcu} . \]
\end{enumerate}
We refer to $(F, \si)$ and $(G, \tau)$ as the (co)pointed triangulated functors 
induced by the copointed object $(P, \rho)$.  
\end{dfn}

See formulas \ref{eqn:1428} and \ref{eqn:1436} regarding the isomorphisms 
$\opn{lu}$ and $\opn{lcu}$. Item (2) of the definition is shown in the 
commutative diagrams below 
in the category $\cat{D}(A \ot B^{\mrm{op}})$. 

\begin{equation} \label{eqb:1430}
\UseTips \xymatrix @C=8ex @R=6ex {
P \ot_{A}^{\mrm{L}} M
\ar[d]_{\rho \, \ot_{A}^{\mrm{L}} \opn{id}_M}
\ar[dr]^{\si_M}
\\
A \ot_{A}^{\mrm{L}} M
\ar[r]_(0.6){\opn{lu}}
&
M
}
\qquad
\UseTips \xymatrix @C=8ex @R=6ex {
M
\ar[r]^(0.3){\opn{lcu}}
\ar[dr]_{\tau_M}
&
\opn{RHom}_{A}(A, M)
\ar[d]^{\opn{RHom}_A(\rho, \opn{id}_M)}
\\
&
\opn{RHom}_{A}(P, M)
}
\end{equation}

\begin{dfn} \label{dfn:1232}
Let $(F, \si)$ and $(G, \tau)$ be the copointed and pointed triangulated 
functors on $\cat{D}(A \ot B^{\mrm{op}})$ from Definition \tup{\ref{dfn:1075}}. 
\begin{enumerate}
\item We define the full triangulated subcategory
$\cat{D}(A \ot B^{\mrm{op}})_F$ of  $\cat{D}(A \ot B^{\mrm{op}})$ to be 
\[ \cat{D}(A \ot B^{\mrm{op}})_F := 
\bigr\{ M \mid \si_M : F(M) \to M \tup{ is an isomorphism} \bigl\} . \]

\item  We define the full triangulated subcategory
$\cat{D}(A \ot B^{\mrm{op}})_G$ of  $\cat{D}(A \ot B^{\mrm{op}})$ to be 
\[ \cat{D}(A \ot B^{\mrm{op}})_G := 
\bigr\{ M \mid \tau_M : M \to G(M) \tup{ is an isomorphism} \bigl\} . \]
\end{enumerate}
\end{dfn}

\begin{rem} \label{rem:1240}
The notation used in the two definitions above was chosen to be consistent with 
that of \cite{PSY1}, in which $P$ is the telescope complex associated to a 
weakly proregular generating sequence of the ideal $\a$. See
Section \ref{sec:comm-rings} above, and Definition 3.8, 
Definition 3.11, Proposition 5.8 and Corollary 5.25 of \cite{PSY1}, 
where the functors are $G = \mrm{L} \La_{\a}$ and $F = \mrm{R} \Ga_{\a}$.
\end{rem}

\begin{lem} \label{lem:1340}
If the copointed object  $(P, \rho)$ is idempotent, then the copointed 
triangulated functor $(F, \si)$ and the 
pointed triangulated functor $(G, \tau)$ on $\cat{D}(A \ot B^{\mrm{op}})$
are idempotent.
\end{lem}

\begin{proof}
For $M \in \cat{D}(A \ot B^{\mrm{op}})$ there are equalities
(up to the associativity isomorphism of $(- \ot_A^{\mrm{L}} -)$, 
that should be inserted in the locations marked by ``$\dag$''): 
\begin{equation} \label{eqn:1437}
F(\si_M) = \opn{id}_P \ot_A^{\mrm{L}} \, 
\bigl( \opn{lu} \circ \, (\rho \ot_{A}^{\mrm{L}} \opn{id}_M) \bigr) 
=^{\dag} \, 
\bigl( \opn{ru} \circ \, (\opn{id}_P \ot_A^{\mrm{L}} \, \rho) \bigr)
\ot_{A}^{\mrm{L}} \opn{id}_M 
\end{equation}
and 
\begin{equation} \label{eqn:1438}
\si_{F(M)} = \opn{lu} \circ \, \bigl( \rho \ot_A^{\mrm{L}} 
(\opn{id}_P \ot_A^{\mrm{L}} \opn{id}_M) \bigr) 
=^{\dag} \, 
\bigl( \opn{lu} \circ (\rho \ot_A^{\mrm{L}} \opn{id}_P) \bigr) \ot_A^{\mrm{L}} 
\opn{id}_M 
\end{equation}
of morphisms
\[ F(F(M)) = P \ot_A^{\mrm{L}} P \ot_A^{\mrm{L}} M \to 
F(M) =  P \ot_A^{\mrm{L}} M  \]
in $\cat{D}(A \ot B^{\mrm{op}})$.
Because both
\begin{equation} \label{eqn:2037}
\opn{lu} \circ \, (\rho \ot_A^{\mrm{L}} \opn{id}_P) \quad 
\tup{and} \quad 
\opn{ru} \circ \, (\opn{id}_P \ot_A^{\mrm{L}} \, \rho)
\end{equation}
are isomorphisms in $\cat{D}(A^{\mrm{en}})$, it follows that the morphisms
$F(\si_M)$ and $\si_{F(M)}$ are isomorphisms in 
$\cat{D}(A \ot B^{\mrm{op}})$. 

There are also  equalities (up to the associativity and adjunction isomorphisms 
of $(- \ot_A^{\mrm{L}} -)$ and $\opn{RHom}_A(-,-)$, 
that should be inserted in the locations marked by by ``$\ddag$''): 
\begin{equation} \label{eqn:1439}
\begin{aligned}
& G(\tau_M) = \opn{RHom}_A(\opn{id}_P,  \opn{RHom}_A(\rho, \opn{id}_M)) \circ 
\opn{RHom}_A(\opn{id}_P, \opn{lcu}) \\
& \qquad =^{\ddag} \, 
\opn{RHom}_A(\rho \ot_A^{\mrm{L}} \opn{id}_P, \opn{id}_M) \circ  
\opn{RHom}_A(\opn{lu}, \opn{id}_M) 
\end{aligned}
\end{equation}
and 
\begin{equation} \label{eqn:1440}
\begin{aligned}
& \tau_{G(M)} = \opn{RHom}_A(\rho, \opn{id}_{\opn{RHom}_A(P, M)}) \circ 
\opn{lcu} \\
& \qquad =^{\ddag} \, 
\opn{RHom}_A(\opn{id}_P \ot_A^{\mrm{L}} \, \rho) 
\circ  \opn{RHom}_A(\opn{ru}, \opn{id}_M) 
\end{aligned}
\end{equation}
of morphisms
\[ \begin{aligned}
& G(M) = \opn{RHom}_A(P, M) \to 
\\
&
\qquad G(G(M)) = \opn{RHom}_A \bigl( P, \opn{RHom}_A(P, M)) \cong
\opn{RHom}_A(P \ot_A^{\mrm{L}} P, M) 
\end{aligned}   \]
in $\cat{D}(A \ot B^{\mrm{op}})$. 
Because of the isomorphisms (\ref{eqn:2037}) in $\cat{D}(A^{\mrm{en}})$, it 
follows that the morphisms $G(\tau_M)$ and $\tau_{G(M)}$ are isomorphisms in 
$\cat{D}(A \ot B^{\mrm{op}})$. 
\end{proof}

\begin{lem} \label{lem:1470}
Consider the functors $F$ and $G$ on $\cat{D}(A \ot B^{\mrm{op}})$ 
from Definition \tup{\ref{dfn:1075}}. For every  
$M, N \in \cat{D}(A \ot B^{\mrm{op}})$ there is a bijection 
\[ \opn{Hom}_{\cat{D}(A \ot B^{\mrm{op}})}(F(M), N) \cong 
\opn{Hom}_{\cat{D}(A \ot B^{\mrm{op}})}(M, G(N)) , \]
and it is functorial in $M$ and $N$.  
\end{lem}

\begin{proof}
Choose a K-injective resolution $N \to J$ in 
$\cat{C}(A \ot B^{\mrm{op}})$, and a K-flat resolution 
$\til{P} \to P$ in $\cat{C}(A^{\mrm{en}})$.
The usual Hom-tensor adjunction gives rise to an isomorphism 
\begin{equation} \label{eqn:1470}
\opn{Hom}_{A \ot B^{\mrm{op}}}(\til{P} \ot_A M, J) \cong 
\opn{Hom}_{A \ot B^{\mrm{op}}}(M, \opn{Hom}_{A}(\til{P}, J))
\end{equation}
in $\cat{C}(\K)$. From this we deduce that 
$\opn{Hom}_{A}(\til{P}, J)$
is K-injective in $\cat{C}(A \ot B^{\mrm{op}})$.
We see that the isomorphism (\ref{eqn:1470}) represents an isomorphism 
\begin{equation} \label{eqn:1471}
\opn{RHom}_{A \ot B^{\mrm{op}}}(P \ot_A^{\mrm{L}} M, N) \cong 
\opn{RHom}_{A \ot B^{\mrm{op}}}(M, \opn{RHom}_{A}(P, N))
\end{equation}
in $\cat{D}(\K)$. Taking $\opn{H}^0$ in (\ref{eqn:1471}) gives us the 
isomorphism 
\[ \opn{Hom}_{\cat{D}(A \ot B^{\mrm{op}})}(P \ot_A^{\mrm{L}} M, N) \cong 
\opn{Hom}_{\cat{D}(A \ot B^{\mrm{op}})}(M, \opn{RHom}_{A}(P, N)) \]
in $\cat{M}(\K)$. This is what we want. 
\end{proof}

\begin{lem} \label{lem:1180}
Consider the functors $F$ and $G$ on $\cat{D}(A \ot B^{\mrm{op}})$ 
from Definition \tup{\ref{dfn:1075}}.  The kernel of $F$ equals the kernel 
of $G$. Namely for every $M \in \cat{D}(A \ot B^{\mrm{op}})$ we have $F(M) = 0$ 
if and only if $G(M) = 0$. 
\end{lem}

\begin{proof}
We shall use the adjunction formula from Lemma \ref{lem:1470}, with $N = M$. 

First assume $F(M) = 0$. Then 
$\opn{Hom}_{\cat{D}(A \ot B^{\mrm{op}})}(F(M), M)$ is zero, 
and by Lemma \ref{lem:1470} we see that 
$\opn{Hom}_{\cat{D}(A \ot B^{\mrm{op}})}(M, G(M))$ is zero too. 
This implies that the morphism $\tau_M : M \to G(M)$ is zero.
Applying $G$ to it we deduce that the morphism 
\[ G(\tau_M) : G(M) \to G(G(M)) \]
is zero. But by Lemma \ref{lem:1340} the pointed functor 
$(G, \tau)$ is idempotent, and this means that $G(\tau_M)$ is an isomorphism. 
Therefore $G(M) = 0$. 

Now assume that $G(M) = 0$. Again using Lemma \ref{lem:1470}, but now in 
the reverse direction, we see that the morphism
$\si_M : F(M) \to M$ is zero. Therefore the morphism
\[ F(\si_M) : F(F(M)) \to F(M) \]
is zero. But by Lemma \ref{lem:1340} the copointed functor 
$(F, \si)$ is idempotent, and this means that $F(\si_M)$ is an isomorphism. 
Therefore $F(M) = 0$. 
\end{proof}

Recall that Conventions \ref{conv:1425} and \ref{conv:1445} are in force. 

\begin{thm}[Abstract Equivalence] \label{thm:1075}
Let $A$ and $B$ be flat rings, and let $(P, \rho)$ be an 
idempotent copointed object in $\cat{D}(A^{\mrm{en}})$.
Consider the triangulated functors 
\[ F, G : \cat{D}(A \ot B^{\mrm{op}}) \to \cat{D}(A \ot B^{\mrm{op}}) \]
and the categories 
$\cat{D}(A \ot B^{\mrm{op}})_F$ and $\cat{D}(A \ot B^{\mrm{op}})_G$ from 
Definitions \tup{\ref{dfn:1075}}
and \tup{\ref{dfn:1232}}. Then the following hold\tup{:}
\begin{enumerate}
\item The functor $G$ is a right adjoint to $F$. 

\item The copointed triangulated functor $(F, \si)$ and the 
pointed triangulated functor $(G, \tau)$ are idempotent.

\item The categories  $\cat{D}(A \ot B^{\mrm{op}})_F$ and 
$\cat{D}(A \ot B^{\mrm{op}})_G$ are the essential 
images of the functors $F$ and $G$ respectively. 

\item The functor 
\[ F : \cat{D}(A \ot B^{\mrm{op}})_G \to \cat{D}(A \ot B^{\mrm{op}})_F \]
is an equivalence of triangulated categories, with quasi-inverse $G$. 
\end{enumerate}
\end{thm}

\begin{proof}
(1) This is Lemma \ref{lem:1470}. 

\medskip \noindent
(2) This is Lemma \ref{lem:1340}. 

\medskip \noindent
(3) Take any $M \in \cat{D}(A \ot B^{\mrm{op}})_G$. Then $M \cong G(M)$, so 
that $M$ is in the essential image of $G$. Conversely, suppose there is an 
isomorphism $\phi : M \iso G(N)$ for some $N \in \cat{D}(A \ot B^{\mrm{op}})$. 
We have to prove that $\tau_M$ is an isomorphism. There is a commutative 
diagram 
\[ \UseTips \xymatrix @C=8ex @R=6ex {
M
\ar[r]^{ \phi }
\ar[d]_{ \tau_M }
& 
G(N)
\ar[d]^{ \tau_{G(N)} }
\\
G(M)
\ar[r]^(0.44){ G(\phi) }
&
G(G(N))
} \]
in $\cat{D}(A \ot B^{\mrm{op}})$ with horizontal isomorphisms.
By Lemma \ref{lem:1340} the morphism $\tau_{G(N)}$ is an isomorphism. 
Therefore $\tau_M$ is an isomorphism. 

A similar argument (with reversed arrows) tells us that the essential image of 
$F$ is $\cat{D}(A \ot B^{\mrm{op}})_F$.

\medskip \noindent 
(4) The morphism $\rho : P \to A$  sits inside a 
distinguished triangle 
\begin{equation} \label{eqn:1184}
P \xar{\rho} A \to N \xar{\vartriangle}  
\end{equation}
in $\cat{D}(A^{\mrm{en}})$. Let us apply the functor 
$P \ot_A^{\mrm{L}} - $ to (\ref{eqn:1184}). We get a distinguished triangle 
\[ P \ot_A^{\mrm{L}} P 
\xar{\opn{ru} \circ \, (\opn{id} \ot_A^{\mrm{L}} \, \rho)}
P \to P \ot_A^{\mrm{L}} N \xar{\vartriangle} \]
in $\cat{D}(A^{\mrm{en}})$.
By the idempotence condition, the first morphism above is an isomorphism; and 
hence 
$P \ot_A^{\mrm{L}} N = 0$. Therefore for every $M \in \cat{D}(A \ot 
B^{\mrm{op}})$ the complex
\[ F(N \ot_A^{\mrm{L}} M) = P \ot_A^{\mrm{L}} N \ot_A^{\mrm{L}} M \]
is zero. Lemma \ref{lem:1180} tells us that 
\[ G(N \ot_A^{\mrm{L}} M) = 
\opn{RHom}_{A}(P, N \ot_A^{\mrm{L}} M) \]
is zero. 

Now we go back to the distinguished triangle (\ref{eqn:1184}) and we apply to 
it the functor $(-) \ot_A^{\mrm{L}} M$, and then the functor
$\opn{RHom}_{A}(P, - )$. The result is the distinguished triangle 
\[ \opn{RHom}_{A}(P, P \ot_A^{\mrm{L}} M) \xar{ \al_M } 
\opn{RHom}_{A}(P, M) \to 
\opn{RHom}_{A}(P, N \ot_A^{\mrm{L}} M)  \xar{ \vartriangle } \]
in $\cat{D}(A \ot B^{\mrm{op}})$. Because the third term is zero, it follows 
that $\al_M : G(F(M)) \to G(M)$ is an isomorphism. 
If moreover $M \in \cat{D}(A \ot B^{\mrm{op}})_G$, then $\tau_M$ is an 
isomorphism too, and thus we have an isomorphism 
\[ \tau_M^{-1} \circ \al_M : G(F(M)) \to M \]
that's functorial in $M$. 

Similarly, if we apply the functor 
$(-) \ot_A^{\mrm{L}} P$ to (\ref{eqn:1184}), we get a distinguished triangle 
\[ P \ot_A^{\mrm{L}} P 
\xar{\opn{lu} \circ \, (\rho \, \ot_A^{\mrm{L}} \opn{id})}
P \to N \ot_A^{\mrm{L}} P \xar{\vartriangle} \]
in $\cat{D}(A^{\mrm{en}})$.
By the idempotence condition, the first morphism above is an isomorphism; and 
hence 
$N \ot_A^{\mrm{L}} P = 0$. Therefore for every $M \in \cat{D}(A \ot 
B^{\mrm{op}})$ the complex
\[ G(\opn{RHom}_{A}(N, M)) = 
\opn{RHom}_{A}(P, \opn{RHom}_{A}(N, M)) \cong 
\opn{RHom}_{A}(N \ot_A^{\mrm{L}} P, M) \]
is zero. Lemma \ref{lem:1180} tells us that 
\[ F(\opn{RHom}_{A}(N, M)) = P \ot_A^{\mrm{L}} \opn{RHom}_{A}(N, M) \]
is zero. 

Next we apply the functor $\opn{RHom}_{A}(-, M)$, 
and then the functor $P \ot_A^{\mrm{L}} (-)$, to the distinguished triangle 
(\ref{eqn:1184}). We obtain distinguished triangle 
\[ P \ot_A^{\mrm{L}} \opn{RHom}_{A}(N, M)  \to 
P \ot_A^{\mrm{L}} M \xar{\be_M}
P \ot_A^{\mrm{L}} \opn{RHom}_{A}(P, M)  \xar{ \vartriangle } \]
in $\cat{D}(A \ot B^{\mrm{op}})$. By the previous calculation the first term in 
this triangle is zero, and so 
$\be_M : F(M) \to F(G(M))$ is an isomorphism. 
If moreover $M \in \cat{D}(A \ot B^{\mrm{op}})_F$, then $\si_M$ is an 
isomorphism too, and thus we have an isomorphism 
\[ \be_M \circ \si_M^{-1} : M \to F(G(M)) \]
that's functorial in $M$.
\end{proof}

\begin{rem} \label{rem:1421}
The content of this section is not very hard to extend to the DG setup: $A$ and 
$B$ can be nonpositive K-flat central DG $\K$-rings, as in Remark 
\ref{rem:1420}.
\end{rem}

\begin{rem} \label{rem:1445}
Theorem \ref{thm:1075} can be interpreted as saying that every idempotent 
copointed object in $\cat{D}(A^{\mrm{en}})$ induces a recollement on the
category $\cat{D}(A \ot B^{\mrm{op}})$. This is not unexpected; cf.\ 
\cite[Paragraph 4.13.1]{Kr}. This perspective, in the context of weakly stable 
torsion classes, will be examined in a future paper.
\end{rem}

%\cleardoublepage
\section{From Torsion Classes to Copointed Objects} 
\label{sec:tors-to-obj}

In this section we prove Theorems \ref{thm:1200} and \ref{thm:1370}.
Conventions \ref{conv:1425} and \ref{conv:1445} are in place. 
Recall that $\K$ is a commutative base ring, and $A$ and $B$ are flat 
central $\K$-rings. We write $\ot$ as shorthand for $\ot_{\K}$. There is a 
restriction functor 
\[ \opn{Rest}_A : \cat{M}(A \ot B^{\mrm{op}}) \to \cat{M}(A) \]
that forgets the right $B$-module structure. This extends to restriction 
functors on $\cat{C}(-)$ and $\cat{D}(-)$ with the same notation.

Let $\catt{T}$ be a torsion class in $\cat{M}(A)$.
The torsion functor associated to  $\catt{T}$  is $\Ga_{\catt{T}}$, and the 
Gabriel filter is $\opn{Filt}(\catt{T})$. See Section \ref{sec:tors-cls} for a 
review of these concepts.

\begin{dfn} \label{dfn:1475}
Given a torsion class $\catt{T} \sub \cat{M}(A)$, the 
{\em bimodule torsion class} 
\[ \catt{T} \ot B^{\mrm{op}} \sub \cat{M}(A \ot B^{\mrm{op}}) \]
is  defined as follows:
\[ \catt{T} \ot B^{\mrm{op}} := 
\bigl\{ M \in \cat{M}(A \ot B^{\mrm{op}}) \mid 
\opn{Rest}_A(M) \in \catt{T} \bigr\} .  \]
\end{dfn}

\begin{rem} \label{rem:2040}
The expression $\catt{T} \ot B^{\mrm{op}}$ in the definition above does mean 
that this is a genuine tensor product here; it is only suggestive notation. 

However, on the level of Gabriel filters, the expression
$\catt{T} \ot B^{\mrm{op}}$ is a better 
approximation of what really happens:  
the filter $\opn{Filt}(\catt{T}  \ot B^{\mrm{op}})$
is the filter of left ideals in the ring in $A \ot B^{\mrm{op}}$ 
that is generated by the set of left ideals
\[ \bigl\{ \a \ot B^{\mrm{op}} \sub A \ot B^{\mrm{op}}
\mid \a \in \opn{Filt}(\catt{T})  \bigr\} . \]
See next example. 
\end{rem}

\begin{exa} \label{exa:2040}   
Suppose $\a \sub A$ is a two-sided ideal that is finitely generated as a left 
ideal.  In the notation of Definition \ref{dfn:1065}, there is a torsion class
$\catt{T} :=  \catt{T}_{\a} \sub \cat{M}(A)$.
This places us in the situation of Definition \ref{dfn:1475}, and there is a 
bimodule torsion class 
$\catt{T} \ot B^{\mrm{op}}$ in $\cat{M}(A \ot B^{\mrm{op}})$.

There is another way to view the bimodule torsion class 
$\catt{T} \ot B^{\mrm{op}}$. Consider 
the two-sided ideal
$\a \ot B^{\mrm{op}}$ in the ring $A \ot B^{\mrm{op}}$. 
Then, as torsion classes 
in $\cat{M}(A \ot B^{\mrm{op}})$,
and with the notation of Definition \ref{dfn:1065}, we have
$\catt{T} \ot B^{\mrm{op}} = \catt{T}_{\a \ot B^{\mrm{op}}}$.
\end{exa}

The torsion functor $\Ga_{\catt{T} \ot B^{\mrm{op}}}$ on
$\cat{M}(A \ot B^{\mrm{op}})$
satisfies this formula: there is equality
\begin{equation}  \label{eqn:1050} 
\Ga_{\catt{T} \ot B^{\mrm{op}}}(M) = \lim_{\substack{ \xar{ } \\ \a \in 
\opn{Filt}(\catt{T}) }} 
\, \opn{Hom}_A(A / \a, M)   
\end{equation}
of submodules of $M$. 
There is a morphism of functors 
\begin{equation}   \label{eqn:1080} 
\si : \Ga_{\catt{T} \ot B^{\mrm{op}}} \to 
\opn{Id}_{\cat{M}(A \ot B^{\mrm{op}})} .
\end{equation}
The pair $(\Ga_{\catt{T} \ot B^{\mrm{op}}}, \si)$ 
is an idempotent copointed functor on the category  \lb 
$\cat{M}(A \ot B^{\mrm{op}})$. 
The copointed functor $(\Ga_{\catt{T} \ot B^{\mrm{op}}}, \si)$ extends in the 
obvious way 
to complexes, giving rise to an idempotent copointed functor on the category  
$\cat{C}(A \ot B^{\mrm{op}})$.

By definition the diagram of functors
\begin{equation} \label{eqn:1250}
 \UseTips \xymatrix @C=8ex @R=6ex {
\cat{M}(A \ot B^{\mrm{op}})
\ar[r]^{ \Ga_{\catt{T} \ot B^{\mrm{op}}} }
\ar[d]_{\opn{Rest}_A}
&
\cat{M}(A \ot B^{\mrm{op}})
\ar[d]^{\opn{Rest}_A}
\\
\cat{M}(A)
\ar[r]^{ \Ga_{\catt{T}} }
&
\cat{M}(A)
} 
\end{equation}
is commutative. Moreover, for every $M \in \cat{M}(A \ot B^{\mrm{op}})$ 
there is equality 
\begin{equation} \label{eqn:1475}
\opn{Rest}_A(\si_M) = \si_{\opn{Rest}_A(M)}
\end{equation}
of homomorphisms 
\[ \opn{Rest}_A(\Ga_{\catt{T} \ot B^{\mrm{op}}}(M)) \to \opn{Rest}_A(M) \]
in $\cat{M}(A)$, because they are both the inclusion of the $\K$-module 
$\Ga_{\catt{T}}(M)$ into the $\K$-module $M$. 

In case $B = \K$, so that $A \ot B^{\mrm{op}} = A$ and 
$\catt{T} \ot B^{\mrm{op}} = \catt{T}$, we are back in the situation studied 
in Section \ref{sec:tors-cls}.  

Here is a definition resembling Definition \ref{dfn:1130}. 

\begin{dfn} \label{dfn:1385}
A complex $I \in \cat{C}(A \ot B^{\mrm{op}})$ is called {\em $\catt{T}$-flasque 
over $A$} 
if $\opn{Rest}_A(I) \in \cat{C}(A)$ is $\catt{T}$-flasque, in the sense of 
Definition \ref{dfn:1038}(2).
\end{dfn}

The next two lemmas can be deduced from standard properties of right derived 
functors. Still, we shall go through the details, because they will serve as 
references in subsequent constructions. 

\begin{lem}  \label{lem:1053}
The functor $\Ga_{\catt{T} \ot B^{\mrm{op}}}$
has a right derived functor
\[ (\mrm{R} \Ga_{\catt{T} \ot B^{\mrm{op}}}, \xi^{\mrm{R}}) : 
\cat{D}(A \ot B^{\mrm{op}}) \to \cat{D}(A \ot B^{\mrm{op}}) . \]
If $M \in \cat{D}(A \ot B^{\mrm{op}})$ is a $\catt{T}$-flasque complex over 
$A$, 
then 
\[ \xi^{\mrm{R}}_M : \Ga_{\catt{T} \ot B^{\mrm{op}}}(M) \to 
\mrm{R}\Ga_{\catt{T} \ot B^{\mrm{op}}}(M) \]
is an isomorphism.
\end{lem}

\begin{proof}
For each $M \in \cat{C}(A \ot B^{\mrm{op}})$ 
we choose a K-injective resolution $\ze_M : M \to I_M$
in $\cat{C}(A \ot B^{\mrm{op}})$.
By Lemma \ref{lem:1070} the object 
$\opn{Rest}_A(I_M) \in \cat{C}(A)$ is K-injective, and hence it is 
$\catt{T}$-flasque. We define
\[ \mrm{R} \Ga_{\catt{T} \ot B^{\mrm{op}}}(M) := 
\Ga_{\catt{T} \ot B^{\mrm{op}}}(I_M) \]
and 
\[ \xi^{\mrm{R}}_M := \Ga_{\catt{T} \ot B^{\mrm{op}}}(\ze_M) : 
\Ga_{\catt{T} \ot B^{\mrm{op}}}(M) \to
\mrm{R} \Ga_{\catt{T} \ot B^{\mrm{op}}}(M) . \]

If $M$ is a $\catt{T}$-flasque complex over $A$, then the 
homomorphism
\[ \Ga_{\catt{T} \ot B^{\mrm{op}}}(\ze_M) : \Ga_{\catt{T} \ot B^{\mrm{op}}}(M) 
\to 
\Ga_{\catt{T} \ot B^{\mrm{op}}}(I_M) \]
in $\cat{C}(A \ot B)$ is a quasi-isomorphism, and hence $\xi^{\mrm{R}}_M$ is an 
isomorphism. 
\end{proof}

Here is the bimodule version of Lemma \ref{lem:1063}.

\begin{lem}  \label{lem:1082}
Consider the triangulated functor
$(\mrm{R} \Ga_{\catt{T} \ot B^{\mrm{op}}}, \xi^{\mrm{R}})$ from Lemma 
\tup{\ref{lem:1053}}.
There is a unique morphism
\[ \si^{\mrm{R}} : \mrm{R} \Ga_{\catt{T} \ot B^{\mrm{op}}} \to \opn{Id} \]
of triangulated functors from $\cat{D}(A \ot B^{\mrm{op}})$ to itself, 
satisfying this condition\tup{:}
for every $M \in \cat{D}(A \ot B^{\mrm{op}})$ the diagram 
\[ \UseTips \xymatrix @C=6ex @R=6ex {
\Ga_{\catt{T} \ot B^{\mrm{op}}}(M)
\ar[r]^{ \xi^{\mrm{R}}_M }
\ar[dr]_{ \si_M }
&
\mrm{R} \Ga_{\catt{T} \ot B^{\mrm{op}}}(M)
\ar[d]^{ \si^{\mrm{R}}_M }
\\
&
M
} \]
in $\cat{D}(A \ot B^{\mrm{op}})$ is commutative. 
\end{lem}

\begin{proof}
When $I$ is $\catt{T}$-flasque over $A$ the morphism $\xi^{\mrm{R}}_I$ is an 
isomorphism, and so the morphism 
$\si^{\mrm{R}}_I : \mrm{R} \Ga_{\catt{T} \ot B^{\mrm{op}}}(I) \to I$ 
must be $\si^{\mrm{R}}_I = \si_I \circ (\xi^{\mrm{R}}_I)^{-1}$. 
This implies the uniqueness of the morphism of functors $\si^{\mrm{R}}$. 

For existence we use the chosen K-injective resolutions $\ze_M : M \to I_M$ 
from the proof of Lemma \ref{lem:1053}, and define
\[ \si^{\mrm{R}}_{I_M} := \si_{I_M} \circ (\xi^{\mrm{R}}_{I_M})^{-1} \]
and
\[ \si^{\mrm{R}}_{M} := \ze_M^{-1} \circ  \si^{\mrm{R}}_{I_M} \circ 
\mrm{R} \Ga_{\catt{T} \ot B^{\mrm{op}}}(\ze_M) . \]
\end{proof}

We now have a triangulated copointed functor 
$(\mrm{R} \Ga_{\catt{T} \ot B^{\mrm{op}}}, \si^{\mrm{R}})$ on the category \lb 
$\cat{D}(A \ot B^{\mrm{op}})$.

\begin{lem}  \label{lem:1095}

Consider the triangulated copointed functors 
$(\mrm{R} \Ga_{\catt{T}}, \si^{\mrm{R}})$
and \lb 
$(\mrm{R} \Ga_{\catt{T} \ot B^{\mrm{op}}}, \si^{\mrm{R}})$
on the categories $\cat{D}(A)$ and $\cat{D}(A \ot B^{\mrm{op}})$ respectively. 
There is an isomorphism 
\[ \eta: \opn{Rest}_A \circ \, \mrm{R} \Ga_{\catt{T} \ot B^{\mrm{op}}} \iso 
\mrm{R} \Ga_{\catt{T}} \circ \opn{Rest}_A \]
of triangulated functors 
$\cat{D}(A \ot B^{\mrm{op}}) \to \cat{D}(A)$,
such that for every object 
$M \in \lb \cat{D}(A \ot B^{\mrm{op}})$
there is equality  
\[ \tag{$\diamond$} 
\opn{Rest}_A(\si^{\mrm{R}}_M) = 
\si^{\mrm{R}}_{\opn{Rest}_A(M)} \circ \eta_M \]
of morphisms 
\[ \opn{Rest}_A(\mrm{R} \Ga_{\catt{T} \ot B^{\mrm{op}}}(M)) \to 
\opn{Rest}_A(M) \]
in $\cat{D}(A)$.
\end{lem}

\begin{proof}
We use the K-injective resolutions $\ze_M : M \to I_M$
in $\cat{C}(A \ot B^{\mrm{op}})$ from the proof of Lemma \ref{lem:1053}
to present the functor 
$\mrm{R} \Ga_{\catt{T} \ot B^{\mrm{op}}}$.
Let us also choose a system of K-injective resolutions 
$\th_N : N \to J_N$ in the category $\cat{C}(A)$, and use that to 
present the functor $\mrm{R} \Ga_{\catt{T}}$.

For every complex $M \in \cat{C}(A \ot B^{\mrm{op}})$,
with image $N := \opn{Rest}_A(M) \in \cat{C}(A)$,
there exists a quasi-isomorphism 
$\de_M : \opn{Rest}_A(I_M) \to J_N$
in $\cat{C}(A)$, unique up to homotopy, such that the diagram
\[ \UseTips \xymatrix @C=12ex @R=6ex {
\opn{Rest}_A(M)
\ar[r]^{\opn{Rest}_A(\ze_M)}
\ar[d]_{\opn{id}}
&
\opn{Rest}_A(I_M)
\ar[d]^{\de_M}
\\
N
\ar[r]^{\th_N}
&
J_N
} \]
is commutative up to homotopy. 
Because both $\opn{Rest}_A(I_M)$ and $J_N$ are $\catt{T}$-flasque over $A$,
the homomorphism 
\[ \eta_M := \Ga_{\catt{T}}(\de_M) : \Ga_{\catt{T}}(\opn{Rest}_A(I_M)) \to 
\Ga_{\catt{T}}(J_N) \]
in $\cat{C}(A)$ is a quasi-isomorphism. The commutativity of diagram 
(\ref{eqn:1250}) implies that 
\[ \Ga_{\catt{T}}(\opn{Rest}_A(I_M)) =
\opn{Rest}_A(\Ga_{\catt{T} \ot B^{\mrm{op}}}(I_M))  \]
as submodules of $\opn{Rest}_A(I_M)$. 
Thus, as $M$ varies, we obtain an isomorphism of triangulated functors 
\[ \eta : \opn{Rest}_A \circ \, \mrm{R} \Ga_{\catt{T} \ot B^{\mrm{op}}} \iso 
\mrm{R} \Ga_{\catt{T}} \circ \opn{Rest}_A . \]

Finally, formula (\ref{eqn:1475}) says that
$\opn{Rest}_A(\si_{I_M}) = \si_{\opn{Rest}_A(I_M)}$
as homomorphisms 
\[ \opn{Rest}_A(\Ga_{\catt{T} \ot B^{\mrm{op}}}(I_M)) \to \opn{Rest}_A(I_M) \] 
in $\cat{C}(A)$. 
The construction of $\si^{\mrm{R}}_M$ in the proof of Lemma \ref{lem:1082}
shows that equality ($\diamond$) holds. 
\end{proof}

In view of Lemma \ref{lem:1095}, there is no harm to stop using the longhand
notation $\catt{T} \ot B^{\mrm{op}}$ for the bimodule torsion class in 
$\cat{M}(A \ot B^{\mrm{op}})$. From here on we shall mostly use the 
notation $\catt{T}$ to denote both the original left module torsion class and 
the extended bimodule torsion class. Thus we shall usually write 
$\mrm{R} \Ga_{\catt{T}}$ for 
the derived torsion functors on $\cat{D}(A)$ and on 
$\cat{D}(A \ot B^{\mrm{op}})$, and the context will determine which category 
is involved. This simplified notation will be especially helpful because the 
ring $B$ will be varying. 

For the next definition we take $B = A$, so that 
$A \ot B^{\mrm{op}} = A^{\mrm{en}}$ and, in the longhand notation, the bimodule 
torsion class is 
$\catt{T} \ot A^{\mrm{op}} \sub \cat{M}(A^{\mrm{en}})$.

The category $\cat{D}(A^{\mrm{en}})$ carries a monoidal structure -- 
see Proposition \ref{prop:1425}. Copointed objects in monoidal categories
were introduced in Definition \ref{dfn:1185}.
 
\begin{dfn} \label{dfn:1200}
Consider the copointed triangulated functor 
$(\mrm{R} \Ga_{\catt{T}}, \si^{\mrm{R}})$ on \lb 
$\cat{D}(A^{\mrm{en}})$.
Define the object 
\[ P := \mrm{R} \Ga_{\catt{T}}(A) \in \cat{D}(A^{\mrm{en}}) \]
and the morphism
\[ \rho : P \to A , \quad \rho := \si^{\mrm{R}}_A \]
in $\cat{D}(A^{\mrm{en}})$. 
We call the pair $(P, \rho)$ the {\em copointed object in 
$\cat{D}(A^{\mrm{en}})$ induced by $\catt{T}$}. 
\end{dfn}

Recall that Conventions \ref{conv:1425} and \ref{conv:1445} are assumed.
In the next theorem, $B$ is an arbitrary (flat central) $\K$-ring; it could, 
for instance, be $A$ or $\K$.
By Proposition \ref{prop:1425} there is a left monoidal action
$(- \ot_A^{\mrm{L}} -)$ of 
$\cat{D}(A^{\mrm{en}})$ on $\cat{D}(A \ot B^{\mrm{op}})$.

\begin{thm}[Representability of Derived Torsion] \label{thm:1200}
Let $A$ and $B$ be flat rings. Let $\catt{T}$ be a quasi-compact, 
finite dimensional, weakly stable torsion class in $\cat{M}(A)$, and let
$(P, \rho)$ be the copointed object in $\cat{D}(A^{\mrm{en}})$ 
from Definition \tup{\ref{dfn:1200}}. There is a unique isomorphism 
\[ \ga : P \ot_{A}^{\mrm{L}} (-) \iso \mrm{R} \Ga_{\catt{T}} \]
of triangulated functors from $\cat{D}(A \ot B^{\mrm{op}})$ to itself, 
such that for every complex $M$ the diagram 
\[ \UseTips \xymatrix @C=8ex @R=6ex {
P \ot_{A}^{\mrm{L}} M
\ar[r]^{\ga_{M}}
\ar[d]_{ \rho \, \ot_A^{\mrm{L}} \, \opn{id}_M }
&
\mrm{R} \Ga_{\catt{T}}(M)
\ar[d]^{ \si^{\mrm{R}}_M }
\\
A \ot^{\mrm{L}}_{A} M
\ar[r]^{ \opn{lu} }
&
M
} \]
in $\cat{D}(A \ot B^{\mrm{op}})$ is commutative. 
\end{thm}
 
\begin{proof}
We begin by constructing the morphism of triangulated functors $\ga$.
For each complex $M \in \cat{C}(A \ot B^{\mrm{op}})$
we choose a K-injective resolution 
$\ze_M : M \to I_M$,
and a K-flat resolution 
$\th_M : Q_M \to M$,
both in $\cat{C}(A \ot B^{\mrm{op}})$.
Note that $I_M$ is $\catt{T}$-flasque over $A$, and $Q_M$ is K-flat over 
$A$. We use these choices for presentations of the right derived functor 
\[ (\mrm{R} \Ga_{\catt{T}}, \xi^{\mrm{R}}) : \cat{D}(A \ot B^{\mrm{op}}) \to 
\cat{D}(A \ot B^{\mrm{op}}) \] 
and the left derived bifunctor 
\[ \bigl( (- \ot_{A}^{\mrm{L}} -), \xi^{\mrm{L}} \bigr) : 
\cat{D}(A^{\mrm{en}}) \times \cat{D}(A \ot B) \to \cat{D}(A \ot B^{\mrm{op}}) . 
\]
Let us also choose a K-injective resolution 
$\de : A \to J$ in $\cat{C}(A^{\mrm{en}})$. 
With these choices we have the following presentations:
$P = \Ga_{\catt{T}}(J)$,
$\mrm{R} \Ga_{\catt{T}}(M) = \Ga_{\catt{T}}(I_M)$
and
$P \ot_A^{\mrm{L}} M = \Ga_{\catt{T}}(J) \ot_{A} Q_M$.

Let $N \in \cat{C}(A \ot B^{\mrm{op}})$. Given homogeneous elements 
$x \in \Ga_{\catt{T}}(J)$ and $n \in N$, the tensor 
$x \ot n$ belongs to $\Ga_{\catt{T}}(J \ot_{A} N)$;
cf.\ formula (\ref{eqn:1050}). 
In this way we obtain a homomorphism 
\begin{equation} \label{eqn:1370}
\til{\ga}_N : \Ga_{\catt{T}}(J) \ot_{A} N \to 
\Ga_{\catt{T}}(J \ot_{A} N) 
\end{equation}
in $\cat{C}(A \ot B^{\mrm{op}})$, 
which is functorial in $N$. 

Consider a complex $M \in \cat{C}(A \ot B^{\mrm{op}})$. We have the solid 
diagram 
\begin{equation} \label{eqn:1200}
\UseTips \xymatrix @C=6ex @R=6ex {
A \ot_{A} Q_M
\ar[r]^(0.6){\opn{lu}}
\ar[d]_{ \de \, \ot_A \, \opn{id} }
&
Q_M
\ar[r]^{\th_{M}}
&
M
\ar[d]^{ \ze_M }
\\
J \ot_{A} Q_M 
\ar@{-->}[rr]^(0.6){\chi_M}
&
&
I_M
} 
\end{equation}
in $\cat{C}(A \ot B^{\mrm{op}})$.
The homomorphisms $\de \, \ot_A \, \opn{id}$, $\th_{M} \circ \opn{lu}$  and
$\ze_M$ in  $\cat{C}(A \ot B^{\mrm{op}})$ are quasi-isomorphisms.
Because $I_M$ is K-injective in $\cat{C}(A \ot B^{\mrm{op}})$, 
there is a quasi-isomorphism
\begin{equation} \label{eqn:1361}
\chi_M : J \ot_{A} Q_M \to I_M 
\end{equation}
that makes this diagram commutative up to homotopy. 

We now form the following diagram 
\begin{equation} \label{eqn:1355}
 \UseTips \xymatrix @C=10ex @R=8ex {
\Ga_{\catt{T}}(J) \ot_{A} Q_M
\ar[r]^{\til{\ga}_{Q_M} }
\ar[d]_{ \si_{J} \, \ot_A \, \opn{id} }
&
\Ga_{\catt{T}}(J \ot_{A} Q_M)
\ar[r]^(0.55){ \Ga_{\catt{T}}(\chi_M) } 
\ar[d]_{ \si_{J \ot_{A} Q_M} }
&
\Ga_{\catt{T}}(I_M)
\ar[d]_{ \si_{I_M} }
\\
J \ot_{A} Q_M
\ar[r]^{ \opn{id} }
&
J \ot_{A} Q_M
\ar[r]^(0.55){ \chi_M }
&
I_M
\\
A \ot_{A} Q_M 
\ar[u]^{ \de \, \ot_A \, \opn{id} }
\ar[r]^{ \opn{id} }
&
A \ot_{A} Q_M 
\ar[u]^{ \de \, \ot_A \, \opn{id} }
\ar[r]^(0.55){\th_{M} \circ \opn{lu}}
&
M
\ar[u]^{ \ze_M } 
} 
\end{equation}
in $\cat{C}(A \ot B^{\mrm{op}})$. Here $\til{\ga}_{Q_M}$ is the homomorphism 
from (\ref{eqn:1370}) with $N := Q_M$.
The diagram (\ref{eqn:1355}) is 
commutative up to homotopy. (Actually all small squares, except the bottom 
right one, are commutative in the strict sense.)
The vertical arrows $\de \, \ot_A \opn{id}$ and $\ze_M$ are 
quasi-isomorphisms. Passing to $\cat{D}(A \ot B^{\mrm{op}})$
we get a commutative diagram, with vertical isomorphisms between the second and 
third rows. The diagram with the four extreme objects only is the one we are 
looking for. By construction it is a commutative diagram
in $\cat{D}(A \ot B^{\mrm{op}})$, and it is functorial 
in $M$. The morphism 
\[ \ga_M : P \ot_A^{\mrm{L}} M \to \mrm{R} \Ga_{\catt{T}}(M) \] 
is represented by $\Ga_{\catt{T}}(\chi_M) \circ \til{\ga}_{Q_M}$. 

It remains to prove that $\ga_M$
is an isomorphism for every $M \in \cat{D}(A \ot B^{\mrm{op}})$.
Because the functor $\opn{Rest}_A$ is conservative, it suffices to prove that 
the morphism 
\[ \opn{Rest}_A(\ga_M) : \opn{Rest}_A(P \ot_A^{\mrm{L}} M) \to 
\opn{Rest}_A(\mrm{R} \Ga_{\catt{T}}(M)) \]
in $\cat{D}(A)$ is an isomorphism. Going over all the details of the 
construction above, and noting that 
$\ze_M : M \to I_M$ and $\th_M : Q_M \to M$ 
are K-flat and K-injective resolutions, respectively, also in 
$\cat{C}(A)$, we might as well forget about the ring $B$. 

Another way to justify the elimination of $B$ is this: we want to prove 
that the homomorphism 
\[ \Ga_{\catt{T}}(\chi_M) \circ \til{\ga}_{Q_M} : 
\Ga_{\catt{T}}(J) \ot_{A} Q_M \to \Ga_{\catt{T}}(I) \]
in $\cat{C}(A \ot B^{\mrm{op}})$ is a quasi-isomorphism. For that we can forget 
about the $B$-module structure. 

So now we are in the case $B = \K$, $A \ot B^{\mrm{op}} = A$ and (in longhand)
$\catt{T} \ot B^{\mrm{op}} = \catt{T}$, and we want to 
prove that $\ga_M$ is an isomorphism for every $M \in \cat{D}(A)$.
By Theorem \ref{thm:1080} the functor $\mrm{R} \Ga_{\catt{T}}$
on $\cat{D}(A)$ is quasi-compact. The functor $P \ot_{A}^{\mrm{L}} (-)$ is also
quasi-compact.  This means that we can use \cite[Lemma 4.1]{PSY2}, and it tells 
us that it suffices to prove that $\ga_M$ is an isomorphism for $M = A$. 

Let us examine the morphism $\ga_A$, i.e.\ $\ga_M$ for $M = A$. 
We can choose the K-flat resolution
$\th_A : Q_A \to A$ in $\cat{C}(A)$ to be the identity of $A$. 
Also, we can choose the K-injective resolution 
$\ze_A : A \to I_A$ in $\cat{C}(A)$ to be the restriction of 
$\de : A \to J$. Then the homomorphism 
$\chi_A : J \ot_A Q_A \to I_A$
in diagram (\ref{eqn:1200}) can be chosen to be 
$\chi_A = \opn{id} \ot_A \opn{id}$. 
We get a commutative diagram 
\[ \UseTips \xymatrix @C=10ex @R=7ex {
\Ga_{\catt{T}}(J) \ot_{A} Q_A
\ar[r]^{ \til{\ga}_{Q_A} }
\ar[d]_{ \opn{id} \ot_A \opn{id} }
&
\Ga_{\catt{T}}(J \ot_{A} Q_A)
\ar[r]^(0.6){ \Ga_{\catt{T}}(\chi_A) } 
\ar[d]_{  \Ga_{\catt{T}}(\opn{id} \ot_A \opn{id}) }
&
\Ga_{\catt{T}}(I_A)
\ar[d]_{\opn{id}}
\\
\Ga_{\catt{T}}(J) \ot_{A} A
\ar[r]^{\til{\ga}_{A}}
&
\Ga_{\catt{T}}(J \ot_{A} A)
\ar[r]^(0.6){\Ga_{\catt{T}}(\opn{ru})} 
&
\Ga_{\catt{T}}(J)
} \]
in $\cat{C}(A)$. 
The horizontal arrows in the second row, and the vertical arrows, are all
bijective. We conclude that 
$\Ga_{\catt{T}}(\chi_A) \circ \til{\ga}_{Q_A}$ is bijective. Hence 
$\ga_A$ is an isomorphism in $\cat{D}(A)$.
\end{proof}

\begin{rem} \label{rem:1460}
Theorem \ref{thm:1200} resembles \cite[Lemma 3.4]{WZ}. There the authors  
only considered the case where $A$ is a complete semilocal noetherian ring
(central over a base field $\K$), and $\catt{T}$ is torsion at the Jacobson 
radical of $A$.
\end{rem}

Again we remind that Conventions \ref{conv:1425} and \ref{conv:1445} are in 
place. The copointed object induced by a torsion class was introduced in
Definition \tup{\ref{dfn:1200}}. 

\begin{thm} \label{thm:1370}
Let $A$ be a flat ring, and let $\catt{T}$ be a quasi-compact, 
finite dimensional, weakly stable torsion class in $\cat{M}(A)$. 
Let $(P, \rho)$ be the copointed object in the monoidal category 
$\cat{D}(A^{\mrm{en}})$ that is induced by $\catt{T}$.
Then $(P, \rho)$ is an idempotent copointed object. 
\end{thm}

\begin{proof}
We shall start by proving that 
\begin{equation}  \label{eqn:1371}
\opn{lu} \circ \, (\rho \, \ot_A^{\mrm{L}} \, \opn{id}) : 
P \ot_{A}^{\mrm{L}} P \to P 
\end{equation}
is an isomorphism in $\cat{D}(A^{\mrm{en}})$. 
Because the forgetful functor 
$\opn{Rest}_A : \cat{D}(A^{\mrm{en}}) \to \cat{D}(A)$
is conservative, it is enough if we prove that 
$\opn{Rest}_A(\opn{lu} \circ (\rho \, \ot_A^{\mrm{L}} \, \opn{id}))$
is an isomorphism. Let us introduce the temporary notation 
$P' := \opn{Rest}_A(P) \in \cat{D}(A)$. 
With this notation, what we have to show is that 
\begin{equation}  \label{eqn:1372}
\opn{lu} \circ \, (\rho \, \ot_A^{\mrm{L}} \, \opn{id}) : 
P \ot_{A}^{\mrm{L}} P' \to P' 
\end{equation}
is an isomorphism in $\cat{D}(A)$.

Consider Theorem \ref{thm:1200} with $B = \K$ and 
$M = P' \in \cat{D}(A)$. 
There is a commutative diagram 
\[ \UseTips \xymatrix @C=8ex @R=6ex {
P \ot_{A}^{\mrm{L}} P'
\ar[r]^{\ga_{P'}}
\ar[d]_{ \rho \, \ot_A^{\mrm{L}} \, \opn{id} }
&
\mrm{R} \Ga_{\catt{T}}(P')
\ar[d]^{ \si^{\mrm{R}}_{P'} }
\\
A \ot^{\mrm{L}}_{A} P'
\ar[r]^{ \opn{lu} }
&
P'
} \]
in $\cat{D}(A)$, and the horizontal arrows are isomorphisms. 
It suffices to prove that 
$\si^{\mrm{R}}_{P'} : \mrm{R} \Ga_{\catt{T}}(P') \to P'$
is an isomorphism in $\cat{D}(A)$. 
But by Lemma \ref{lem:1095} there is an isomorphism 
$P' \cong \mrm{R} \Ga_{\catt{T}}(A')$,
where 
$A' := \opn{Rest}_A(A) \in \cat{D}(A)$. 
So what we need to prove is that 
\[ \si^{\mrm{R}}_{\mrm{R} \Ga_{\catt{T}}(A')} : 
\mrm{R} \Ga_{\catt{T}}(\mrm{R} \Ga_{\catt{T}}(A')) \to \mrm{R} 
\Ga_{\catt{T}}(A') 
\]
is an isomorphism in $\cat{D}(A)$. This is true because the copointed 
triangulated functor 
$(\mrm{R} \Ga_{\catt{T}}, \si^{\mrm{R}})$ on $\cat{D}(A)$
is idempotent; see Theorem \ref{thm:1080}. 

Now we are going to prove that 
\begin{equation} \label{eqn:1366}
\opn{ru} \circ \, (\opn{id}  \ot_A^{\mrm{L}} \, \rho)  : 
P \ot_{A}^{\mrm{L}} P \to P 
\end{equation}
is an isomorphism in $\cat{D}(A^{\mrm{en}})$. 

We have this diagram in $\cat{D}(A^{\mrm{en}})$, which is commutative up to a 
canonical isomorphism: 
\begin{equation} \label{eqn:1367}
\UseTips \xymatrix @C=12ex @R=6ex {
P \ot_{A}^{\mrm{L}} P
\ar[d]_{ \opn{id}  \ot_A^{\mrm{L}} \, \rho }
&
P \ot_{A}^{\mrm{L}} P \ot_{A}^{\mrm{L}} A
\ar[l]_{ \opn{id}  \ot_A^{\mrm{L}} \opn{ru} }
\ar[d]^{ \opn{id}  \ot_A^{\mrm{L}} \, \rho \, \ot_A^{\mrm{L}} \opn{id} }
\\
P \ot_{A}^{\mrm{L}} A
&
P \ot_{A}^{\mrm{L}} A \ot^{\mrm{L}}_{A} A
\ar[l]_{ \opn{id}  \ot_A^{\mrm{L}} \opn{ru} }
} 
\end{equation}
Using Theorem \ref{thm:1200} with $B = A$ and 
$M = A \in \cat{D}(A^{\mrm{en}})$, 
we have this commutative diagram in $\cat{D}(A^{\mrm{en}})$~:
\begin{equation} \label{eqn:1373}
\UseTips \xymatrix @C=8ex @R=6ex {
P \ot_{A}^{\mrm{L}} A
\ar[r]^{\ga_{A}}
\ar[d]_{ \rho \, \ot_A^{\mrm{L}} \, \opn{id} }
&
\mrm{R} \Ga_{\catt{T}}(A)
\ar[d]^{ \si^{\mrm{R}}_{A} }
\\
A \ot^{\mrm{L}}_{A} A
\ar[r]^{ \opn{lu} }
&
A
} 
\end{equation}
Applying the functor $P \ot_A^{\mrm{L}} (-)$ to diagram (\ref{eqn:1373}), we 
obtain this commutative diagram 
\begin{equation} \label{eqn:1368}
\UseTips \xymatrix @C=10ex @R=6ex {
P \ot_{A}^{\mrm{L}} P \ot_{A}^{\mrm{L}} A
\ar[d]_{ \opn{id}  \ot_A^{\mrm{L}} \, \rho \, \ot_A^{\mrm{L}} \, \opn{id} }
\ar[r]^{ \opn{id} \ot_{A}^{\mrm{L}} \, \ga_A }
&
P \ot_{A}^{\mrm{L}} \mrm{R} \Ga_{\catt{T}}(A)
\ar[d]_{ \opn{id}  \ot_A^{\mrm{L}} \, \si^{\mrm{R}}_A }
\\
P \ot_{A}^{\mrm{L}} A \ot^{\mrm{L}}_{A} A
\ar[r]^{ \opn{id} \ot_{A}^{\mrm{L}} \opn{lu} }
&
P \ot_{A}^{\mrm{L}} A 
}
\end{equation}
in $\cat{D}(A^{\mrm{en}})$.
The last move is using the fact that $\ga$ is an isomorphism of functors; this 
yields the next commutative diagram 
\begin{equation} \label{eqn:1369}
\UseTips \xymatrix @C=14ex @R=6ex {
P \ot_{A}^{\mrm{L}} \mrm{R} \Ga_{\catt{T}}(A)
\ar[d]_{ \opn{id}  \ot_A^{\mrm{L}} \, \si^{\mrm{R}}_A }
\ar[r]^{\ga_{\mrm{R} \Ga_{\catt{T}}(A)}}
&
\mrm{R} \Ga_{\catt{T}}
(\mrm{R} \Ga_{\catt{T}}(A))
\ar[d]_{\mrm{R} \Ga_{\catt{T}}(\si^{\mrm{R}}_A)}
\\
P \ot_{A}^{\mrm{L}} A 
\ar[r]^{\ga_A}
&
\mrm{R} \Ga_{\catt{T}}(A)
}
\end{equation}
in $\cat{D}(A^{\mrm{en}})$. 
All horizontal arrows in diagrams (\ref{eqn:1367}), (\ref{eqn:1373}), 
(\ref{eqn:1368}) and (\ref{eqn:1369}) are isomorphisms. 
Since $\opn{ru}$ is an isomorphism, to prove that (\ref{eqn:1366}) is an 
isomorphism, it is enough to prove that 
the morphism $\opn{id}  \ot_A^{\mrm{L}} \, \rho$
in (\ref{eqn:1367}) is an isomorphism. 
Passing horizontally from diagram (\ref{eqn:1367}) to diagram (\ref{eqn:1369}),
we see that it is enough to prove that the morphism (written in longhand) 
\begin{equation} \label{eqn:1374}
\mrm{R} \Ga_{\catt{T} \ot A^{\mrm{op}}}(\si^{\mrm{R}}_A) : 
\mrm{R} \Ga_{\catt{T} \ot A^{\mrm{op}}}
(\mrm{R} \Ga_{\catt{T} \ot A^{\mrm{op}}}(A)) \to
\mrm{R} \Ga_{\catt{T} \ot A^{\mrm{op}}}(A)
\end{equation}
is an isomorphism in $\cat{D}(A^{\mrm{en}})$.
Because the functor $\opn{Rest}_A$ is conservative, and by Lemma \ref{lem:1095},
it suffices to prove that 
\[ \mrm{R} \Ga_{\catt{T}}(\si^{\mrm{R}}_{A'}) : 
\mrm{R} \Ga_{\catt{T}}(\mrm{R} \Ga_{\catt{T}}(A')) \to
\mrm{R} \Ga_{\catt{T}}(A') \]
is an isomorphism in $\cat{D}(A)$, where, as before, we write  
$A' := \opn{Rest}_A(A) \in \cat{D}(A)$.
This is true because the copointed triangulated functor 
$(\mrm{R} \Ga_{\catt{T}}, \si^{\mrm{R}})$ on $\cat{D}(A)$
is idempotent; see Theorem \ref{thm:1080}. 
\end{proof}

\begin{rem} \label{rem:1422}
As explained in Remark \ref{rem:1420}, the flatness assumption can be 
circumvented using K-flat DG ring resolutions. 
However, the technicalities involved in proving the nonflat versions of 
Theorems \ref{thm:1370} and \ref{thm:1200} turned out to be quite formidable.
These nonflat generalizations will appear in a future paper.
\end{rem}

%\cleardoublepage
\section{Noncommutative MGM Equivalence} 
\label{sec:NC-MGM}

In this section we prove Theorem \ref{thm:1400}, that is an expanded version of 
Theorem \ref{thm:1378} in the Introduction. Conventions \ref{conv:1425} and 
\ref{conv:1445} are in place; in particular, $\K$ is a commutative base ring, 
and $A$ and $B$ are flat central $\K$-rings.
The enveloping ring of $A$ is 
$A^{\mrm{en}} = A \ot A^{\mrm{op}}$. 

Let us briefly recall torsion in bimodule categories, as developed in 
Section \ref{sec:tors-to-obj}. 
Suppose we are given a torsion class $\catt{T} \sub \cat{M}(A)$.
It extends to a bimodule torsion class 
\begin{equation} \label{eqn:2040}
\catt{T} \ot B^{\mrm{op}} \sub \cat{M}(A \ot B^{\mrm{op}}) .
\end{equation}
By definition, a module $M \in \cat{M}(A \ot B^{\mrm{op}})$
is $(\catt{T} \ot B^{\mrm{op}})$-torsion if it is $\catt{T}$-torsion as a left 
$A$-module (i.e.\ after forgetting the right $B$-module structure). 
Because the torsion and derived torsion functors respect the restriction 
functors $\opn{Rest}_{A}$ (namely they do not notice the ring $B$, see Lemma 
\ref{lem:1095}), we often write $\catt{T}$ instead of 
$\catt{T} \ot B^{\mrm{op}}$.

There is a copointed triangulated functor
$(\mrm{R} \Ga_{\catt{T}}, \si^{\mrm{R}})$
on $\cat{D}(A \ot B^{\mrm{op}})$. See Lemma \ref{lem:1063}. 
The category $\cat{D}(A \ot B^{\mrm{op}})_{\catt{T} \tup{-tor}}$
of derived $\catt{T}$-torsion complexes was introduced in Definition 
\ref{dfn:2025}. Since the setting there was a bit different, let us 
recall the definition in the current notation: a complex 
$M \in \cat{D}(A\ot B^{\mrm{op}})$ is called a 
derived $\catt{T}$-torsion complex if the morphism 
$\si^{\mrm{R}}_{M} : \mrm{R} \Ga_{\catt{T}}(M) \to M$
in $\cat{D}(A\ot B^{\mrm{op}})$ is an isomorphism.

For $B = A$ we get a copointed triangulated functor 
$(\mrm{R} \Ga_{\catt{T}}, \si^{\mrm{R}})$ on the category 
$\cat{D}(A^{\mrm{en}})$. It induces a copointed object $(P, \rho)$ in the 
monoidal category $\cat{D}(A^{\mrm{en}})$.
See Lemma \ref{lem:1082} and Definition \ref{dfn:1200}.

Letting 
$F_{\catt{T}} := P \ot_A^{\mrm{L}} (-)$,
Theorem \ref{thm:1200} tells us that there is an isomorphism of copointed
triangulated functors
\begin{equation} \label{eqn:2043}
(F_{\catt{T}}, \si) \cong (\mrm{R} \Ga_{\catt{T}}, \si^{\mrm{R}})
\end{equation}
on $\cat{D}(A\ot B^{\mrm{op}})$. 

\begin{dfn} \label{dfn:2005}
Let $A$ and $B$ be flat central $\K$-rings, let $\catt{T}$ be a torsion 
class in $\cat{M}(A)$, and let 
$(P, \rho)$ be the induced copointed object in the monoidal category 
$\cat{D}(A^{\mrm{en}})$.
\begin{enumerate}
\item Define 
\[ G_{\catt{T}} : \cat{D}(A\ot B^{\mrm{op}}) \to \cat{D}(A\ot B^{\mrm{op}}) \]
to be the triangulated functor
\[ G_{\catt{T}} := \opn{RHom}_A(P, -) . \]
We call it the {\em abstract $\catt{T}$-completion functor}. 

\item Let $\tau : \opn{Id} \to G_{\catt{T}}$
be the morphism of triangulated functors that's 
induced by $(P, \rho)$, as in Definition \ref{dfn:1075}. 
Thus the pair $(G_{\catt{T}}, \tau)$ is a pointed triangulated 
functor on $\cat{D}(A\ot B^{\mrm{op}})$. 

\item A complex $M \in \cat{D}(A\ot B^{\mrm{op}})$ is called 
a {\em derived $\catt{T}$-complete complex} if the morphism
$\tau_{M} : M \to G_{\catt{T}}(M)$ is an isomorphism. 
The full subcategory of $\cat{D}(A \ot B^{\mrm{op}})$ on the derived 
$\catt{T}$-complete complexes is denoted by \lb 
$\cat{D}(A \ot B^{\mrm{op}})_{\catt{T} \tup{-com}}$. 
\end{enumerate}
\end{dfn}

Because $(G_{\catt{T}}, \tau)$ and 
$(\mrm{R} \Ga_{\catt{T}}, \si^{\mrm{R}})$ are  triangulated (co)pointed 
functors (see Definitions \ref{dfn:1003} and \ref{dfn:1210}), the categories 
$\cat{D}(A \ot B^{\mrm{op}})_{\catt{T} \tup{-com}}$
and
$\cat{D}(A \ot B^{\mrm{op}})_{\catt{T} \tup{-tor}}$
are full triangulated categories of $\cat{D}(A \ot B^{\mrm{op}})$. 

In Definition \ref{dfn:1037} we said what it means for the torsion class 
$\catt{T} \sub \cat{M}(A)$
to be quasi-compact, weakly stable and finite dimensional. 
The next theorem is slightly stronger than Theorem \ref{thm:1378} in the 
Introduction, where we had $B = \K$ and thus $A \ot B^{\mrm{op}} = A$. 

\begin{thm}[Noncommutative MGM Equivalence] \label{thm:1400}
Let $A$ and $B$ be flat central $\K$-rings, and let $\catt{T}$ be a 
quasi-compact, weakly stable, finite dimensional torsion class in $\cat{M}(A)$. 
Then the following hold\tup{:}
\begin{enumerate}
\item The functor 
\[ G_{\catt{T}} : \cat{D}(A \ot B^{\mrm{op}}) \to 
\cat{D}(A \ot B^{\mrm{op}}) \]
is right adjoint to 
$\mrm{R} \Ga_{\catt{T}}$. 

\item The copointed triangulated functor 
$(\mrm{R} \Ga_{\catt{T}}, \si^{\mrm{R}})$
on $\cat{D}(A \ot B^{\mrm{op}})$, and the pointed triangulated functor 
$(G_{\catt{T}}, \tau)$ on $\cat{D}(A \ot B^{\mrm{op}})$, are both idempotent.

\item The subcategories 
$\cat{D}(A \ot B^{\mrm{op}})_{\catt{T} \tup{-tor}}$
and 
$\cat{D}(A \ot B^{\mrm{op}})_{\catt{T} \tup{-com}}$ 
are the essential images of the functors $\mrm{R} \Ga_{\catt{T}}$ and 
$G_{\catt{T}}$, respectively. 

\item The functor 
\[ \mrm{R} \Ga_{\catt{T}} : 
\cat{D}(A\ot B^{\mrm{op}})_{\catt{T} \tup{-com}} \to 
\cat{D}(A\ot B^{\mrm{op}})_{\catt{T} \tup{-tor}} \]
is an equivalence of triangulated categories, with quasi-inverse 
$G_{\catt{T}}$. 
\end{enumerate}
\end{thm}

\begin{proof}
There is the isomorphism of copointed triangulated functors (\ref{eqn:2043}),
coming from Theorem \ref{thm:1200}.
Therefore there is equality 
\[ \cat{D}(A\ot B^{\mrm{op}})_{F_{\catt{T}}} =
\cat{D}(A\ot B^{\mrm{op}})_{\catt{T} \tup{-tor}} . \]
By definition there is equality 
\[ \cat{D}(A\ot B^{\mrm{op}})_{G_{\catt{T}}} = 
\cat{D}(A\ot B^{\mrm{op}})_{\catt{T} \tup{-com}} .  \]
We see that items (1)-(4) here are special cases of items (1)-(4), 
respectively, in Theorem \ref{thm:1075}.
\end{proof}

Here are two examples of this result. 

\begin{exa} \label{exa:1455}
Consider the rings $B = A = \K := \Z$, the multiplicatively closed set 
$S := \Z - \{ 0 \}$, and the torsion class $\catt{T} := \catt{T}_{S}$
in $\cat{M}(\Z)$, as in Example \ref{exa:1416}. 
So here 
\[ \cat{M}(A \ot B^{\mrm{op}}) = \cat{M}(\Z) = \cat{Ab} , \]
and $\Ga_{\catt{T}}(M)$ is the usual torsion subgroup of an abelian group $M$. 

Even though the ring $\Z$ is commutative and noetherian, $\catt{T}$
is not a torsion class associated to an ideal, so it is not covered by the 
results of \cite{PSY1} (cf.\ Section \ref{sec:comm-rings}).

Still, it can be shown that $\catt{T}$ satisfies the conditions of 
Theorem \ref{thm:1400}. Furthermore, in this case the functor $G_{\catt{T}}$ is 
the left derived functor $\mrm{L} \La$, 
where $\La : \cat{M}(\Z) \to \cat{M}(\Z)$ is the ``profinite completion'' 
functor 
\[ \La(M) := \lim_{\leftarrow k} \, (M / k \cd M) . \]
Here $k$ runs through the set of positive integers, with its partial order by 
divisibility. 
\end{exa}

\begin{exa} \label{exa:1470}
Take any ring $A$, and let $B := \K$, so that 
$\cat{M}(A \ot B^{\mrm{op}}) = \cat{M}(A)$. 
Suppose $a$ is a regular normalizing element of $A$. 
Recall that this means that $A \cd a = a \cd A$, and $a$ is a non-zero-divisor.
(It follows that there is an automorphism $\ga$ of the ring $A$ such that 
$a \cd b = \ga(b) \cd a$ for all $b \in A$.)  
Let $\a \sub A$ be the ideal generated  by $a$. 

As already mentioned in Example \ref{exa:1417}, the torsion class 
$\catt{T}_{\a} \sub \cat{M}(A)$ is weakly stable. 
It can be shown (see \cite[Lemma 6.4]{Vy}) that 
$\catt{T}_{\a}$ is also quasi-compact and finite dimensional. Therefore Theorem 
\ref{thm:1400} applies. Moreover, in this case the functor $G_{\catt{T}}$ is 
nothing but $\mrm{L} \La_{\a}$, the derived $\a$-adic completion functor. 
\end{exa}

We end this section with two questions related to Theorem \ref{thm:1400}.

\begin{que} \label{que:1380}
Under which conditions is there an additive functor 
$\La : \cat{M}(A) \to \cat{M}(A)$, such that
the right adjoint to $\mrm{R} \Ga_{\catt{T}}$ is 
$G_{\catt{T}} = \mrm{L} \La$~?
Compare to the commutative situation in Theorem \ref{thm:1375}, where 
$\La = \La_{\a}$ is the $\a$-adic completion functor. 
On the other hand, there are counterexamples where this fails (see \cite[Example 6.2]{Vy}). 
\end{que}

\begin{que} \label{que:1660}
By assumption the functor $\mrm{R} \Ga_{\catt{T}}$ has finite cohomological 
dimension. In the commutative case (see Theorem \ref{thm:1375}), where 
$\mrm{R} \Ga_{\catt{T}} = \mrm{R} \Ga_{\a}$, 
its right 
adjoint $G_{\catt{T}} = \mrm{L} \La_{\a}$
also has finite cohomological dimension. Is this true in the 
noncommutative case?
\end{que}

%\cleardoublepage
\section{Symmetric Derived Torsion} 
\label{sec:symm-der-tors}

% \cmnt{new subsec, broken off from subsec 8, very much improved }

In this final section of the paper we prove Theorem \ref{thm:1401}, 
that is Theorem \ref{thm:1379} in the Introduction. Conventions \ref{conv:1425} 
and \ref{conv:1445} are in place. In particular, $\K$ is a nonzero commutative 
base ring.

Let $A$ and $B$ be flat central $\K$-rings. 
We know that the monoidal category $\cat{D}(A^{\mrm{en}})$ has a left monoidal 
action on the category $\cat{D}(A \ot B^{\mrm{op}})$. 
Similarly, the  monoidal category $\cat{D}(B^{\mrm{en}})$ has a right monoidal 
action on $\cat{D}(A \ot B^{\mrm{op}})$. Concretely, given complexes
$P \in \cat{D}(A^{\mrm{en}})$, $Q \in \cat{D}(B^{\mrm{en}})$
and $M \in \cat{D}(A \ot B^{\mrm{op}})$, the monoidal actions are 
\begin{equation} \label{eqn:2016}
P \ot_A^{\mrm{L}} M, \, M \ot_B^{\mrm{L}} Q \, \in \,  
\cat{D}(A \ot B^{\mrm{op}}) . 
\end{equation}
The left and right monoidal action commute, in the sense that there is the  
canonical associativity isomorphism 
\[ (P \ot_A^{\mrm{L}} M) \ot_B^{\mrm{L}} Q \cong 
P \ot_A^{\mrm{L}} (M \ot_B^{\mrm{L}} Q) \in 
\cat{D}(A \ot B^{\mrm{op}}) . \]

Previously we only looked at a 
torsion class $\catt{T} \sub \cat{M}(A)$, that we extended to bimodule torsion 
classes
\[ \catt{T}  \ot B^{\mrm{op}} \sub \cat{M}(A \ot B^{\mrm{op}})
\quad \tup{and} \quad 
\catt{T}  \ot A^{\mrm{op}}  \sub \cat{M}(A^{\mrm{en}}) , \]
as explained in Definition \ref{dfn:1475}. 
We sometimes referred to these bimodule torsion classes succinctly by 
$\catt{T}$, and this abbreviation was justified by Lemma \ref{lem:1095}. 

Here we consider a more complicated situation: there is also a torsion class
$\catt{S}^{\mrm{op}} \sub \cat{M}(B^{\mrm{op}})$.
We extend it by the same procedure (only replacing the roles of $A$ and 
$B^{\mrm{op}}$) to bimodule torsion classes
\[ A \ot \catt{S}^{\mrm{op}} \sub \cat{M}(A \ot B^{\mrm{op}}) 
\quad \tup{and} \quad 
B \ot \catt{S}^{\mrm{op}} \sub \cat{M}(B^{\mrm{en}}) . \]
Again, we sometimes abbreviate the notation to $\catt{S}^{\mrm{op}}$. 

We shall need to enrich another part of our notation to accommodate the more 
complicated situation. For a complex $M \in \cat{D}(A \ot B^{\mrm{op}})$, 
the canonical morphisms of triangulated functors from Lemma 
\ref{lem:1082} will now be denoted as follows:
\begin{equation} \label{eqn:2011}
\si^{\mrm{R}}_{\catt{T},  M} : \mrm{R} \Ga_{\catt{T}}(M) \to M 
\end{equation}
and
\begin{equation} \label{eqn:2012}
\si^{\mrm{R}}_{\catt{S}^{\mrm{op}}, M} : 
\mrm{R} \Ga_{\catt{S}^{\mrm{op}}}(M)  \to M . 
\end{equation}

\begin{exa} \label{exa:2010}
This is a continuation of Example \ref{exa:2040}.
Suppose $\a \sub A$ and $\b^{\mrm{op}} \sub B^{\mrm{op}}$ are two-sided ideals 
that are finitely generated as left ideals.  
Note that $\b^{\mrm{op}}$ can be viewed as a two-sided ideal $\b \sub B$, and 
then it is finitely generated as a {\em right ideal}. 

In the notation of Definition \ref{dfn:1065}, there are torsion classes
$\catt{T} := \catt{T}_{\a} \sub \cat{M}(A)$
and 
$\catt{S}^{\mrm{op}} := \catt{T}_{\b^{\mrm{op}}} \sub \cat{M}(B^{\mrm{op}})$.
This places us in the situation described above:
$\catt{T}$ is a torsion class in $\cat{M}(A)$
and $\catt{S}^{\mrm{op}}$ is a torsion class in $\cat{M}(B^{\mrm{op}})$.
These extend to bimodule torsion classes 
\[ \catt{T}  \ot B^{\mrm{op}} , \, A \ot \catt{S}^{\mrm{op}} 
\sub \cat{M}(A \ot B^{\mrm{op}}) . \]

There is another way to view these bimodule torsion classes.
Consider the two-sided ideals 
$\a \ot B^{\mrm{op}}$ and $A \ot \b^{\mrm{op}}$ in the ring 
$A \ot B^{\mrm{op}}$. Then, as torsion classes in $\cat{M}(A \ot B^{\mrm{op}})$,
and with the notation of Definition \ref{dfn:1065}, we have
\[ \catt{T}  \ot B^{\mrm{op}} = \catt{T}_{\a \ot B^{\mrm{op}}} 
\quad \tup{and} \quad 
A \ot \catt{S}^{\mrm{op}} = \catt{T}_{A \ot \b^{\mrm{op}}} . \]
\end{exa} 

The notation used in Definition \ref{dfn:2025} for cohomologically 
$\catt{T}$-torsion complexes is not sufficient now, since 
we must accommodate $\catt{S}^{\mrm{op}}$-torsion. Hence the next 
definition.

\begin{dfn} \label{dfn:20016}
Let $\catt{T} \sub \cat{M}(A)$ and 
$\catt{S}^{\mrm{op}} \sub \cat{M}(B^{\mrm{op}})$
be torsion classes. We denote by 
\[ \cat{D}_{(\catt{T}, \catt{S}^{\mrm{op}})}(A \ot B^{\mrm{op}}) \]
full subcategory of $\cat{D}(A \ot B^{\mrm{op}})$ 
the  on the complexes $M$ such that 
\[ \opn{H}^i(M) \in (\catt{T} \ot B^{\mrm{op}}) \cap 
(A \ot \catt{S}^{\mrm{op}}) \sub \cat{M}(A \ot B^{\mrm{op}}) \]
for all $i$. 
\end{dfn}

Clearly $\cat{D}_{(\catt{T}, \catt{S}^{\mrm{op}})}(A \ot B^{\mrm{op}})$
is a full triangulated subcategory of $\cat{D}(A \ot B^{\mrm{op}})$. 

\begin{dfn} \label{dfn:2000}
Let $A$ and $B$ be flat central $\K$-rings, let $\catt{T} \sub 
\cat{M}(A)$ and $\catt{S}^{\mrm{op}} \sub \cat{M}(B^{\mrm{op}})$
be torsion classes, and let 
$M \in \cat{D}(A \ot B^{\mrm{op}})$.
\begin{enumerate}
\item The complex $M$ said to have {\em weakly symmetric derived 
$\catt{T}$-$\catt{S}^{\mrm{op}}$-torsion} if 
\[ \mrm{R} \Ga_{\catt{T}}(M) , \,  
\mrm{R} \Ga_{\catt{S}^{\mrm{op}}}(M) \, \in \,  
\cat{D}_{(\catt{T}, \catt{S}^{\mrm{op}})}(A \ot B^{\mrm{op}}) . \]
% 
% for every integer $i$ the 
% bimodules $\opn{H}^i(\mrm{R} \Ga_{\catt{T}}(M))$ 
% and $\opn{H}^i(\mrm{R} \Ga_{\catt{S}^{\mrm{op}}}(M))$
% belong to
% \[ (\catt{T} \ot B^{\mrm{op}}) \cap (A \ot \catt{S}^{\mrm{op}}) 
% \sub \cat{M}(A \ot B^{\mrm{op}}) . \]
 
\item The complex $M$ is said to have {\em symmetric derived 
$\catt{T}$-$\catt{S}^{\mrm{op}}$-torsion} if there is an isomorphism 
\[ \ep_{M} : \mrm{R} \Ga_{\catt{T}}(M) \iso 
\mrm{R} \Ga_{\catt{S}^{\mrm{op}}}(M) \]
in $\cat{D}(A  \ot B^{\mrm{op}})$, 
such that the diagram
\[ \UseTips \xymatrix @C=8ex @R=8ex {
\mrm{R} \Ga_{\catt{T}}(M)
\ar[r]^{\ep_{M}}_{\cong}
\ar[dr]_{\si^{\mrm{R}}_{\catt{T},  M}}
&
\mrm{R} \Ga_{\catt{S}^{\mrm{op}}}(M)
\ar[d]^{\si^{\mrm{R}}_{\catt{S}^{\mrm{op}}, M}}
\\
&
M
} \]
in $\cat{D}(A \ot B^{\mrm{op}})$ is commutative.
Such an isomorphism $\ep_M$ is called a {\em symmetry isomorphism}. 
\end{enumerate}
\end{dfn}

Of course if $M$ has symmetric derived torsion then it has weakly symmetric 
derived torsion. The full subcategory of $\cat{D}(A \ot B^{\mrm{op}})$ on the 
complexes with weakly symmetric derived torsion is triangulated. 

% 
% With this new notation, a complex $M$ has weakly symmetric derived torsion if 
% and only if 
% \[ \mrm{R} \Ga_{\catt{T}}(M) , \, \mrm{R} \Ga_{\catt{S}^{\mrm{op}}}(M) \in 
% \cat{D}_{(\tup{tor}, \tup{tor})}(A \ot B^{\mrm{op}}) . \]

% \cmnt{stronger thm below}

\begin{thm}[Symmetric Derived Torsion] \label{thm:1401}
Let $A$ and $B$ be flat central $\K$-rings, and let $\catt{T} \sub \cat{M}(A)$ 
and $\catt{S}^{\mrm{op}} \sub \cat{M}(B^{\mrm{op}})$
be quasi-compact, weakly stable, finite dimensional torsion classes. 
Let $M \in \cat{D}(A \ot B^{\mrm{op}})$
be a complex with weakly symmetric derived 
$\catt{T}$-$\catt{S}^{\mrm{op}}$-torsion.

Then $M$ has symmetric derived torsion. Moreover, the symmetry isomorphism 
\[ \ep_{M} : \mrm{R} \Ga_{\catt{T}}(M) \iso 
\mrm{R} \Ga_{\catt{S}^{\mrm{op}}}(M) \]
in $\cat{D}(A \ot B^{\mrm{op}})$ is unique, and it is functorial in such 
complexes $M$. 
\end{thm}

The notions of quasi-compact, weakly stable and finite dimensional torsion 
classes were introduced in Definition \ref{dfn:1037}.

\begin{proof}
Let 
$P := \mrm{R} \Ga_{\catt{T}}(A) \in \cat{D}(A^{\mrm{en}})$,
as in Definition \ref{dfn:1200}. By Theorem \ref{thm:1200} we know that there 
is an isomorphism 
\begin{equation} \label{eqn:1460}
\ga_{\catt{T}} : P \ot_A^{\mrm{L}} (-) \iso \mrm{R} \Ga_{\catt{T}}  
\end{equation}
of triangulated functors from $\cat{D}(A \ot B^{\mrm{op}})$ to itself.
And there is an idempotent copointed object
$(P, \si^{\mrm{R}}_{\catt{T},  A})$
in the monoidal category $\cat{D}(A^{\mrm{en}})$. 

The constructions of Section \ref{sec:tors-to-obj}, performed 
with the ring $B^{\mrm{op}}$ instead of $A$, give rise to a triangulated 
functor 
\[ \mrm{R} \Ga_{\catt{S}^{\mrm{op}}} : 
\cat{D}(B^{\mrm{en}}) \to \cat{D}(B^{\mrm{en}}) . \]
Applying this functor to the monoidal unit  
$B \in \cat{D}(B^{\mrm{en}})$, we obtain an object
\[ Q := \mrm{R} \Ga_{\catt{S}^{\mrm{op}}}(B) \in \cat{D}(B^{\mrm{en}}) . \]
After some possibly disorienting switches between rings and their opposites,
and between $A$ and $B$, we realize that Theorem \ref{thm:1200} implies that 
there is an isomorphism 
\begin{equation} \label{eqn:1461}
\ga_{\catt{S}^{\mrm{op}}} : (-) \ot_B^{\mrm{L}} Q \iso 
\mrm{R} \Ga_{\catt{S}^{\mrm{op}}}   
\end{equation}
as triangulated functors from $\cat{D}(A \ot B^{\mrm{op}})$ to itself.
Furthermore, there is an idempotent copointed object
$(Q, \si^{\mrm{R}}_{\catt{S}^{\mrm{op}}, A})$ 
in the monoidal category $\cat{D}(B^{\mrm{en}})$. 

Because the derived tensor product is associative (up to a canonical 
isomorphism, see Proposition \ref{prop:1425}), we deduce from formulas 
(\ref{eqn:1460}) and (\ref{eqn:1461})
that there are isomorphisms
\begin{equation} \label{eqn:1400}
\mrm{R} \Ga_{\catt{S}^{\mrm{op}}} \circ \mrm{R} \Ga_{\catt{T}} \cong
\bigl( P \ot_A^{\mrm{L}} (-) \bigr) \ot_B^{\mrm{L}} Q \cong 
P \ot_A^{\mrm{L}} \bigl( (-) \ot_B^{\mrm{L}} Q \bigr) \cong 
\mrm{R} \Ga_{\catt{T}} \circ \mrm{R} \Ga_{\catt{S}^{\mrm{op}}} 
\end{equation}
of triangulated functors from $\cat{D}(A \ot B^{\mrm{op}})$ to itself.

We are given a complex 
$M \in \cat{D}(A \ot B^{\mrm{op}})$
with weakly symmetric derived torsion. 
Consider the following diagram in $\cat{D}(A \ot B^{\mrm{op}})$.
\begin{equation} \label{eqn:3871} 
\UseTips \xymatrix @C=8ex @R=8ex {
(P \ot^{\mrm{L}}_{A} M) \ot^{\mrm{L}}_{B} Q
\ar[d]_{ \opn{ru} \circ \, 
(\opn{id} \ot^{\mrm{L}}_{A} \, \si^{\mrm{R}}_{\catt{S}^{\mrm{op}}, B}) }
\ar[r]^{\al}_{\cong}
&
P \ot^{\mrm{L}}_{A} (M \ot^{\mrm{L}}_{B} Q)
\ar[d]^{ \opn{lu} \circ \, 
(\si^{\mrm{R}}_{\catt{T}, A} \ot^{\mrm{L}}_{A} \opn{id}) }
\\
P \ot^{\mrm{L}}_{A} M 
\ar[dr]_{ \opn{lu} \circ \, 
(\si^{\mrm{R}}_{\catt{T}, A}  \ot^{\mrm{L}}_{A} \opn{id}) \quad }
&
M \ot^{\mrm{L}}_{B} Q
\ar[d]^{ \opn{ru} \circ \, 
(\opn{id} \ot^{\mrm{L}}_{A} \, \si^{\mrm{R}}_{\catt{S}^{\mrm{op}}, B}) }
\\
&
M
} 
\end{equation}
Here $\al$ is the associativity isomorphism for the derived tensor product,
see (\ref{eqn:1400}). A quick check shows that it is a commutative diagram. 

By Theorem \ref{thm:1200} we get an isomorphic commutative diagram 
in $\cat{D}(A \ot B^{\mrm{op}})$, the solid arrows only:
\begin{equation} \label{eqn:3872}
\UseTips \xymatrix @C=8ex @R=8ex {
\mrm{R} \Ga_{\catt{S}^{\mrm{op}}}(\mrm{R} \Ga_{\catt{T}}(M))
\ar[r]^{}_{\cong}
\ar[d]_{ \si^{\mrm{R}}_{\catt{S}^{\mrm{op}}, \mrm{R} \Ga_{\catt{T}}(M)} }
&
\mrm{R} \Ga_{\catt{T}}(\mrm{R} \Ga_{\catt{S}^{\mrm{op}}}(M))
\ar[d]^{ \si^{\mrm{R}}_{\catt{T}, \mrm{R} \Ga_{\catt{S}^{\mrm{op}}}(M)} }
\\
\mrm{R} \Ga_{\catt{T}}(M)
\ar[dr]_{ \si^{\mrm{R}}_{\catt{T}, M} }
\ar@{-->}[r]^{\ep_M}
&
\mrm{R} \Ga_{\catt{S}^{\mrm{op}}}(M)
\ar[d]^{ \si^{\mrm{R}}_{\catt{S}^{\mrm{op}}, M} }
\\
&
M
} 
\end{equation}
The isomorphism going from diagram (\ref{eqn:3871}) to diagram (\ref{eqn:3872}) 
is by the various compositions of the isomorphisms $\ga_{A, (-)}$ and 
$\ga_{A^{\mrm{op}}, (-)}$. Because 
\[ \mrm{R} \Ga_{\catt{T}}(M) \in 
\cat{D}_{(\catt{T}, \catt{S}^{\mrm{op}})}(A \ot B^{\mrm{op}}) , \]
Theorem \ref{thm:2031} (applied to the ring $A \ot B^{\mrm{op}}$ and the 
torsion class $\catt{S}^{\mrm{op}}$) tells us that 
\[ \mrm{R} \Ga_{\catt{T}}(M) \in 
\cat{D}(A \ot B^{\mrm{op}})_{\catt{S}^{\mrm{op}} \tup{-tor}} . \]
Now Theorem \ref{thm:1400}(2), i.e.\ the idempotence of 
$\mrm{R} \Ga_{\catt{S}^{\mrm{op}}}$,
says that 
$\si^{\mrm{R}}_{\catt{S}^{\mrm{op}}, \mrm{R} \Ga_{\catt{T}}(M)}$
is an isomorphism. Likewise, because 
\[ \mrm{R} \Ga_{\catt{S}^{\mrm{op}}}(M) \in 
\cat{D}_{(\catt{T}, \catt{S}^{\mrm{op}})}(A \ot B^{\mrm{op}}) , \]
%\cat{D}_{(\mrm{tor}, \mrm{tor})}(A^{\mrm{en}}, \mrm{gr}) , \]
this theorem says that 
$\si^{\mrm{R}}_{\catt{T}, \mrm{R} \Ga_{\catt{S}^{\mrm{op}}}(M)}$
is an isomorphism. We define $\ep_M$ to be the unique isomorphism (the dashed 
arrow) that makes diagram (\ref{eqn:3872}) commutative.

The functoriality of $\ep_M$ is a consequence of the functoriality of diagram 
(\ref{eqn:3872}). This diagram also proves that $\ep_M$ is unique. 
\end{proof}

The next example relates symmetric derived torsion with the $\chi$ condition of 
Artin-Zhang \cite{AZ}. 

\begin{exa} \label{exa:2020}
Assume $\K$ is a field, and $A$ is a noetherian central $\K$-ring. 
We consider either of the next two scenarios:
\begin{itemize}
\item $A$ is a connected graded $\K$-ring, as in Example \ref{exa:1066}, 
with augmentation ideal $\m$. 

\item $A$ is a complete semilocal ring, with Jacobson radical $\m$, and 
$A / \m$ is finite over $\K$. 
\end{itemize}
The opposite ring is $A^{\mrm{op}}$, and in it there is the ideal 
$\m^{\mrm{op}}$. 
We have these torsion classes:
$\catt{T} := \catt{T}_{\m} \sub \cat{M}(A)$
and 
$\catt{T}^{\mrm{op}} := \catt{T}_{\m^{\mrm{op}}} \sub \cat{M}(A^{\mrm{op}})$.
The respective torsion functors are 
$\Ga_{\catt{T}}  = \Ga_{\m}$ and 
$\Ga_{\catt{T}^{\mrm{op}}}  = \Ga_{\m^{\mrm{op}}}$.

The {\em left $\chi$ condition} on $A$ (slightly rephrased; cf.\ \cite{WZ}) 
says that if $M$ is a finite $A$-module, then the $A$-modules  
$\mrm{R}^q \Ga_{\m}(M)$ are all cofinite (i.e.\ artinian). 
We say that {\em $A$ satisfies the $\chi$ condition} if both $A$ and 
$A^{\mrm{op}}$ satisfy the left $\chi$ condition. 

If the functors $\mrm{R} \Ga_{\m}$ and $\mrm{R} \Ga_{\m^{\mrm{op}}}$ have finite 
cohomological dimensions, then $A$ is said to have {\em finite local 
cohomological dimension}. 

It can be shown that if $A$ satisfies the $\chi$ condition and has finite local
cohomological dimension, then every 
$M \in \cat{D}(A^{\mrm{en}})$,
whose cohomology bimodules $\mrm{H}^q(M)$
are finite modules on both sides, has weakly symmetric derived 
$\m$-$\m^{\mrm{op}}$-torsion. (For the complete semi-local case see 
\cite[Lemma 2.8]{WZ}.) Theorem \ref{thm:1401} tells us that such $M$ has 
symmetric derived torsion. 
In particular, taking $M = A$, we get a canonical isomorphism 
\[ \ep_A : \mrm{R} \Ga_{\m}(A)  \iso  \mrm{R} \Ga_{\m}(A^{\mrm{op}}) \]
in $\cat{D}(A^{\mrm{en}})$.

In both scenarios there is a bimodule 
$A^* \in \cat{M}(A^{\mrm{en}})$,
which is a torsion module and an injective module on both sides. 
In the graded scenarios $A^*$ is the graded $\K$-linear dual of $A$, and in the 
complete case $A^*$ is the continuous $\K$-linear dual of $A$.

If $A$ satisfies the $\chi$ condition and has finite local
cohomological dimension, then the complex 
\[ R_A := \opn{Hom}_A \bigl( \mrm{R} \Ga_{\m}(A), A^* \bigr) 
\in \cat{D}(A^{\mrm{en}}) \]
is a {\em balanced dualizing complex} over $A$. 
In the graded scenario this was proved by Van den Bergh \cite{VdB}, and in the 
complete scenario this was essentially proved by Wu and Zhang in \cite{WZ}, 
with finishing touches in \cite{VY}. 
\end{exa}

\begin{rem} \label{rem:1350}
We can avoid the assumption that the rings $A$ and $B$ are flat over the 
base ring $\K$. This is done by taking K-flat DG ring resolutions of 
$A$ and $B$, as in Remark \ref{rem:1420}. However, the technical complications 
of such a generalization are quite substantial. Therefore we have decided to 
restrict attention in the present paper to the flat case. The nonflat 
generalization will appear in the future paper \cite{VY}.  
\end{rem}

\begin{rem} \label{rem:1400}
As was discovered by the first author of the present paper, the proof of
\cite[Theorem 1.23]{YZ} was erroneous unless the base ring $\K$ is a field. 
(Even then the proof of \cite[Lemma 1.24]{YZ} was incorrect; but that could be 
easily fixed when $\K$ is a field.) For the purposes of the paper \cite{YZ} 
this error was negligible, because for the remainder of that paper it was 
assumed anyhow that the base ring is a field. 

Finding a correct proof of \cite[Theorem 1.23]{YZ} was one of our goals for 
some time. In Theorem \ref{thm:1401} above we have accomplished that -- almost. 
The caveat is that here we need to assume that the torsion classes $\catt{T}$ 
and $\catt{S}^{\mrm{op}}$ are finite dimensional, and this condition did not 
appear in \cite{YZ}. As for the other conditions: in \cite{YZ} the torsion 
classes were assumed to be stable, and here only weak stability is needed. The 
condition  ``locally finitely resolved'' in \cite{YZ} implies 
quasi-compactness 
here. 
\end{rem}

%\cleardoublepage

\end{document}